\documentclass[12pt, ams fonts]{amsart}

\usepackage{color}
\usepackage{stmaryrd}
\usepackage{tikz-cd}
\usepackage{mathtools}
\usepackage{amssymb,amsmath}
\usepackage{amsfonts}
\usepackage{graphicx}
\usepackage{nicematrix}

\usepackage[vcentermath,enableskew]{youngtab}

\input xy
\xyoption{all}

\font\teneufm=eufm10 \font\seveneufm=eufm7 \font\fiveeufm=eufm5
\newfam\eufmfam
\textfont\eufmfam=\teneufm \scriptfont\eufmfam=\seveneufm
\scriptscriptfont\eufmfam=\fiveeufm

\newcommand{\symdiffsmall}{\mathbin{\vcenter{\hbox{$\scriptstyle\triangle$}}}}
\newcommand{\C}{\mathbb{C}}

\newcommand{\Z}{\mathbb{Z}}

\newcommand{\A}{\mathcal A}

\newcommand{\h}{\mathfrak{h}}

\DeclareMathOperator{\Supp}{Supp}
\DeclareMathOperator{\Span}{Span} 
\DeclareMathOperator{\Hom}{Hom}

 \DeclareMathOperator{\Aut}{Aut}

\DeclareMathOperator{\diag}{diag}
\DeclareMathOperator{\Mat}{Mat}
\DeclareMathOperator{\GL}{GL}
\DeclareMathOperator{\Res}{Res}
\DeclareMathOperator{\I}{I}

\DeclareMathOperator{\End}{End}
\DeclareMathOperator{\comp}{comp}

\DeclareMathOperator{\card}{card}
\DeclareMathOperator{\lcm}{lcm}

\renewcommand{\Re}{\text{Re\,}}
\newcommand{\np}{\medskip\noindent}

\numberwithin{equation}{section}

\newtheorem{definition}{Definition}[section]

\newtheorem{example}[definition]{Example}

\theoremstyle{remark}

\newtheorem{remark}[definition]{Remark} 

\theoremstyle{plain} 

\newtheorem{theorem}[definition]{Theorem}
\newtheorem{lemma}[definition]{Lemma}
\newtheorem{corollary}[definition]{Corollary}
\newtheorem{proposition}[definition]{Proposition}

\def\Z{\mathbb Z}
\def\C{\mathbb C}

\begin{document}

\title{Generalized permutation matrices and non-weight modules over $\mathfrak{sl}(m|1)$}
\author[I. Dimitrov]{Ivan Dimitrov}
	\address{I. Dimitrov: Department of Mathematics and Statistics, Queen's University, Kingston, ON K7L 3N6, Canada}
	\email{dimitrov@queensu.ca}	
\author[K. Nguyen]{Khoa Nguyen}
	\address{K. Nguyen: Department of Mathematics and Statistics, Queen's University, Kingston, ON K7L 3N6, Canada}
	\email{k.nguyen@queensu.ca}	
\author[C. Paquette]{Charles Paquette}
\address{C. Paquette : Department of Mathematics and Computer Science, Royal Military College of Canada}
        \email{charles.paquette.math@gmail.com}
\author[D. Wehlau]{David Wehlau}
\address{D. Wehlau: Department of Mathematics and Computer Science, Royal Military College of Canada}
        \email{wehlau@rmc.ca}

\begin{abstract}
We study the category $\mathcal{M}_{\mathfrak{sl}(m|1)}(k|k)$ of $\mathcal U(\mathfrak h)\text{-free}$ $\mathcal U(\mathfrak{sl}(m|1))$-modules of rank $k$ in each parity (rank $(k|k)$), where $k\in\Z_{\geq1}$. We construct an explicit family of such modules, provide an isomorphism theorem, and establish an indecomposability criterion.
\end{abstract}
	

\maketitle

\section{Introduction}
\np
Modules over Lie algebras and superalgebras are essential objects with numerous applications in mathematics and physics. The study of these modules is typically divided into various categories. A classical focus is the category of \emph{weight modules}, whose objects decompose into direct sums of weight spaces relative to a Cartan subalgebra $\mathfrak h$; see \cite{DMP,G2,GPS,G} and references therein for the superalgebra setting. By contrast, \emph{$\mathcal U(\mathfrak h)$-free modules} form an important category of non-weight modules, which recently gained a lot of attention. Upon restriction to $\mathfrak h$, they are free $\mathcal U(\mathfrak h)$-modules of finite rank. This class was introduced independently by J.~Nilsson \cite{Nil1} and by H.~Tan and K.~Zhao \cite{TZ2}. Nilsson classified all rank-$1$ $\mathcal U(\mathfrak h)$-free modules for $\mathfrak{sl}(n{+}1)$ and $\mathfrak{sp}(2n)$ \cite{Nil1,Nil2}, while Tan and Zhao classified the rank-$1$ $\mathcal U(\mathfrak h)$-free modules for the Witt algebras $W_n^+$ and $W_n$ \cite{TZ1}. Subsequent families of finite-rank $\mathcal U(\mathfrak h)$-free modules were constructed for $\mathfrak{sl}(2)$ \cite{MP} and $\mathfrak{sl}(n{+}1)$ \cite{GN}.

\np
The categories of $\mathcal{U}(\mathfrak{h})$-free modules of rank 1 over basic Lie superalgebras were first studied by Y. Cai and K. Zhao in \cite{CZ}. In particular,  they showed that these categories are empty for all basic Lie superalgebras except for the case of  $\mathfrak{osp}(1|2n)$, and they classified all  $\mathcal{U}(\mathfrak{h})$-free modules of rank 1 over $\mathfrak{osp}(1|2n)$. Since then, further families of $\mathcal{U}(\mathfrak{h})$-free modules have been constructed and studied over other Lie superalgebras beyond the basic ones, including super-Virasoro algebras (in \cite{YYX1}), the untwisted $N = 2$ superconformal algebras (\cite{YYX2}), the twisted $N=2$ superconformal algebra (\cite{CDL}), the topological $N = 2$ super-$BMS_3$ algebra (\cite{LSZ}). The categories of $\mathcal{U}(\mathfrak{h})$-free modules of rank 2 as well as rank $(1|1)$ over $\mathfrak{sl}(m|1)$ was studied and classified in \cite{DN} . Building on the framework of \cite{DN}, in this paper, we construct a family of indecomposable $\mathcal{U}(\mathfrak{h})$-free modules of rank $(k|k)$ over $\mathfrak{sl}(m|1)$.

\np
 The paper is organized as follows. Section~2 collects preliminaries: the Weyl superalgebra $\mathcal{D}(m|n)$; weight modules for $\mathfrak{sl}(m|1)$ and for $\mathcal{D}(m|1)$; and the category $\mathcal{M}_{\mathfrak{sl}(m|1)}(k|k)$ of $\mathcal{U}(\mathfrak{sl}(m|1))$-supermodules, which are $\mathcal U(\mathfrak h)\text{-free}$ of rank $k|k$ (rank $k$ in each parity). Section~3 introduces the notion of generalized permutation matrices over a ring $\mathcal R$ and constructs a family of $\mathcal U(\mathfrak{sl}(m|1))$-modules of rank $k|k$, parametrized by an $m$-tuple of $k\times k$ generalized permutation matrices (Definition \ref{M(A_i)}). These serve as tools for studying $\mathcal U(\mathfrak h)\text{-free}$ modules of rank $(k|k)$, which occur as a special case of the exponential modules defined next. In Section~4, we give the definition of exponential modules for $\mathfrak{sl}(m|1)$, depending on a polynomial $g(x_1,\dots,x_m)\in\C[x_1,\dots,x_m]$ and a subset $\mathcal S\subseteq\{1,\dots,m\}$.  If
\begin{equation}\label{exponential module condition}
    g(x_1,\dots,x_m)=\sum_{i=1}^m a_i x_i^{k_i}\quad(a_i\in\C^\times,\ k_i\in\Z_{\geq1}),
\end{equation}
then these exponential modules are $\mathcal U(\mathfrak h)\text{-free}$ modules of rank $(K|K)$ with $K:=\prod_{i=1}^m k_i$. We show that these modules are indecomposable if and only if $\gcd(k_1,\dots,k_m)=1$ (Theorem \ref{maintheorem_indec}), and give an explicit isomorphism criterion for them (Theorem \ref{main-isom}). 

\np
In particular, we present the moduli space of exponential modules satisfying \eqref{exponential module condition} for a fixed $K:=\prod_{i=1}^m k_i$
as a finite union of weighted projective spaces (Corollary~\ref{isomclass}).

\section{Preliminaries and Notation}
\np
Throughout this paper, we denote the sets of integers, complex numbers, and nonzero complex numbers by $\Z$, $\C$, and $\C^\times$, respectively. For any integer $k$, we write $\Z_{\geq k}$ for the set $\{i\in \Z: i\geq k\}$. Given a Lie (super)algebra $\mathfrak{g}$, we denote its universal enveloping algebra by $\mathcal{U}(\mathfrak{g})$, and we fix a Cartan subalgebra $\mathfrak{h} \subseteq \mathfrak{g}$. For a ring $R$, we write $R^{\times}$  for the group of units in $R$, and $\Mat_N(R)$ for the ring of $N \times N$ matrices with entries in $R$. For a set $\mathcal S$, we denote its cardinality by $\card(\mathcal S)$. All the vector spaces and algebras are assumed to be defined over $\C$. Let $m\in \Z_{\geq 1}$. For ${\bf b}=(b_1,\dots,b_m)\in\C^m$, $l\in\Z$, and $1\le i\le m$, set
$$
\mathbf b + l\varepsilon_i := (b_1,\dots,b_i+l,\dots,b_m),\quad\mathbf b_{[i,m]} := (b_i,\dots,b_m),\quad
\mathbf x^{\mathbf b} := \prod_{t=1}^m x_t^{b_t},
\quad
\mathbf x^{\mathbf b + l\varepsilon_i} := x_i^{\,l}\,\mathbf x^{\mathbf b}.
$$

\np
For any sets $A$ and $B$, we denote their symmetric difference  by $A \symdiffsmall B$, i.e., $A \symdiffsmall B:= \left(A \setminus B \right) \;\cup\; \left(B \setminus A \right). $

\subsection{Basis of  $\mathfrak{sl} (m|1)$}
Throughout, let $\mathbf m:=\{1,\dots,m\}$. For $i, j \in \mathbf{m} \cup \{\bar{1}\}$, we denote by $e_{ij}$ the $(m+1) \times (m+1)$ matrix with a $1$ in position $(i, j)$ and zeros elsewhere. For all $i \in {\bf{m}}$, set $h_{i}  := e_{ii} + e_{\bar{1}\bar{1}}$. Then
 $$
\mathfrak{sl}(m|1)\;=\;
\Span\!\left\{\,h_i,\;e_{\ell j} \bigm|
      i\in \mathbf{m},\;
      \ell,j\in\mathbf{m}\cup\{\bar 1\},\;
      \ell\neq j\right\},\quad\mathfrak{h} = \Span\left\{ h_{i}  \bigm| i\in {\bf{m}}\right\}.$$
Note that $\mathcal{U}(\h) \simeq {\mathbb C} [{\bf h}]$, where ${\mathbb C} [{\bf h}]:= \C[h_{1},\dots,h_{m}]$.

\subsection{Automorphisms of ${\mathbb C} [{\bf h}]$ }
For $ i \in {\bf{m}} $, let $\sigma_i \in \Aut( {\mathbb C} [\bf h] ) $ be defined by 
$$
\sigma_i\; \Big(f(h_1,\dots,h_m)\Big)\;=\;f(h_1,\dots,h_i-1,\dots,h_m),
\quad\text{for}\quad f(h_1,\dots,h_m)\in \C[\mathbf h].
$$
Set
$$
\Delta \;:=\; \prod_{i=1}^m \sigma_i \;=\; \sigma_1\dots\sigma_m,
\qquad
\Delta^{-1} \;:=\; \prod_{i=1}^m \sigma_i^{-1} \;=\; \sigma_1^{-1}\dots\sigma_m^{-1}.
$$
For each $i\in \mathbf m$, set
$$
\Delta_i:=\prod_{j\in \mathbf m\setminus\{i\}}\sigma_j
=\prod_{\substack{j=1\\ j\neq i}}^{m}\sigma_j
=\Delta\,\sigma_i^{-1},\qquad
\Delta_i^{-1}:=\prod_{j\in \mathbf m\setminus\{i\}}\sigma_j^{-1}
=\Delta^{-1}\sigma_i.
$$
Explicitly, for $f(h_1,\dots,h_m)\in\C[\mathbf h]$,

$$\Delta^{\pm1} \; \Big(f(h_1,\dots,h_m)\Big) 
  \;=\; f(h_1\mp1,\dots,h_m\mp1),$$
$$\Delta_i^{\pm1}\; \Big(f(h_1,\dots,h_m)\Big) 
  \;=\; f(h_1\mp1,\dots,h_{i-1}\mp1,\;h_i,\;h_{i+1}\mp1,\dots,h_m\mp1).$$

\begin{remark}
For $i,j \in \mathbf m$, 
    $$\Delta_i^{-1}\big(\Delta_j(f({\bf{h}})\big) = \left(\prod_{\ell\in{\bf{m}}\setminus \{i\}} \sigma^{-1}_\ell \right)\!\!\!\left(\prod_{p\in{\bf{m}}\setminus \{j\}} \sigma_p \right)(f({\bf{h}})) = \sigma_i \sigma^{-1}_j(f({\bf{h}})).$$
\end{remark}

\subsection{Weyl superalgebras}
Let $m,n \in \Z_{\geq 0}$. The \emph{Weyl superalgebra} ${\mathcal D} (m|n)$ is the associative superalgebra with generators $\left\{x_i,\partial_{x_i}\bigm| i=1,\dots,m;-1,\dots,-n\right\}$ of parity
$$\overline{x_i} = \overline{\partial_{x_i}} = \begin{cases} 0~\text{if}~ i>0\\ 1 ~\text{if}~ i<0\end{cases},$$
satisfying the relations
$$[x_i,x_j]=\left[\partial_{x_i},\partial_{x_j}\right]=0, \quad \left[\partial_{x_i}, x_j\right]= \delta_{ij}, $$
where $[u,v]:=uv - (-1)^{\bar{u}\bar{v}}vu$.

\begin{remark}
For $n=1$, $\mathcal D(m|1)=\left\langle x_i,\partial_{x_i},\xi,\partial_\xi\bigm|i=1,\dots,m\right\rangle$ with
$\overline{x_i}=\overline{\partial_{x_i}}=0$ and $\overline{\xi}=\overline{\partial_\xi}=1$.
\end{remark}

\subsection{Weight modules over $\mathfrak{sl}(m|1)$ and $\mathcal{D}(m|1)$ }
\begin{definition}
    We will say that an $\mathfrak{sl}(m|1)$-module $M$ is a \emph{weight module} if 
	$$
	M = \bigoplus_{\lambda \in \h^*} M^{\lambda},\quad \text{where}\quad
	M^{\lambda} = \{m \in M \bigm| hm = \lambda(h)m, \mbox{ for every }h \in \h \}.$$
\end{definition}
\np
To introduce the notion of a weight module over ${\mathcal D} (m|1)$, we first denote 
 $$\mathfrak{h}_{{\mathcal D} (m|1)} := \Span \left\{x_i\partial_{x_i},\xi\partial_{\xi} \bigm| i\in {\bf{m}}  \right\}$$
 and define the following subset
 $$\mathfrak{h}^{\vee}_{{\mathcal D} (m|n)} := \Span \left\{ \mu \in \mathfrak{h}^*_{{\mathcal D} (m|1)}  \bigm| \mu(\xi\partial_{\xi}) 
 \in \{0,1\} \right\}. $$

 \begin{definition} \label{D(m|1)-weightmodule}
    A ${\mathcal D} (m|1)$-module $M$ is a weight module if $M$ decomposes as 
    $$M = \bigoplus_{ \mu \in \mathfrak{h}^{\vee}_{{\mathcal D} (m|1)}} M^{\mu},\quad M^{\mu}:= \left\{m\in M \bigm| u\cdot m = \mu(u)m, \;\text{for all}\;\; u \in \mathfrak{h}_{{\mathcal D} (m|1)}\right\}. $$
\end{definition}

 \subsection{The category $\mathcal{M}_{\mathfrak{sl}(m|1)}\bigl(k|k\bigr)$ over $\mathcal{U}(\mathfrak{sl}(m|1))$}
Let $k\in\Z_{\geq 1}$. Define the full subcategory $\mathcal{M}_{\mathfrak{sl}(m|1)}(k|k)\subset \mathcal{U}(\mathfrak{sl}(m|1))\text{-}\mathrm{mod}$,
consisting of $\Z_2$–graded modules $M=M_{\bar0}\oplus M_{\bar1}$ such that, upon restriction to $\mathcal U(\mathfrak{h})$,
$$
\Res^{\mathcal{U}(\mathfrak{sl}(m|1))}_{\mathcal{U}(\mathfrak{h})} M \;\simeq\;
M_{\bar0}\oplus M_{\bar1},\qquad
M_{\bar i}\simeq \mathcal{U}(\mathfrak{h})^{\oplus k}\; (i=0,1),
$$
as $\mathcal{U}(\mathfrak{h})$–modules. Morphisms in the category $\mathcal{M}_{\mathfrak{sl}(m|1)}(k|k)$ are $\Z_2$-graded $\mathcal{U}(\mathfrak{sl}(m|1))$-homomorphisms.
\np
\begin{remark}
For $k\in \Z_{\geq 1}$, we call the objects of $\mathcal{M}_{\mathfrak{sl}(m|1)}(k|k)$ as \emph{$\mathcal{U}(\mathfrak{h})$-free modules of rank $(k|k)$}.
\end{remark}

\begin{remark}
Throughout the paper, we will often identify the underlying $\Z_2$-graded vector space of any module $M\in\mathcal{M}_{\mathfrak{sl}(m|1)}(k|k)$ with
$$
M=M_{\bar0}\,\oplus\, M_{\bar1}\;=\; \C[{\bf h}]^{\oplus k}\,\oplus\,\C[{\bf h}]^{\oplus k}.
$$
\np
We write elements of $M$ as column vectors $\bigl[f_1({\bf h})\;\;\dots\;\;f_{2k}({\bf h})\bigr]^{\mathsf T}$ with $f_j({\bf h})\in\C[{\bf h}]$. Throughout, $h_i$ acts on $M$ by componentwise multiplication. Namely, for $i\in{\bf m}$,
$$h_i \; \cdot \; \bigl[f_1({\bf h})\;\;\dots\;\;f_{2k}({\bf h})\bigr]^{\mathsf T} \;=\; \bigl[h_if_1({\bf h})\;\;\dots\;\;h_if_{2k}({\bf h})\bigr]^{\mathsf T}. $$
\end{remark}

\begin{lemma}\label{structuresl(m|1)}
For $M \in \mathcal{M}_{\mathfrak{sl}(m|1)}\bigl(k|k\bigr)$, fix an identification $M\;=\;\C[{\bf h}]^{\oplus k} \oplus \C[{\bf h}]^{\oplus k}$ and write $\mathbf f({\bf h}) := \bigl[f_1({\bf h})\;\;\dots\;\;f_{2k}({\bf h})\bigr]^{\mathsf T}$, where $f_r({\bf h}) \in \C[\mathbf h]$. Then, for $i,j\in\mathbf m$ with $i\neq j$,
$$e_{i\bar 1}\cdot \mathbf f({\bf h})
  = E_{i\bar 1}({\bf h})\,\Delta_i^{-1}\bigl(\mathbf f({\bf h})\bigr),\;\; e_{\bar 1 i}\cdot \mathbf f({\bf h})
  = E_{\bar 1 i}({\bf h})\,\Delta_i\bigl(\mathbf f({\bf h})\bigr),\;\; e_{ij}\cdot \mathbf f({\bf h})
  = E_{ij}({\bf h})\,\sigma_i\sigma_j^{-1}\bigl(\mathbf f({\bf h})\bigr).$$
\np  
Moreover,
$$
E_{i\bar 1}({\bf h})=
\begin{bmatrix}
0 & A_{i\bar 1}({\bf h})\\
B_{i\bar 1}({\bf h}) & 0
\end{bmatrix},
\quad
E_{\bar 1 i}({\bf h})=
\begin{bmatrix}
0 & A_{\bar 1 i}({\bf h})\\
B_{\bar 1 i}({\bf h}) & 0
\end{bmatrix},
\quad 
E_{ij}({\bf h})=
\begin{bmatrix}
A_{ij}({\bf h}) & 0\\
0 & B_{ij}({\bf h})
\end{bmatrix},
$$
where $A_{IJ}({\bf h}), B_{IJ}({\bf h}) \in \Mat_k\bigl(\C[{\bf h}]\bigr)$ for all $I,J \in {\bf m}\cup\{\bar 1\}$.
\end{lemma}

\begin{proof}
For $\ell \in \{1,2,\dots,2k\},$  let $e_\ell$ be the column vector with a $1$ in the $\ell$-th position and $0$ elsewhere. Since $h_ie_{i\bar{1}} = e_{i\bar{1}}h_i $, and $h_je_{i\bar{1}} - e_{i\bar{1}}h_j = -e_{i\bar{1}}$ for all $j\neq i$, it follows (by induction on the degree) that, for every $g({\bf h})\in\C[{\bf h}]$,
$$e_{i\bar{1}}\; g({\bf{h}}) = \left(\prod_{j\in{\bf{m}}\setminus \{i\}} \sigma^{-1}_j \right)\!\!\Big(g({\bf{h}})\Big)\;e_{i\bar{1}} = \Delta^{-1}_i\big(g({\bf{h}})\big)\;e_{i\bar{1}}.$$

\np
Therefore,
$$e_{i\bar{1}}\cdot \mathbf f({\bf h})
=\sum_{\ell=1}^{2k} \Delta_i^{-1}\!\bigl(f_\ell({\bf h})\bigr)\,\bigl(e_{i\bar{1}}\cdot e_\ell\bigr)
=\Bigl[\,e_{i\bar{1}}\cdot e_1 \ \; \dots \ \; e_{i\bar{1}}\cdot e_{2k}\,\Bigr]\,
\Delta_i^{-1}\!\left( \mathbf f({\bf h}) \right).$$
\np
Define $E_{i\bar{1}}({\bf h})
:= \bigl[\,e_{i\bar{1}}\!\cdot e_1 \ \dots \ e_{i\bar{1}}\!\cdot e_{2k}\,\bigr]
\;\in\; \Mat_{2k}\bigl(\C[{\bf h}]\bigr)$. Then, for any column vector $\mathbf f({\bf h}) := \bigl[f_1({\bf h})\;\;\dots\;\;f_{2k}({\bf h})\bigr]^{\mathsf T}$, $e_{i\bar{1}}\cdot
\mathbf f({\bf h})
\;=\;
E_{i\bar{1}}({\bf h})\,
\Delta_i^{-1}\!\left(
\mathbf f({\bf h})\right)$. Similarly, we deduce
$$
e_{\bar{1} i}\cdot
\mathbf f({\bf h})
=
E_{\bar{1}i}({\bf h})\,
\Delta_i\left(
\mathbf f({\bf h})
\right),
\qquad
e_{ij}\cdot
\mathbf f({\bf h})
=
E_{ij}({\bf h})\,
\sigma_i\sigma_j^{-1}\!\left(
\mathbf f({\bf h})
\right),
$$
where
$$
E_{\bar{1}i}({\bf h})
:= \bigl[\,e_{\bar{1}i}\cdot e_1 \ \dots \ e_{\bar{1}i}\cdot e_{2k}\,\bigr],\quad E_{ij}({\bf h})
:= \bigl[\,e_{ij}\cdot e_1 \ \dots \ e_{ij}\cdot e_{2k}\,\bigr]
\in \Mat_{2k}\!\bigl(\C[{\bf h}]\bigr).$$

\np
Since, for $\epsilon\in\{0,1\}$,
$$
x\cdot M_{\bar\epsilon}\subseteq M_{\overline{\epsilon+1}}  \quad\text{for all }x\in\{e_{i\bar1},\,e_{\bar1 i}\}, \qquad\text{whereas}\qquad e_{ij}\cdot M_{\bar\epsilon}\subseteq M_{\bar\epsilon}, $$
we conclude that the representing matrices must have the forms
$$
E_{i\bar 1}({\bf h})=
\begin{bmatrix}
0 & A_{i\bar 1}({\bf h})\\
B_{i\bar 1}({\bf h}) & 0
\end{bmatrix},
\quad
E_{\bar 1 i}({\bf h})=
\begin{bmatrix}
0 & A_{\bar 1 i}({\bf h})\\
B_{\bar 1 i}({\bf h}) & 0
\end{bmatrix},
\quad 
E_{ij}({\bf h})=
\begin{bmatrix}
A_{ij}({\bf h}) & 0\\
0 & B_{ij}({\bf h})
\end{bmatrix},
$$
with $A_{IJ}({\bf h}),B_{IJ}({\bf h})\in \Mat_k\bigl(\C[{\bf h}]\bigr)$ for all $I,J\in {\bf m}\cup\{\bar 1\}$.
\end{proof}

\subsection{The category $ \mathcal{M}_{\mathfrak{sl}(m|1)}\bigl(1|1\bigr)$} 
In this subsection, we recall the description of the category $\mathcal{M}_{\mathfrak{sl}(m|1)}(1|1)$ from \cite{DN}.

\begin{definition} \label{sl(m|1)rank2}
    Let $\mathbf{a}:=(a_1,\dots,a_m)\in(\mathbb{C}^{\times})^{m}$ and let $\mathcal{S}\subseteq\mathbf{m}$.
Define
$$
M(\mathbf{a},\mathcal{S})
:=\C[{\bf h}] \;\oplus \; \C[{\bf h}]
$$
endowed with the following $\mathcal{U}(\mathfrak{sl}(m|1))$-action: for $i\in{\bf m}$ and $f_1({\bf h}),f_2({\bf h})\in\C[{\bf h}]$,
    $$h_i \cdot \begin{bmatrix} f_1({\bf{h}})  \\ f_{2}({\bf{h}})\end{bmatrix} = \begin{bmatrix}h_if_1({\bf{h}})  \\ h_if_{2}({\bf{h}}) \end{bmatrix}.$$
 \begin{enumerate}  
   \item[$\bullet$] If $i\in\mathcal S$,
$$
e_{i\bar{1}}\cdot
\begin{bmatrix} f_1({\bf h})\\ f_2({\bf h}) \end{bmatrix}
=
\begin{bmatrix} 0 & a_i h_i\\ 0 & 0 \end{bmatrix}
\Delta_i^{-1}
\left(\begin{bmatrix} f_1({\bf h})\\ f_2({\bf h}) \end{bmatrix}\right),
\qquad
e_{\bar{1} i}\cdot
\begin{bmatrix} f_1({\bf h})\\ f_2({\bf h}) \end{bmatrix}
=
\begin{bmatrix} 0 & 0\\ a_i^{-1} & 0 \end{bmatrix}
\Delta_i
\left(\begin{bmatrix} f_1({\bf h})\\ f_2({\bf h}) \end{bmatrix}\right).
$$
\item[$\bullet$]If $i\notin\mathcal S$,
$$
e_{i\bar{1}}\cdot
\begin{bmatrix} f_1({\bf h})\\ f_2({\bf h}) \end{bmatrix}
=
\begin{bmatrix} 0 & a_i\\ 0 & 0 \end{bmatrix}
\Delta_i^{-1}
\left(\begin{bmatrix} f_1({\bf h})\\ f_2({\bf h}) \end{bmatrix}\right),
\qquad
e_{\bar{1} i}\cdot
\begin{bmatrix} f_1({\bf h})\\ f_2({\bf h}) \end{bmatrix}
=
\begin{bmatrix} 0 & 0\\ a_i^{-1} h_i & 0 \end{bmatrix}
\Delta_i
\left(\begin{bmatrix} f_1({\bf h})\\ f_2({\bf h}) \end{bmatrix}\right).
$$
\end{enumerate}
\end{definition}

\np
\begin{remark}
    We omit an explicit closed formula for the action of $e_{ij}$ with distinct $i,j\in{\bf m}$, since it is determined by the actions of odd generators $e_{i\bar1}$ and $e_{\bar1 j}$. Indeed,
$$e_{ij} \cdot \begin{bmatrix} f_1({\bf{h}})  \\ f_{2}({\bf{h}})\end{bmatrix} = e_{i\bar{1}}\cdot e_{\bar{1}j} \cdot \begin{bmatrix} f_1({\bf{h}})  \\ f_{2}({\bf{h}})\end{bmatrix} +  e_{\bar{1}j} \cdot e_{i\bar{1}}\cdot \begin{bmatrix} f_1({\bf{h}})  \\ f_{2}({\bf{h}})\end{bmatrix}. $$
For instance, if $i\in \mathcal{S}$ and $j\notin \mathcal{S}$, then
$$e_{ij} \cdot \begin{bmatrix} f_1({\bf{h}})  \\ f_{2}({\bf{h}})\end{bmatrix} = \dfrac{a_i}{a_j}\begin{bmatrix} h_i(h_j+1) & 0 \\ 0 & (h_i-1)h_j \end{bmatrix}\, \sigma_i \sigma_j^{-1}\,\left(\begin{bmatrix} f_1({\bf{h}})  \\ f_{2}({\bf{h}})\end{bmatrix} \right) .$$
\end{remark}

\begin{theorem}[\cite{DN}, Thm.~4.10]\label{classisl(m|1)}
The following hold.
\begin{enumerate}
    \item For ${\bf a}=(a_1,\dots,a_m)\in(\C^\times)^m$ and $\mathcal S\subseteq{\bf m}$, the actions in Definition~\ref{sl(m|1)rank2} endow $M({\bf a},\mathcal S)$ with a $\mathcal U(\mathfrak{sl}(m|1))$-module structure.
    \item If $M\in \mathcal{M}_{\mathfrak{sl}(m|1)}(1|1)$, then
    $M \,\simeq\, M({\bf a},\mathcal S)$,
    for some ${\bf a}\in(\C^\times)^m$ and $\mathcal S\subseteq{\bf m}$.
\end{enumerate}
\end{theorem}

\subsection{Modules over ${\mathfrak{sl}(1|1)}$} 
For $\mathcal U(\mathfrak h)$-free modules of rank $(1|1)$ over $\mathcal U(\mathfrak{sl}(1|1))$, we replace Definition~\ref{sl(m|1)rank2} with the following.

\begin{definition}
Let $\alpha\in\C^\times$. Endow $M(\alpha h_1)$ and $M(\alpha)$ with the same $\Z_2$-graded vector space
$$
M(\alpha h_1)=M(\alpha)=\C[h_1]\,\oplus \, \C[h_1],
$$
with the $\mathfrak h$–action
$$
h_1\cdot
\begin{bmatrix} f_1(h_1)\\ f_2(h_1) \end{bmatrix}
=
\begin{bmatrix} h_1 f_1(h_1)\\ h_1 f_2(h_1) \end{bmatrix}.
$$
The odd generators act as follows:
\begin{enumerate}
\item[$\bullet$] $M(\alpha h_1)$:
$$
e_{1\bar1}\cdot
\begin{bmatrix} f_1(h_1)\\ f_2(h_1) \end{bmatrix}
=
\begin{bmatrix} 0 & \alpha h_1\\ 0 & 0 \end{bmatrix}
\begin{bmatrix} f_1(h_1)\\ f_2(h_1) \end{bmatrix},
\qquad
e_{\bar11}\cdot
\begin{bmatrix} f_1(h_1)\\ f_2(h_1) \end{bmatrix}
=
\begin{bmatrix} 0 & 0\\ \alpha^{-1} & 0 \end{bmatrix}
\begin{bmatrix} f_1(h_1)\\ f_2(h_1) \end{bmatrix}.
$$

\item[$\bullet$]$M(\alpha)$:
$$
e_{1\bar1}\cdot
\begin{bmatrix} f_1(h_1)\\ f_2(h_1) \end{bmatrix}
=
\begin{bmatrix} 0 & \alpha\\ 0 & 0 \end{bmatrix}
\begin{bmatrix} f_1(h_1)\\ f_2(h_1) \end{bmatrix},
\qquad
e_{\bar11}\cdot
\begin{bmatrix} f_1(h_1)\\ f_2(h_1) \end{bmatrix}
=
\begin{bmatrix} 0 & 0\\ \alpha^{-1} h_1 & 0 \end{bmatrix}
\begin{bmatrix} f_1(h_1)\\ f_2(h_1) \end{bmatrix}.
$$
\end{enumerate}
Here $f_1,f_2\in\C[h_1]$. These actions endow $M(\alpha h_1)$ and $M(\alpha)$ with $\mathcal U(\mathfrak{sl}(1|1))$-module structures.
\end{definition}

\np
\begin{proposition}\label{simplified-sl(1|1)}
For any $\alpha\in \C^\times$, $M(\alpha h_1)\simeq M(h_1)$ and $M(\alpha)\simeq M(1)$. Moreover, every $M\in\mathcal{M}_{\mathfrak{sl}(1|1)}(1|1)$ is isomorphic to either $M(h_1)$ or $M(1)$.
\end{proposition}

\section{Generalized permutation matrix and $\mathcal{U}(\mathfrak{h})$-free modules of finite rank over $\mathfrak{sl}(m|1)$}
\np
To generalize the rank-$(1|1)$ module $M(\mathbf a,\mathcal S)$ (Definition~\ref{sl(m|1)rank2}) to higher-rank
$\mathcal U(\mathfrak h)$–free modules, we first introduce generalized permutation matrices over a ring $\mathcal R$.

\subsection{Generalized permutation matrices and $h_i$-companion pairs}

\begin{definition}[\textbf{Generalized permutation matrix (GPM)}] Let $\mathcal{R}$ be a ring, and $n \in \Z_{\geq 1}$. A matrix $A \in \Mat_n(\mathcal{R})$ is called a \emph{generalized permutation matrix (GPM)} if each row and each column contains exactly one non‑zero entry.
\end{definition}

\np
Let $A\in\Mat_n(\mathcal R)$ be a GPM.
The positions of the non-zero entries in $A$ determine a permutation 
matrix $P_\pi$ associated to a permutation $\pi\in S_n$.
Furthermore $A$ admits a unique factorization:
$$
A=P_\pi D,\qquad
D=\diag\bigl(A_{\pi(1),1},\dots,A_{\pi(n),n}\bigr),
$$
as the product of a permutation matrix and  a diagonal matrix. 
We write $\pi(A)$ to denote the permutation $\pi$.

\begin{example} The matrix
$$A = \begin{bmatrix} u_1 &0&0&0\\ 0&0&u_3 &0 \\0& 0 & 0 & u_4 \\ 0&u_2 &0&0\end{bmatrix},$$
where $u_i\in\C[x]\setminus\{0\}$ for $i \in \{1,2,3,4\}$,
is a GPM with $\pi(A) = (2\,4\,3)$ and $D = \diag \bigl(u_1, u_2, u_3, u_4\bigr)$.
\end{example}

\begin{definition}[\textbf{$h_i$-companion pair}]
Let $i\in{\bf m}$ and $\ell\in\Z_{\ge1}$. A pair $(A,A^{\comp})$ with
$A,A^{\comp}\in \Mat_\ell\bigl(\C[h_i]\bigr)$ is called an \textbf{$h_i$-companion pair} if
$$
\text{$A$ is a GPM}\quad\text{and}\quad
A\,A^{\comp}=A^{\comp}A=h_i\,\I_\ell.
$$
\end{definition}

\begin{remark}
   If $(A,A^{\comp})$ is an $h_i$-companion pair, then $A^{\comp}$ is also a GPM and
$$
\pi\!\left(A^{\comp}\right)=\pi(A)^{-1}.
$$
\end{remark}
\np
\begin{example}
    For $a_1, a_2, a_3 \in \C^{\times}$, the pair $\left(\begin{bmatrix}a_1 &0 &0 \\0&0&a_2h_i \\0&a_3h_i &0 \end{bmatrix}, \begin{bmatrix}a_1^{-1}h_i &0 &0 \\0&0&a_3^{-1} \\0&a_2^{-1} &0 \end{bmatrix} \right) $ is an $h_i$-companion pair.
\end{example}

\np
\begin{remark}\label{observation-sl(m|1)-evenrank}
Since $A$ and $A^{\comp}$ are GPMs and
$
A\,A^{\comp}=h_i\,\I_\ell,
$
each nonzero entry of $A$ (and likewise of $A^{\comp}$) is either a nonzero scalar or a nonzero scalar times $h_i$.
\end{remark}

\np
\begin{lemma}\label{equivofA_12}
Fix $i\in{\bf m}$ and $\ell\geq 1$. Let $(A,A^{\comp})$ and $(B,B^{\comp})$ be $h_i$-companion pairs with
$A,A^{\comp},B,B^{\comp}\in \Mat_\ell(\C[h_i])$, and let $W(\mathbf h),V(\mathbf h)\in \Mat_\ell(\C[\mathbf h])$. Then the following are equivalent:
\begin{enumerate}
\item $\,W(\mathbf h)\,A \;=\; B\,\Delta_i^{-1}\bigl(V(\mathbf h)\bigr)$;
\item $\,V(\mathbf h)\,A^{\comp} \;=\; B^{\comp}\,\Delta_i\bigl(W(\mathbf h)\bigr)$.
\end{enumerate}
\end{lemma}

\begin{proof}
Since $(A,A^{\comp})$ and $(B,B^{\comp})$ are $h_i$-companion pairs, we have
$$AA^{\comp}=A^{\comp}A= BB^{\comp}=B^{\comp}B=h_i\,\I_\ell.$$
Moreover, $\Delta_i$ acts trivially on $\C[h_i]$, hence on any matrix with
entries in $\C[h_i]$; in particular,
$$
\Delta_i(A)=A,\quad \Delta_i(A^{\comp})=A^{\comp},\quad \Delta_i(B)=B,\quad \Delta_i(B^{\comp})=B^{\comp}.
$$
\np
The statement follows from the chain of equivalences
$$
\begin{aligned}
   & W(\mathbf{h})\,A
     \;=\;
     B\,\Delta_i^{-1}\bigl(V(\mathbf{h})\bigr)\\
   \Longleftrightarrow\;&
     \Delta_i\bigl(W(\mathbf{h})\,A\bigr)
     \;=\;  \Delta_i\Bigl(B\,\Delta_i^{-1}\bigl(V(\mathbf{h})\bigr)\Bigr)\\
   \Longleftrightarrow\;&
     \Delta_i\bigl(W(\mathbf{h})\bigr)\,A
     \;=\;
     B\,V(\mathbf{h})\\
   \Longleftrightarrow\;&
     B^{\comp}\,\Delta_i\bigl(W(\mathbf{h})\bigr)\,A\,A^{\comp}
     \;=\;
     B^{\comp}\,B\,V(\mathbf{h})\,A^{\comp}\\
   \Longleftrightarrow\;&
     B^{\comp}\,\Delta_i\bigl(W(\mathbf{h})\bigr)\,(h_i\,\mathbf{I}_\ell)
     \;=\;
     (h_i\,\mathbf{I}_\ell)\,V(\mathbf{h})\,A^{\comp}\\
   \Longleftrightarrow\;&
     B^{\comp}\,\Delta_i\bigl(W(\mathbf{h})\bigr)
     \;=\;
     V(\mathbf{h})\,A^{\comp}. 
\end{aligned} 
$$
\end{proof}
\np
The next two lemmas will be used in the proof of Theorem~\ref{maintheorem_indec}.

\begin{lemma}\label{simplifymatrixentries}
Let $g(\mathbf h)\in \C[\mathbf h]$ and $k\in\Z_{\geq 1}$. Fix distinct indices $i,j\in\mathbf m$.
Let $\ell\geq 1$ and choose integers
$$
1\leq i_1 < j_1 < i_2 < j_2 < \dots < i_\ell < j_\ell \leq k,
\quad\text{or}\quad
1\leq j_1 < i_1 < j_2 < i_2 < \dots < j_\ell < i_\ell \leq k.
$$
Then
$$
  g(\mathbf h)
  \;=\;
  \prod_{r=1}^{\ell}\frac{h_j+i_r}{\,h_j+j_r\,}\,
  \bigl(\sigma_i\sigma_j^{-1}\bigr)^{k}\bigl(g(\mathbf h)\bigr)
  \quad\text{implies}\quad
  g(\mathbf h)=0.
  $$
\end{lemma}

\begin{proof}
    Since 
 $$
  g(\mathbf h)
  \;=\;
  \prod_{r=1}^{\ell}\frac{h_j+i_r}{\,h_j+j_r\,}\,
  \bigl(\sigma_i\sigma_j^{-1}\bigr)^{k}\bigl(g(\mathbf h)\bigr),$$
it follows that $(h_j+ i_r) \,|\, g(\mathbf h)$, for all $r \in \{1,\dots,\ell\}$. Iterating the same argument (equivalently, replacing $h_j$ by $h_j+pk$) gives
$$(h_j+ i_r +pk) \,|\, g(\mathbf h),\quad \text{for all}\quad  p\in \Z_{\geq 0}.$$
\np
Hence, $g(\mathbf h)=0.$
\end{proof}
\np
\begin{lemma} \label{simplifymatrixentries2}
Let $k \in \Z_{\geq 1}$ and $\alpha, \beta \in \C$ with $\alpha \neq \beta$ and $|\Re(\alpha)-\Re(\beta)|<k$. 
Let $W({\bf{h}}) \in \Mat_k(\C[{\bf{h}}])$ and let $A(h_j)\in \Mat_k(\C[h_j])$ be a GPM such that
$$
W({\bf{h}}) 
= \frac{h_i - \alpha}{h_i - \beta}\,
A(h_j)\,\bigl(\sigma_i\sigma_j^{-1}\bigr)^{k}\bigl(W({\bf{h}})\bigr)\,A(h_j)^{-1},
\quad \text{where } i,j \in {\bf{m}},\; i \neq j.
$$
Then $W({\bf{h}})=0$.
\end{lemma}

\begin{proof}
    Let $W({\bf{h}}) = \bigl(w_{p,q}({\bf{h}})\bigr)_{p,q=1}^k$, where $w_{p,q}({\bf{h}}) \in \C[{\bf{h}}]$. Since 
     $$
W({\bf{h}}) = \frac{h_i - \alpha}{h_i - \beta}\,
A(h_j)\,\bigl(\sigma_i\sigma_j^{-1}\bigr)^{k}\bigl(W({\bf{h}})\bigr)\,A(h_j)^{-1},
$$
then for $p,q \in \{1,\dots, k\}$, there exists $n(p,q) \in \Z_{\geq 1}$ such that
\begin{equation} \label{eqlem}
    w_{p,q}({\bf{h}}) 
= f_{p,q}(h_j)\,
\prod_{\ell=0}^{n(p,q)-1} \frac{h_i - \alpha - \ell k}{h_i - \beta - \ell k}\;
\bigl(\sigma_i\sigma_j^{-1}\bigr)^{k\,n(p,q)}\bigl(w_{p,q}({\bf{h}})\bigr), 
\end{equation}
for some $f_{p,q}(h_j) \in \C(h_j)$. Since the real parts of the roots
$$\beta,\; \beta + k,\; \dots,\; \beta + (n(p,q)-1) k$$
differ by multiples of $k$, $|\Re(\alpha)-\Re(\beta)|<k$, and $\alpha \neq \beta$,
it follows that 
$$(h_i-\beta) \,|\,\bigl(\sigma_i\sigma_j^{-1}\bigr)^{k\,n(p,q)}\bigl(w_{p,q}({\bf{h}})\bigr),$$
which implies
$$(h_i-\beta+k\,n(p,q)) \,|\,w_{p,q}({\bf{h}})\,.$$
\np
Equivalently, there exists $w'(\mathbf h) \in \C[{\bf{h}}]$ such that $w_{p,q}({\bf{h}}) = (h_i-\beta+k\,n(p,q))w'(\mathbf h)$. Substituting this expression into \eqref{eqlem}, we obtain 
$$
w'(\mathbf h) 
= f_{p,q}(h_j)\,
 \frac{\prod_{\ell=0}^{n(p,q)-1}(h_i - \alpha - \ell k)}{(h_i-\beta+k\,n(p,q))\prod_{\ell=1}^{n(p,q)-1}(h_i - \beta - \ell k)}\;
\bigl(\sigma_i\sigma_j^{-1}\bigr)^{k\,n(p,q)}\bigl(w'(\mathbf h)\bigr) .$$
\np
This implies that $(h_i-\beta+k\,n(p,q))$ divides  $\bigl(\sigma_i\sigma_j^{-1}\bigr)^{k\,n(p,q)}\bigl(w'(\mathbf h)\bigr)$
and hence it divides $\bigl(\sigma_i\sigma_j^{-1}\bigr)^{k\,n(p,q)}\bigl(w_{p,q}(\mathbf h)\bigr)$. 
Repeating this argument, we conclude that
$$(h_i-\beta+mk\,n(p,q)) \,|\, w_{p,q}({\bf{h}}),\qquad \text{for all} \quad m \in \Z_{\geq 1},$$
which implies $w_{p,q}({\bf{h}}) =0$. This concludes the proof.
\end{proof}

\np
\subsection{The module $M(A_1,\dots,A_m)$}
 Let $ \ell \in \Z_{\geq 1}$. 
 
 \begin{proposition}\label{sl(m|1)-module}
 Let $A_i, A_i^{\comp} \in \Mat_\ell\bigl(\C[h_i]\bigr)$
 form $h_i$-companion pairs for $1 \leq i \leq m$, i.e., 
\[A_iA_i^{\comp}=A_i^{\comp}A_i=h_i\,\I_\ell \ . \]
Define an action of the generators $e_{i\bar{1}}$, and $e_{\bar{1}i}$ of $\mathcal{U}(\mathfrak{sl}(m|1))$ on \newline
$\begin{bmatrix}f_1({\bf{h}})\\\vdots \\f_{2\ell}({\bf{h}}) \end{bmatrix} \in 
\mathbb{C}[\mathbf{h}]^{\oplus \ell}\;\oplus\;\mathbb{C}[\mathbf{h}]^{\oplus \ell}$ by the formulas
    
    \begin{equation} \label{eq_new_action1}
    \begin{array}{ccl}
    e_{i\bar{1}} \cdot \begin{bmatrix}f_1({\bf{h}})\\\vdots \\f_{2\ell}({\bf{h}}) \end{bmatrix} &=& \left[\begin{array}{ c | c }
    \mathbf 0 & A_i \\
    \hline
    \mathbf 0 & \mathbf 0
  \end{array}\right] \Delta^{-1}_i\left(\begin{bmatrix}f_1({\bf{h}})\\\vdots \\f_{2\ell}({\bf{h}}) \end{bmatrix} \right), \\
  &&\\
 e_{\bar{1}i}  \cdot \begin{bmatrix}f_1({\bf{h}})\\\vdots \\f_{2\ell}({\bf{h}}) \end{bmatrix} &=& \left[\begin{array}{ c | c }
    \mathbf 0 & \mathbf 0 \\
    \hline
    A_i^{\comp} & \mathbf 0
  \end{array}\right] \Delta_i\left(\begin{bmatrix}f_1({\bf{h}})\\\vdots \\f_{2\ell}({\bf{h}}) \end{bmatrix} \right) \ . 
  \end{array}
  \end{equation}
Then \eqref{eq_new_action1} endows $\mathbb{C}[\mathbf{h}]^{\oplus \ell}\;\oplus\;\mathbb{C}[\mathbf{h}]^{\oplus \ell}$
with a $\mathcal{U}(\mathfrak{sl}(m|1))$-module structure. Moreover, this module is $\mathcal{U}(\mathfrak{h})$-free of
rank $(\ell|\ell)$, i.e.,
\begin{equation} \label{eq_new_action2}
h_i \cdot \begin{bmatrix}f_1({\bf{h}})\\\vdots \\f_{2\ell}({\bf{h}}) \end{bmatrix}= \begin{bmatrix}h_if_1({\bf{h}})\\\vdots \\h_if_{2\ell}({\bf{h}}) \end{bmatrix} \ .
\end{equation}

\end{proposition}

\begin{proof} It suffices to check that \eqref{eq_new_action1} and  
\eqref{eq_new_action2} satisfy the commutation relations in $\mathfrak{sl}(m|1)$. We omit this calculation here.
\end{proof}

\np
\begin{definition} \label{M(A_i)}
    We denote the $\mathcal{U}(\mathfrak{sl}(m|1))$-module defined above by $M(A_1,\dots,A_m)$.
\end{definition}

\np Note that in the notation $M(A_1,\dots,A_m)$, we suppress the companion matrices 
$A_1^{\comp}, \dots, A_m^{\comp}$. Nevertheless, whenever we discuss the module 
$M(A_1,\dots,A_m)$, we will assume that $A_1, \dots, A_m$ admit companion matrices.

\np
\begin{remark}
 If $\ell = 1$, then $M(A_1,\dots,A_m)$ coincides with $M(\mathbf{a},\mathcal{S})$ (see Definition \ref{sl(m|1)rank2}) for appropriate $A_i$. Indeed, for $\mathbf a=(a_1,\dots,a_m)\in(\C^\times)^m$ and
$\mathcal S\subseteq\mathbf m$, set
$$
A_i :=
\begin{cases}
[a_i h_i], & i \in \mathcal{S}, \\
[a_i],     & i \notin \mathcal{S},
\end{cases}
\qquad (1 \le i \le m).
$$
Then $M(A_1,\dots,A_m) =  M(\mathbf{a},\mathcal{S}).$
\end{remark}
\np
\begin{proposition} \label{simplified-W(h)}
    For each $i \in {\bf{m}}$, let $(A_i,A_i^{\comp})$ and $(B_i,B_i^{\comp})$ be $h_i$-companion pairs with
$A_i,A_i^{\comp},B_i,B_i^{\comp}\in \Mat_\ell(\C[h_i])$. Then 
\[\Hom_{\mathcal{M}_{\mathfrak{sl}(m|1)}(\ell|\ell)} (M(A_1,\dots,A_m), M(B_1,\dots,B_m))_{\bar{1}} = 0\ .\]
More precisely,
every $\mathcal{U}(\mathfrak{sl}(m|1))$-homomorphism $\Phi \colon M(A_1,\dots,A_m)\to M(B_1,\dots,B_m)$ is of the block–diagonal form
$$\Phi(\mathbf f({\bf h})) = \left[\begin{array}{ c | c }
    W_1({\bf{h}}) & \mathbf 0 \\
    \hline
    \mathbf 0 & W_4({\bf{h}})
  \end{array}\right]\;\mathbf f({\bf h}),\quad \text{for all}\quad \mathbf f({\bf h}) \in \C[{\bf{h}}]^{\oplus 2\ell}, $$
where $W_1({\bf{h}}), W_4({\bf{h}}) \in \Mat_{\ell}(\C[{\bf{h}}])$ satisfy
  $$W_1({\bf h})\,A_i \;=\; B_i\,\Delta_i^{-1}\bigl(W_4({\bf h})\bigr), \quad\text{for all}\quad i \in \mathbf{m}. $$
Moreover,  $\Phi$ is an isomorphism if and only if $ W_1({\bf{h}}), W_4({\bf{h}}) \in \GL_{\ell}\left(\C[{\bf{h}}]\right).$  
\end{proposition}

\begin{proof}
    Let $\Phi:M(A_1,\dots, A_m) \to  M(B_1,\dots, B_m) $ be a $\mathcal{U}(\mathfrak{sl}(m|1))$-homomorphism. Then,
$$\Phi( \mathbf f({\bf h})) = W_{\Phi}({\bf{h}})\;\mathbf f({\bf h})=\left[\begin{array}{ c | c }
    W_1({\bf{h}}) & W_2({\bf{h}}) \\
    \hline
    W_3({\bf{h}}) & W_4({\bf{h}})
  \end{array}\right]\;\mathbf f({\bf h}),\quad \text{for all}\quad \mathbf f({\bf h}) \in \C[{\bf{h}}]^{\oplus 2\ell}, $$
where $W_i({\bf{h}}) \in \Mat_{\ell}(\C[{\bf{h}}])$.
 In particular, $\Phi = \Phi_{\bar{0}} + \Phi_{\bar{1}}$, where 
  $$\Phi_{\bar{0}}\left(\mathbf f({\bf h})\right) = \left[\begin{array}{ c | c }
    W_1({\bf{h}}) & \mathbf 0 \\
    \hline
    \mathbf 0 & W_4({\bf{h}})
  \end{array}\right]\; \mathbf f({\bf h}), \qquad \Phi_{\bar{1}}\left(\mathbf f({\bf h})\right) = \left[\begin{array}{ c | c }
    \mathbf 0 & W_2({\bf{h}}) \\
    \hline
    W_3({\bf{h}}) & \mathbf 0
  \end{array}\right]\; \mathbf f({\bf h}) .$$
\np  
Since 
  $$\Phi\left(e_{i \bar{1}} \cdot \mathbf f({\bf h})\right) =\left(\Phi_{\bar{0}}+ \Phi_{\bar{1}}\right)\left(e_{i \bar{1}} \cdot\mathbf f({\bf h})\right)= e_{i\bar{1}} \cdot \Phi_{\bar{0}}\left(\mathbf f({\bf h})\right) - e_{i\bar{1}} \cdot \Phi_{\bar{1}}\left(\mathbf f({\bf h})\right),\; \text{for all}\;\; \mathbf f({\bf h}) \in \C[{\bf{h}}]^{\oplus 2\ell},$$
we obtain
      \begin{align*}
          \left[\begin{array}{ c | c }
    W_1({\bf{h}}) & W_2({\bf{h}}) \\
    \hline
    W_3({\bf{h}}) & W_4({\bf{h}})
  \end{array}\right] \left[\begin{array}{ c | c }
    \mathbf 0 & A_{i} \\
    \hline
    \mathbf 0 & \mathbf 0
  \end{array}\right] \Delta^{-1}_i\left(\mathbf f({\bf h})\right) &=\left[\begin{array}{ c | c }
    \mathbf 0 & B_{i} \\
    \hline
    \mathbf 0 & \mathbf 0
  \end{array}\right] \Delta^{-1}_i\left(\left[\begin{array}{ c | c }
    W_1({\bf{h}}) & \mathbf 0 \\
    \hline
    \mathbf 0 & W_4({\bf{h}})
  \end{array}\right] \mathbf f({\bf h})\right)\\
  &-\left[\begin{array}{ c | c }
    \mathbf 0 & B_{i} \\
    \hline
    \mathbf 0 & \mathbf 0
  \end{array}\right] \Delta^{-1}_i\left(\left[\begin{array}{ c | c }
    \mathbf 0 & W_2({\bf{h}}) \\
    \hline
    W_3({\bf{h}}) & \mathbf 0
  \end{array}\right] \mathbf f({\bf h})\right) .
      \end{align*}

\np
Then,
    $$\left[\begin{array}{ c | c }
    \mathbf 0 & W_1({\bf{h}})A_i \\
    \hline
    \mathbf 0 & W_3({\bf{h}})A_i
  \end{array}\right]\Delta^{-1}_i\left(\mathbf f({\bf h})\right)=\left[\begin{array}{ c | c }
    -B_i \Delta^{-1}_i\left(W_3({\bf{h}})\right) & B_i \Delta^{-1}_i\left(W_4({\bf{h}})\right) \\
    \hline
    \mathbf 0 & \mathbf 0
  \end{array}\right]\Delta^{-1}_i\left(\mathbf f({\bf h})\right),$$
  for all $\mathbf f({\bf h}) \in \C[{\bf{h}}]^{\oplus 2\ell} $.
 Therefore,
  $$W_1({\bf h})\,A_i \;=\; B_i\,\Delta_i^{-1}\bigl(W_4({\bf h})\bigr),\quad\text{and}\quad B_i\, \Delta^{-1}_i\left(W_3({\bf{h}})\right) \;=\; W_3({\bf{h}})\,A_i \;=\; 0. $$

\np
    Since $A_i$ is a GPM, it follows that  $W_3({\bf{h}}) = 0$. Similarly,
  $$\Phi\left(e_{\bar{1}i} \cdot\mathbf f({\bf h})\right) =\left(\Phi_{\bar{0}}+ \Phi_{\bar{1}}\right)\left(e_{ \bar{1}i} \cdot\mathbf f({\bf h})\right)= e_{\bar{1}i} \cdot \Phi_{\bar{0}}\left(\mathbf f({\bf h})\right) - e_{\bar{1}i} \cdot \Phi_{\bar{1}}\left(\mathbf f({\bf h})\right),\; \text{for all}\;\; \mathbf f({\bf h}) \in \C[{\bf{h}}]^{\oplus 2\ell},$$

implies
\begin{align*}
   \left[\begin{array}{ c | c }
    W_1({\bf{h}}) & W_2({\bf{h}}) \\
    \hline
    W_3({\bf{h}}) & W_4({\bf{h}})
  \end{array}\right] \left[\begin{array}{ c | c }
    \mathbf 0 & \mathbf 0 \\
    \hline
    A_{i}^{\comp} & \mathbf 0
  \end{array}\right] \Delta_i\left(\mathbf f({\bf h})\right) &=\left[\begin{array}{ c | c }
    \mathbf 0 & \mathbf 0 \\
    \hline
    B_{i}^{\comp} & \mathbf 0
  \end{array}\right] \Delta_i\left(\left[\begin{array}{ c | c }
    W_1({\bf{h}}) & \mathbf 0 \\
    \hline
    \mathbf 0 & W_4({\bf{h}})
  \end{array}\right] \mathbf f({\bf h})\right)\\
  &-\left[\begin{array}{ c | c }
    \mathbf 0 & \mathbf 0 \\
    \hline
    B_{i}^{\comp} & \mathbf 0
  \end{array}\right] \Delta_i\left(\left[\begin{array}{ c | c }
    \mathbf 0 & W_2({\bf{h}}) \\
    \hline
    W_3({\bf{h}}) & \mathbf 0
  \end{array}\right] \mathbf f({\bf h})\right).   
\end{align*}
\np
Then,
    $$\left[\begin{array}{ c | c }
    W_2({\bf{h}})A_i^{\comp} & \mathbf 0 \\
    \hline
    W_4({\bf{h}})A_i^{\comp} & \mathbf 0
  \end{array}\right]\Delta_i\left(\mathbf f({\bf h})\right)=\left[\begin{array}{ c | c }
    \mathbf 0 & \mathbf 0 \\
    \hline
    B_i^{\comp} \Delta_i\left(W_1({\bf{h}})\right) & -B_i^{\comp} \Delta_i\left(W_2({\bf{h}})\right)
  \end{array}\right]\Delta_i\left(\mathbf f({\bf h})\right),$$
  for all $\mathbf f({\bf h}) \in \C[{\bf{h}}]^{\oplus 2\ell} $. Therefore, 
$$W_4({\bf{h}})\, A_i^{\comp} \;=\;  B_i^{\comp}\, \Delta_i\left(W_1({\bf{h}})\right),\quad\text{and}\quad B_i^{\comp}\, \Delta_i\left(W_2({\bf{h}})\right) \;=\; W_2({\bf{h}})\,A_i^{\comp} \;=\; 0. $$
\np
    Since $A_i^{\comp}$ is a GPM, it follows that $ W_2({\bf{h}}) = 0$. By Lemma \ref{equivofA_12}, we have
   $$W_1({\bf{h}}) A_i =  B_i \Delta^{-1}_i\left(W_4({\bf{h}})\right) \quad\text{and}\quad W_4({\bf{h}}) A_i^{\comp} =  B_i^{\comp} \Delta_i\left(W_1({\bf{h}})\right) $$
are equivalent. Therefore, every homomorphism 
$$\Phi:M(A_1,\dots, A_m) \;\to\; M(B_1,\dots, B_m) $$ 
can be written as
$$\Phi(\mathbf f({\bf h})) = \left[\begin{array}{ c | c }
    W_1({\bf{h}}) & \mathbf 0 \\
    \hline
    \mathbf 0 & W_4({\bf{h}})
  \end{array}\right]\;\mathbf f({\bf h}),\quad \text{for all}\quad \mathbf f({\bf h}) \in \C[{\bf{h}}]^{\oplus 2\ell}, $$
where $W_1({\bf{h}}), W_4({\bf{h}}) \in \Mat_{\ell}(\C[{\bf{h}}])$ satisfy
  $$W_1({\bf h})\,A_i \;=\; B_i\,\Delta_i^{-1}\bigl(W_4({\bf h})\bigr), \quad\text{for all}\quad i \in \mathbf{m}. $$

\np
  The fact that $\Phi$ is an isomorphism if and only if 
  $ W_1({\bf{h}})$ and $ W_4({\bf{h}})$ are invertible is
 a tautology.
\end{proof}

\np 
\begin{lemma}
    Let $A_i\in\Mat_\ell\bigl(\C[h_i]\bigr)$ $(i\in\mathbf m)$, each admitting an $h_i$–companion $A_i^{\comp}\in\Mat_\ell\bigl(\C[h_i]\bigr)$. Then
    \begin{enumerate}
        \item The $\mathcal{U}(\mathfrak{sl}(m|1))$-module $M(A_1,\dots, A_m)$ has infinite length.
        \item Suppose the index set $\{1,2,\dots,\ell\}$ decomposes as a disjoint union  
      $$\{1,2,\dots,\ell\}\;=\; I \sqcup J,\qquad I\neq \varnothing,\; J\neq \varnothing,$$
      where $I$ and $J$ are $\pi(A_i)$‑stable for every $i \in \mathbf{m}$.  
      Then $M(A_1,\dots,A_m)$ is decomposable.
    \end{enumerate}
\end{lemma}

\begin{proof}
{\bf (1)} We denote $H:=\sum_{j=1}^m h_j$ and write a generic vector of $M(A_1,\dots,A_m)$ as
$$\bigl[f_{-\ell}({\bf h})\;\;\dots\;\;f_{-1}({\bf h})\;\;f_{1}({\bf h})\;\;\dots\;\;f_{\ell}({\bf h})\bigr]^{\mathsf T},\qquad f_j({\bf h})\in \C[{\bf h}].$$
\np
For any $F\in\C[X]$ and $i\in{\bf m}$,
\begin{align*}e_{i\bar1}\cdot
\begin{bmatrix}F(H+m-1)f_{-\ell}({\bf h})\\\vdots\\F(H+m-1)f_{-1}({\bf h})\\F(H)f_{1}({\bf h})\\\vdots\\F(H)f_{\ell}({\bf h})\end{bmatrix}
&=\begin{bmatrix}A_i\begin{bmatrix}F(H+m-1)\Delta_i^{-1}(f_{1}({\bf h}))\\\vdots\\F(H+m-1)\Delta_i^{-1}(f_{\ell}({\bf h}))\end{bmatrix}\\ 0\\\vdots\\ 0\end{bmatrix},\end{align*}

\begin{align*}
    e_{\bar{1}i} \cdot \begin{bmatrix}F(H+m-1)f_{-\ell}({\bf h})\\\vdots\\F(H+m-1)f_{-1}({\bf h})\\F(H)f_{1}({\bf h})\\\vdots\\F(H)f_{\ell}({\bf h})\end{bmatrix} =
  \begin{bmatrix}  0\\ \vdots\\0\\A^{\comp}_{i} \begin{bmatrix} F\left(H \right)\Delta_i(f_{-\ell}(\mathbf h))\\ \vdots \\F\left(H \right)\Delta_i(f_{-1}(\mathbf h))\end{bmatrix} \end{bmatrix}.
\end{align*}
\np
Define
$$
M_F \;:=\; \Bigl(F(H+m-1)\,\C[{\bf h}]\Bigr)^{\oplus \ell}\ \oplus\ \Bigl(F(H)\,\C[{\bf h}]\Bigr)^{\oplus \ell}
\ \subseteq\ \C[{\bf h}]^{\oplus 2\ell}.
$$
\np
The above computation shows that $M_F$ is a submodule of $M(A_1,\dots, A_m)$, for all $F \in \C[X]$. Moreover, fixing an infinite sequence $\{\lambda_r\}_{r\geq 1}$ of complex numbers and setting
$$
F_0(X):=1, \qquad F_k(X):=\prod_{r=1}^k (X-\lambda_r) ,
$$
\np
we obtain the  filtration
 $$ \dots\; \subsetneq M_k \subsetneq \;\dots\; \subsetneq M_2 \subsetneq M_1 \subsetneq M_0 = M(A_1,\dots, A_m),  $$
where 
  $$M_k := \Bigl(F_k(H+m-1)\C[{\bf{h}}]\Bigr)^{\oplus \ell}  
\;\bigoplus\; \Bigl(F_k(H)\C[{\bf{h}}]\Bigr)^{\oplus \ell}.$$

\np
{\bf (2)} Since $I$ and $J$ are $\pi(A_i)$-stable for all $i\in \mathbf m$, and
$\pi\left(A_i^{\comp}\right)=\pi(A_i)^{-1}$, they are also
$\pi\left(A_i^{\comp}\right)$-stable for all $i\in \mathbf m$.
Let $\{e_{-\ell},\dots,e_{-1},e_1,\dots,e_\ell\}$ be the standard
$\C[{\bf h}]$-basis of $M(A_1,\dots,A_m)$, and set
$$
M_I:=\bigoplus_{|i|\in I}\C[{\bf h}]\,e_i, \qquad M_J:=\bigoplus_{|j|\in J}\C[{\bf h}]\,e_j.
$$
\np
Then, $M_I$ and $M_J$ are nonzero $\mathcal U(\mathfrak{sl}(m|1))$-submodules and $M(A_1,\dots,A_m)=M_I\oplus M_J$. 
\end{proof}

\subsection{$\mathcal{U}(\mathfrak{h})$-duality}
In this subsection, we define the notion of $\mathcal{U}(\mathfrak h)$-duality on $\mathcal{M}_{\mathfrak{sl}(m|1)}(\ell|\ell)$ for $\ell\in \Z_{\geq 1}$. For $M\in\mathcal{M}_{\mathfrak{sl}(m|1)}(\ell|\ell)$, set
$$
M^\vee := \Hom_{\,\mathcal{U}(\mathfrak h)}\bigl(M,\mathcal{U}(\mathfrak h)\bigr) .
$$
Since $\mathcal{U}(\mathfrak h)$ is commutative, $M^\vee$ is a $\mathcal{U}(\mathfrak h)$-module.

\begin{lemma} \label{dualstructure}
    Let $\ell\in \Z_{\geq 1}$ and $M\in \mathcal{M}_{\mathfrak{sl}(m|1)}(\ell|\ell)$. For $i,j\in\mathbf{m}$ with $i\neq j$, $\omega\in M^{\vee}$, and $v\in M$, define
$$
  (e_{i\bar{1}}\cdot \omega)(v)
  := \Delta_i^{-1}\bigl(\,\omega(e_{\bar{1}i}\cdot v)\,\bigr), \qquad 
  (e_{\bar{1}i}\cdot \omega)(v)
  := \Delta_i\bigl(\,\omega(e_{i\bar{1}}\cdot v)\,\bigr),$$
 $$ (e_{ij}\cdot \omega)(v)
  := \sigma_i\sigma_j^{-1}\bigl(\,\omega(e_{ji}\cdot v)\,\bigr).
  $$
   Then these formulas endow $M^{\vee}$ with a $\mathcal{U}(\mathfrak{sl}(m|1))$-module structure compatible with the 
   $\mathcal{U}(\mathfrak h)$-module structure.
\end{lemma}
\begin{proof}Direct computation confirms that the commutation relations are satisfied. 
\end{proof}
\begin{proposition} \label{dualsl(m|1)}
     For each $i \in {\bf{m}}$, let $(A_i,A_i^{\comp})$ be a $h_i$-companion pair with
$A_i,A_i^{\comp}\in \Mat_\ell(\C[h_i])$. Then 
    $$M(A_1,\dots, A_m)^{\vee} \;\simeq \; M\left(\left(A^{\comp}_{1}\right)^{\mathsf T},\dots, \left(A^{\comp}_{m}\right)^{\mathsf T}\right).$$
\end{proposition}
\begin{proof}
Let $\left\{e_1,\dots,e_{2\ell}\right\}$ be a standard ordered
$\C[{\bf h}]$-basis of $M(A_1,\dots,A_m)$. For $v \in M(A_1,\dots,A_m)$, write
$$v = \bigl[f_1({\bf h})\;\;\dots\;\;f_{2\ell}({\bf h})\bigr]^{\mathsf T} =\sum_{i=1}^{2\ell} f_i({\bf{h}})e_i. $$
\np
Let $\{\varepsilon_1,\dots,\varepsilon_{2\ell}\}$ be the ordered $\C[\mathbf h]$-dual basis, so $\varepsilon_i(e_j)=\delta_{ij}$ for all $i,j$.
Then for any
$$
\omega \in M(A_1,\dots,A_m)^{\vee}
=\Hom_{\C[{\bf h}]}\big(M(A_1,\dots,A_m),\C[{\bf h}]\big),
$$
there exist unique $\omega_i({\bf h})\in \C[{\bf h}]$ such that $\omega=\bigl[\omega_1({\bf h})\;\;\dots\;\;\omega_{2\ell}({\bf h})\bigr]^{\mathsf T} = \sum_{i=1}^{2\ell} \omega_i({\bf h})\,\varepsilon_i$. Hence,
$$\omega(v) \;=\; \sum_{i=1}^{2\ell} \omega_i({\bf{h}}) f_i({\bf{h}}) \;=\; \omega^{\mathsf T}v . $$
From Lemma \ref{dualstructure}, for all $i \in \bf{m}$, we have
\begin{multline*}
   (e_{i\bar{1}}\cdot \omega)(v) = \left( \Delta^{-1}_i \circ \omega \circ e_{\bar{1}i} \right)(v) = \Delta_i^{-1}\left(\omega^{\mathsf T} \left[\begin{array}{ c | c }
    \mathbf 0 & \mathbf 0 \\
    \hline
    A_i^{\comp} & \mathbf 0
  \end{array}\right]\Delta_i(v)\right)\\
  =\left(\Delta_i^{-1}\left( \left[\begin{array}{ c | c }
    \mathbf 0 & \left(A_i^{\comp}\right)^{\mathsf T} \\
    \hline
    \mathbf 0 & \mathbf 0
  \end{array}\right]\right)\Delta_i^{-1} (\omega) \right)^{\mathsf T} (v)=\left( \left[\begin{array}{ c | c }
    \mathbf 0 & \left(A_i^{\comp}\right)^{\mathsf T} \\
    \hline
    \mathbf 0 & \mathbf 0
  \end{array}\right]\Delta_i^{-1} (\omega) \right)^{\mathsf T} (v), 
\end{multline*}

\begin{multline*}
     (e_{\bar{1}i}\cdot \omega)(v) = \left( \Delta_i \circ \omega \circ e_{i\bar{1}}\right)(v)
    =\Delta_i\left(\omega^{\mathsf T} \left[\begin{array}{ c | c }
    \mathbf 0 & A_i \\
    \hline
    \mathbf 0 & \mathbf 0
  \end{array}\right]\Delta^{-1}_i(v)\right)\\
  =\left(\Delta_i\left( \left[\begin{array}{ c | c }
    \mathbf 0 & \mathbf 0 \\
    \hline
    A_i^{\mathsf T} & \mathbf 0
  \end{array}\right]\right)\Delta_i(\omega) \right)^{\mathsf T} (v)
  =\left( \left[\begin{array}{ c | c }
    \mathbf 0 & \mathbf 0 \\
    \hline
    A_i^{\mathsf T} & \mathbf 0
  \end{array}\right]\Delta_i(\omega) \right)^{\mathsf T} (v).
\end{multline*}
\np
 The above computations imply 

$$e_{\bar{1}i}\cdot\omega = \left[\begin{array}{ c | c }
    \mathbf 0 & \left(A_i^{\comp}\right)^{\mathsf T} \\
    \hline
    \mathbf 0 & \mathbf 0
  \end{array}\right]\Delta_i^{-1}(\omega),\quad \text{and} \quad e_{\bar{1}i}\cdot\omega = \left[\begin{array}{ c | c }
    \mathbf 0 & \mathbf 0 \\
    \hline
    A_i^{\mathsf T} & \mathbf 0
  \end{array}\right]\Delta_i(\omega),$$
for all  $i \in {\bf{m}}$. Therefore, $M(A_1,\dots, A_m)^{\vee} \;\simeq \; M\left(\left(A^{\comp}_{1}\right)^{\mathsf T},\dots, \left(A^{\comp}_{m}\right)^{\mathsf T}\right).$
\end{proof}

\section{Exponential modules for $\mathcal{U}(\mathfrak{sl}(m|1))$ }
\np
This section is devoted to constructing and studying objects of $\mathcal{M}_{\mathfrak{sl}(m|1)}(\ell|\ell)$, which arise as a special case of \emph{exponential modules} over $\mathcal{U}(\mathfrak{sl}(m|1))$. From now on, for $n\in\Z_{\geq 1}$, we take $S_n$ to be the symmetric group on $\{0,1,\dots,n-1\}$, replacing the earlier convention $\{1,\dots,n\}$. We begin with some examples of ${\mathcal D}(m|1)$-modules.

\np
Let ${\boldsymbol{\mu}} := (\mu_1, \dots, \mu_m) \in \mathbb{C}^m$, and define ${\bf{x}}^{\boldsymbol{\mu}} := x_1^{\mu_1} \dots x_m^{\mu_m}$. Then the space
 $$\mathcal{F}({\boldsymbol{\mu}}):={\bf{x}}^{\boldsymbol{\mu}}\C\left[x_1^{\pm 1},\dots,x_m^{\pm 1}, \xi\right]={\bf{x}}^{\boldsymbol{\mu}}\C\left[x_1^{\pm 1},\dots,x_m^{\pm 1}\right]\; \oplus \; {\bf{x}}^{\boldsymbol{\mu}}\xi \C\left[x_1^{\pm 1},\dots,x_m^{\pm 1}\right], $$ 
 has a natural structure of ${\mathcal D} (m|1)$-module. That is, the action of ${\mathcal D} (m|1)$ on a basis vector of $\mathcal{F}({\boldsymbol{\mu}})$ is as follows
$$\partial_{x_i}\cdot {\bf x}^{\bf{a}} \xi^\varepsilon = a_i {\bf x}^{{\bf{a}} -\varepsilon_i}\xi^\varepsilon, \quad \partial_{\xi} \cdot {\bf x}^{\bf{a}} \xi^\varepsilon = \varepsilon{\bf x}^{\bf{a}} ,$$
 $$x_i \cdot {\bf x}^{\bf{a}} \xi^\varepsilon = {\bf x}^{{\bf{a}}+\varepsilon_i} \xi^\varepsilon, \quad \xi \cdot {\bf x}^{\bf{a}} \xi^\varepsilon = (1-\varepsilon){\bf x}^{\bf{a}} \xi,$$
where ${\bf{a}}:=(a_1,a_2,\dots, a_m) \in \C^m$ and $\varepsilon \in \{0,1\}$. 

\np
 \begin{remark}
The ${\mathcal{D}}(m|1)$-module $\mathcal{F}({\boldsymbol{\mu}})$ is a weight module (see Definition \ref{D(m|1)-weightmodule}), which was extensively studied in \cite{GPS}. In particular, $\mathcal{F}({\boldsymbol{\mu}})$ is indecomposable for all ${\boldsymbol{\mu}} \in \mathbb{C}^m$. Moreover, if $\mu_i \notin \mathbb{Z}$ for all $i \in {\bf{m}}$, then the module $\mathcal{F}({\boldsymbol{\mu}})$ is simple.
\end{remark}

\np
\begin{lemma} \label{D(m|1)-exp}
Let $g({\bf x}) \in \C[x_1,\dots, x_m]$ and $\varepsilon \in \{0,1\}$. The space $\C[x_1,\dots, x_m, \xi]\,e^{g({\bf x})}$ becomes a ${\mathcal D}(m|1)$-module via the following action:
$$ x_i \cdot {\bf x}^{\bf{k}} \xi^\varepsilon e^{g({\bf x})} = {\bf x}^{{\bf{k}}+\varepsilon_i} \xi^\varepsilon e^{g({\bf x})}, \quad \xi \cdot {\bf x}^{\bf{k}} \xi^\varepsilon e^{g({\bf x})} = (1-\varepsilon){\bf x}^{\bf{k}} \xi e^{g({\bf x})},\quad \partial_{\xi} \cdot {\bf x}^{\bf{k}} \xi^\varepsilon e^{g({\bf x})} = \varepsilon{\bf x}^{\bf{k}} e^{g({\bf x})}, $$
$$\partial_{x_i}\cdot {\bf x}^{\bf{k}} \xi^\varepsilon e^{g({\bf x})} = k_i {\bf x}^{{\bf{k}} -\varepsilon_i}\xi^\varepsilon e^{g({\bf x})} + \dfrac{\partial g({\bf x})}{\partial x_i}{\bf x}^{\bf{k}} \xi^\varepsilon e^{g({\bf x})}, $$
for all ${\bf{k}} \in \left(\Z_{\geq 0}\right)^m$.
\end{lemma}

\begin{proof}
     For $g({\bf x}) \in \C[x_1,\dots, x_m]$, we define the automorphism $\theta_{g({\bf x})}$ of $\mathcal{D}(m|1)$ via
    $$\theta_{g({\bf x})} (x_i) = x_i, \qquad \theta_{g({\bf x})}(\partial_{x_i}) = \partial_{x_i} +\dfrac{\partial g({\bf x})}{\partial x_i},\qquad \theta_{g({\bf x})} (\xi) = \xi, \qquad \theta_{g({\bf x})}(\partial_{\xi}) = \partial_{\xi},$$
    for all $i \in {\bf{m}}$. It is straightforward to verify that the $\mathcal{D}(m|1)$-action on $\C[x_1,\dots,x_m,\xi]e^{g(\mathbf x)}$ is the twist of the natural action on $\C[x_1,\dots,x_m,\xi]$ by the automorphism $\theta_{g(\mathbf x)}$.
\end{proof}

\begin{proposition}
Let $\mathcal{S}\subseteq {\bf m}$. Define $\Phi^{\mathcal S}_{m|1}:\mathcal{U}(\mathfrak{sl}(m|1))\to \mathcal{D}(m|1)$ on generators by
$$\Phi^{\mathcal S}_{m|1}(h_i)=
\begin{cases}
x_i\,\partial_{x_i}+\xi\,\partial_{\xi}, & i\notin \mathcal S,\\
-\,x_i\,\partial_{x_i}-1+\xi\,\partial_{\xi}, & i\in \mathcal S,
\end{cases} \qquad
\Phi^{\mathcal S}_{m|1}(e_{ij})=
\begin{cases}
x_i\,\partial_{x_j}, & i,j\notin \mathcal S,\\
\partial_{x_i}\,\partial_{x_j}, & i\in \mathcal S,\ j\notin \mathcal S,\\
-\,x_i x_j, & i\notin \mathcal S,\ j\in \mathcal S,\\
-\,x_j\,\partial_{x_i}, & i,j\in \mathcal S,
\end{cases}
$$
$$
\Phi^{\mathcal S}_{m|1}(e_{i\bar{1}})=
\begin{cases}
x_i\,\partial_{\xi}, & i\notin \mathcal S,\\
\partial_{x_i}\,\partial_{\xi}, & i\in \mathcal S,
\end{cases}
\qquad
\Phi^{\mathcal S}_{m|1}(e_{\bar{1}i})=
\begin{cases}
\xi\,\partial_{x_i}, & i\notin \mathcal S,\\
-\,x_i\,\xi, & i\in \mathcal S.
\end{cases}
$$
Then $\Phi^{\mathcal S}_{m|1}$ extends to an algebra homomorphism
$\Phi^{\mathcal S}_{m|1}:\mathcal{U}(\mathfrak{sl}(m|1))\to \mathcal{D}(m|1)$.
\end{proposition}

\begin{proof}
For all distinct $i,j \in \mathbf m$, a direct verification shows that the correspondence
$$e_{ij} \mapsto x_i\partial_{x_j}, \quad e_{i\bar{1}} \mapsto x_i\partial_{\xi}, \quad e_{\bar{1}i} \mapsto \xi\partial_{x_i}, \quad h_i \mapsto x_i\partial_{x_i}+\xi\partial_{\xi},$$
    extends to a homomorphism $\Phi_{m|1}: \mathcal{U}(sl(m|1)) \to \mathcal{D}(m|1) $. For $\mathcal{S} \subseteq {\bf{m}}$, we define the automorphism $\Psi_{\mathcal{S}}$ of $\mathcal{D}(m|1)$ as follows
$$\Psi_{\mathcal S}(\xi)=\xi,\quad \Psi_{\mathcal S}(\partial_{\xi})=\partial_{\xi},\qquad
\Psi_{\mathcal S}(x_i)=\begin{cases}  \partial_{x_i}, & i\in\mathcal S,\\ x_i, & i\notin\mathcal S,
\end{cases} \quad
\Psi_{\mathcal S}(\partial_{x_i})= \begin{cases}-\,x_i, & i\in\mathcal S,\\ \partial_{x_i}, & i\notin\mathcal S. \end{cases} $$
\np
Then, $\Phi^{\mathcal S}_{m|1}\;=\;\Psi_{\mathcal S}\circ \Phi_{m|1}. $
\end{proof} 

\begin{definition}[Exponential modules of $\mathcal{U}(\mathfrak{sl}(m|1))$]\label{generalexp-module}
For $\mathcal{S}\subseteq {\bf{m}}$, $g({\bf x}) \in \C[x_1,\dots, x_m]$, we write
$$
E\bigl(g({\bf x}),\mathcal{S}\bigr)\;:=\;\C[x_1,\dots,x_m,\xi]\,e^{g({\bf x})},
$$
to denote the $\mathcal{U}(\mathfrak{sl}(m|1))$-module defined via the homomorphism $\Phi^{\mathcal{S}}_{m|1}$.
\end{definition}
\np
\begin{remark}
    In the case where $g({\bf x}) = c \in \C$, the module $E(c, \mathcal{S})$ is a weight module.
\end{remark}
\np
\begin{lemma} \label{higherrank-exp}
    Let $\mathcal{S} \subseteq {\bf m}$, $\alpha, a_i \in \C^{\times}$, and $k_i \in \Z_{\geq 1}$ for all $ i\in {\bf{m}}$. Then,
    \begin{enumerate}
        \item $E\bigl(\sum_{i=1}^m a_i x_i^{k_i},\,\mathcal S\bigr) \in \mathcal{M}_{\mathfrak{sl}(m|1)}\bigl(\prod^m_{i=1}k_i|\prod^m_{i=1}k_i\bigr)$.
        \item For all $\boldsymbol{\ell}:= (\ell_1,\ell_2,\dots,\ell_m) \in (\Z_{\geq 0})^{m}$ and $m\geq 2$, $ E\bigl(\alpha{\bf{x}}^{\boldsymbol{\ell}},\mathcal{S}\bigr) \notin \mathcal{M}_{\mathfrak{sl}(m|1)}\bigl(p|p\bigr)  $ for all $p \in \Z_{\geq 1}$.
    \end{enumerate}
\end{lemma}

\begin{proof}
Throughout, put ${\bf b}:=(b_1,\dots,b_m)\in\Z_{\geq 0}^m$ and let $g(\mathbf x):=\sum_{i=1}^m a_i x_i^{k_i}$. 

\np
{\bf (1)}We address part (1) by first computing the explicit action of $h_j$. For $ j\notin \mathcal S$,
$$
\begin{aligned}
h_j \cdot {\bf x}^{\bf b} e^{g(\mathbf x)}
&=\bigl(x_j \partial_{x_j} + \xi \partial_{\xi}\bigr)\, {\bf x}^{\bf b} e^{g(\mathbf x)}
= b_j\,{\bf x}^{\bf b} e^{g(\mathbf x)} + k_j a_j\,{\bf x}^{{\bf b}+k_j\varepsilon_j} e^{g(\mathbf x)},\\
h_j \cdot {\bf x}^{\bf b}\xi\, e^{g(\mathbf x)}
&=\bigl(x_j \partial_{x_j} + \xi \partial_{\xi}\bigr)\, {\bf x}^{\bf b}\xi\, e^{g(\mathbf x)}
= (b_j+1)\,{\bf x}^{\bf b}\xi\, e^{g(\mathbf x)} + k_j a_j\,{\bf x}^{{\bf b}+k_j\varepsilon_j}\xi\, e^{g(\mathbf x)}.
\end{aligned}
$$
\np
For $ j\in \mathcal S$,
$$
\begin{aligned}
h_j \cdot {\bf x}^{\bf b} e^{g(\mathbf x)}
&=\bigl(-x_j \partial_{x_j} - 1 + \xi \partial_{\xi}\bigr)\, {\bf x}^{\bf b} e^{g(\mathbf x)}
= (-b_j-1)\,{\bf x}^{\bf b} e^{g(\mathbf x)} - k_j a_j\,{\bf x}^{{\bf b}+k_j\varepsilon_j} e^{g(\mathbf x)},\\
h_j \cdot {\bf x}^{\bf b}\xi\, e^{g(\mathbf x)}
&=\bigl(-x_j \partial_{x_j} - 1 + \xi \partial_{\xi}\bigr)\, {\bf x}^{\bf b}\xi\, e^{g(\mathbf x)}
= -b_j\,{\bf x}^{\bf b}\xi\, e^{g(\mathbf x)} - k_j a_j\,{\bf x}^{{\bf b}+k_j\varepsilon_j}\xi\, e^{g(\mathbf x)}.
\end{aligned}
$$

\np
Therefore, for $\varepsilon\in\{0,1\}$,
$$
\mathcal{U}(\mathfrak{h})\cdot {\bf x}^{\bf b}\,\xi^{\varepsilon} e^{g(\mathbf x)}
= {\bf x}^{\bf b}\,\xi^{\varepsilon}\,\C[x_1^{k_1},\dots,x_m^{k_m}]\,e^{g(\mathbf x)}.
$$
\np
Set $\;\mathfrak{R}(k_1,\dots,k_m):=\Z/k_1\Z \times \dots \times \Z/k_m\Z$. We have 
$$
\mathbf{x}^{\mathbf r}\,\xi^{\varepsilon}\,\C[x_1^{k_1},\dots,x_m^{k_m}]\,e^{g(\mathbf x)}
\;=\;
\Span\Big\{
\mathbf{x}^{\mathbf b}\,\xi^{\varepsilon}\,e^{g(\mathbf x)}
:\ \mathbf b\in\Z_{\geq 0}^{\,m},\ b_i\equiv r_i \!\!\!\pmod{k_i}\ \forall i
\Big\}.
$$
\np
Consequently,
$$E\left(\sum_{i=1}^m a_i x_i^{k_i},\,\mathcal S\right) = \bigoplus_{\varepsilon \in \{0,1\}}\bigoplus_{{\bf{r}} \in \mathfrak{R}(k_1,\dots,k_m)} {\bf x}^{{\bf{r}}}\xi^{\varepsilon}\C[x_1^{k_1},x_2^{k_2},\dots,x_m^{k_m}]e^{\sum_{i=1}^{m}a_ix_i^{k_i}},$$
as a vector space. Since $\mathcal{U}(\mathfrak{h})\cdot {\bf x}^{\bf b}\,\xi^{\varepsilon} e^{g(\mathbf x)}
= {\bf x}^{\bf b}\,\xi^{\varepsilon}\,\C[x_1^{k_1},\dots,x_m^{k_m}]\,e^{g(\mathbf x)}$ for all $\mathbf{b} \in \Z_{\geq 0}^{\,m}$, and $\varepsilon\in\{0,1\}$, the above is a direct-sum decomposition of $\mathcal U (\mathfrak h)$-modules. This completes the proof of part (1).

\np
{\bf (2)} To prove part (2), we make the following observation.
If $\ell_i=0$ for some $i\in{\bf m}$, then $\partial_{x_i}\left(e^{\alpha{\bf x}^{\boldsymbol{\ell}}}\right)=0$, and hence
$$
\begin{aligned}
&h_i \cdot {\bf x}^{\bf b} e^{\alpha{\bf x}^{\boldsymbol{\ell}}}
= \bigl(x_i \partial_{x_i} + \xi \partial_{\xi}\bigr)\,{\bf x}^{\bf b} e^{\alpha{\bf x}^{\boldsymbol{\ell}}}
= b_i\,{\bf x}^{\bf b} e^{\alpha{\bf x}^{\boldsymbol{\ell}}},\quad\text{for } i\notin\mathcal S\\
&h_i \cdot {\bf x}^{\bf b} e^{\alpha{\bf x}^{\boldsymbol{\ell}}}
= \bigl(-x_i \partial_{x_i} - 1 + \xi \partial_{\xi}\bigr)\,e^{\alpha{\bf x}^{\boldsymbol{\ell}}}
= (-b_i - 1)\,{\bf x}^{\bf b} e^{\alpha{\bf x}^{\boldsymbol{\ell}}},\quad \text{for } i\in\mathcal S.
\end{aligned}
$$
\np
Hence, $(h_i-b_i)(h_i+b_i+1)$ annihilates ${\bf x}^{\bf b} e^{\alpha{\bf x}^{\boldsymbol{\ell}}}$. Consequently,
$$
E\left(\alpha{\bf x}^{\boldsymbol{\ell}},\,\mathcal S\right)\notin 
\mathcal{M}_{\mathfrak{sl}(m|1)}(p|p)
\quad\text{for all}\quad p\in\Z_{\geq1}.
$$
\np
We now assume that $\boldsymbol{\ell} \in (\Z_{\geq 1})^{m}$. For  $j\notin\mathcal S$,
\begin{align*}
h_j \cdot {\bf x}^{\bf b} e^{\alpha{\bf x}^{\boldsymbol{\ell}}}
&=\bigl(x_j \partial_{x_j} + \xi \partial_{\xi}\bigr)\,{\bf x}^{\bf b} e^{\alpha{\bf x}^{\boldsymbol{\ell}}}
= b_j\,{\bf x}^{\bf b} e^{\alpha{\bf x}^{\boldsymbol{\ell}}} + \ell_j\alpha\,{\bf x}^{{\bf b}+\boldsymbol{\ell}} e^{\alpha{\bf x}^{\boldsymbol{\ell}}},\\
h_j \cdot {\bf x}^{\bf b}\xi\, e^{\alpha{\bf x}^{\boldsymbol{\ell}}}
&=\bigl(x_j \partial_{x_j} + \xi \partial_{\xi}\bigr)\,{\bf x}^{\bf b}\xi\, e^{\alpha{\bf x}^{\boldsymbol{\ell}}}
= (b_j+1)\,{\bf x}^{\bf b}\xi\, e^{\alpha{\bf x}^{\boldsymbol{\ell}}} + \ell_j\alpha\,{\bf x}^{{\bf b}+\boldsymbol{\ell}}\xi\, e^{\alpha{\bf{x}}^{\boldsymbol{\ell}}}.
\end{align*}
\np
For $j \in \mathcal S$,
  \begin{align*}
h_j \cdot {\bf x}^{\bf b} e^{\alpha{\bf x}^{\boldsymbol{\ell}}}
&=\bigl(-x_j \partial_{x_j} - 1 + \xi \partial_{\xi}\bigr)\,{\bf x}^{\bf b} e^{\alpha{\bf x}^{\boldsymbol{\ell}}}
= (-b_j-1)\,{\bf x}^{\bf b} e^{\alpha{\bf x}^{\boldsymbol{\ell}}} - \ell_j\alpha\,{\bf x}^{{\bf b}+\boldsymbol{\ell}} e^{\alpha{\bf x}^{\boldsymbol{\ell}}},\\
h_j \cdot {\bf x}^{\bf b}\xi\, e^{\alpha{\bf x}^{\boldsymbol{\ell}}}
&=\bigl(-x_j \partial_{x_j} - 1 + \xi \partial_{\xi}\bigr)\,{\bf x}^{\bf b}\xi\, e^{\alpha{\bf x}^{\boldsymbol{\ell}}}
= -b_j\,{\bf x}^{\bf b}\xi\, e^{\alpha{\bf x}^{\boldsymbol{\ell}}} - \ell_j\alpha\,{\bf x}^{{\bf b}+\boldsymbol{\ell}}\xi\, e^{\alpha{\bf x}^{\boldsymbol{\ell}}}.
\end{align*}
\np
Proceeding as in the proof of part~(1), $E\left(\alpha\,\mathbf{x}^{\boldsymbol{\ell}},\,\mathcal{S}\right)$ 
admits the following $\mathcal{U}(\mathfrak{h})$-module decomposition: 
$$E\left(\alpha{\bf x}^{\boldsymbol{\ell}},\,\mathcal S\right)= \bigoplus_{\varepsilon \in \{0,1\}}\bigoplus_{{\bf{r}} \in \mathfrak{P}(\ell_1,\dots,\ell_m)}{\bf x}^{{\bf{r}}}\xi^{\varepsilon} \C\left[{\bf x}^{{\boldsymbol{\ell}}}\right]e^{\alpha{\bf{x}}^{\boldsymbol{\ell}}},$$
where 
$$\mathfrak{P}(\ell_1,\dots,\ell_m) := \left\{{\bf{r}}= (r_1,\dots,r_m)\bigm|\ \text{there exists}~i\in {\bf m} ~\text{such that}~ 0 \leq r_i \leq \ell_i-1\right\}.$$ 
\np
Since $\card(\mathfrak{P}(\ell_1,\dots,\ell_m)) = \infty$, it follows that
$$
E\left(\alpha{\bf x}^{\boldsymbol{\ell}},\,\mathcal S\right)\notin 
\mathcal{M}_{\mathfrak{sl}(m|1)}(p|p)
\quad\text{for all}\quad p\in\Z_{\geq 1}. \qedhere
$$
\end{proof}

\np
Next we show that every module in $ \mathcal{M}_{\mathfrak{sl}(m|1)}\bigl(1|1\bigr)$ stated in Theorem~\ref{classisl(m|1)} is of the
form $E\bigl(a_1 x_1 + \dots + a_m x_m,\mathcal{S}\bigr)$ for appropriate $\mathbf{a} \in (\C^\times)^m$ and $\mathcal{S} \subseteq \mathbf m$.

\np
\begin{theorem}\label{classi-sl(m|1)asD(m|1)-mod}
Let $\mathcal{S} \subseteq \mathbf{m}$ and let $a_i \in \C^\times$ for $1 \leq i \leq m$.  Then
$$
E\bigl(a_1 x_1 + \dots + a_m x_m,\;\mathcal{S}\bigr)
\;\simeq\;
M\bigl(\mathbf{a}_{\mathcal{S}},\; \mathbf{m}\setminus\mathcal{S}\bigr),
$$
where $\mathbf{a}_{\mathcal{S}} := (\,(\mathbf{a}_{\mathcal{S}})_1,\dots,(\mathbf{a}_{\mathcal{S}})_m\,) \in (\C^\times)^m$ is defined by
$$
(\mathbf{a}_{\mathcal{S}})_i \;=\;
\begin{cases}
a_i,      & i \in \mathcal{S}, \\
a_i^{-1}, & i \notin \mathcal{S}.
\end{cases}
$$
\end{theorem}

\begin{proof}
For brevity, fix $\mathcal S\subseteq {\bf m}$ and denote the following falling/rising factorial notation:
$$
(x)^{\underline{k}}:=x(x-1)\dots(x-k+1),\qquad
(x)^{\overline{k}}:=x(x+1)\dots(x+k-1),
$$
where $k \in \Z_{\geq 0}$, and $(x)^{\underline{0}}=(x)^{\overline{0}}:=1$. For ${\bf k}=(k_1,\dots,k_m)\in\Z_{\ge 0}^m$, set
$$
u({\bf k},\mathcal S)
:= \prod_{i\notin \mathcal S}\left(\frac{1}{a_i^{k_i}}\,(h_i)^{\underline{k_i}}\right)
   \prod_{j\in \mathcal S}\left(\Bigl(-\frac{1}{a_j}\Bigr)^{k_j}\,(h_j+1)^{\overline{k_j}}\right),
$$
$$
v({\bf k},\mathcal S)
:= \prod_{i\notin \mathcal S}\left(\frac{1}{a_i^{k_i}}\,(h_i-1)^{\underline{k_i}}\right)
   \prod_{j\in \mathcal S}\left(\Bigl(-\frac{1}{a_j}\Bigr)^{k_j}\,h_j^{\overline{k_j}}\right).
$$
\np
Define
$$
C_1:=\bigl\{\,u({\bf k},\mathcal S)\bigm| {\bf k}\in\Z_{\geq 0}^m\,\bigr\},\qquad
C_2:=\bigl\{\,v({\bf k},\mathcal S)\bigm| {\bf k}\in\Z_{\geq 0}^m\,\bigr\}.
$$
\np
It is immediate that both $C_1$ and $C_2$ are $\C$-bases of $\mathcal U(\mathfrak h)$. We define the following linear map
   $$
\Theta:\ E\left(a_1x_1+\dots+a_m x_m,\mathcal S\right)\;\rightarrow\;
M\left({\bf a}_{\mathcal S},\,{\bf m}\setminus\mathcal S\right)
$$
by setting
$$
\Theta\left({\bf x}^{\bf k}\,e^{\sum_{i=1}^m a_i x_i}\right)=
\begin{bmatrix} u({\bf k},\mathcal S) \\ 0 \end{bmatrix},
\qquad
\Theta\left({\bf x}^{\bf k}\,\xi\,e^{\sum_{i=1}^m a_i x_i}\right)=
\begin{bmatrix} 0 \\ v({\bf k},\mathcal S) \end{bmatrix}.
$$
\np 
We verify that $\Theta$ is a $\mathcal{U}(\mathfrak{sl}(m|1))$-homomorphism by checking the following relations:

\noindent For $j\notin \mathcal S$,
\begin{align*}
\Theta\!\left(h_j\cdot \mathbf x^{\mathbf k} e^{\sum_{i=1}^m a_i x_i}\right)
&=\Theta\!\left((x_j\partial_{x_j}+\xi\partial_\xi)\bigl(\mathbf x^{\mathbf k} e^{\sum_{i=1}^m a_i x_i}\bigr)\right) \\
&= k_j\,\Theta\!\left(\mathbf x^{\mathbf k} e^{\sum_{i=1}^m a_i x_i}\right)
   + a_j\,\Theta\!\left(\mathbf x^{\mathbf k+\varepsilon_j} e^{\sum_{i=1}^m a_i x_i}\right) \\
&= k_j \begin{bmatrix} u(\mathbf k,\mathcal S) \\  0 \end{bmatrix}
   + a_j \begin{bmatrix} u(\mathbf k+\varepsilon_j,\mathcal S) \\  0 \end{bmatrix} \\
&= h_j \begin{bmatrix} u(\mathbf k,\mathcal S) \\  0 \end{bmatrix}
 = h_j\cdot \Theta\!\left(\mathbf x^{\mathbf k} e^{\sum_{i=1}^m a_i x_i}\right),
\end{align*}
\begin{align*}
 \Theta\!\left(h_j\cdot \mathbf x^{\mathbf k}\xi\, e^{\sum_{i=1}^m a_i x_i}\right)
&=\Theta\!\left((x_j\partial_{x_j}+\xi\partial_\xi)\bigl(\mathbf x^{\mathbf k}\xi\, e^{\sum_{i=1}^m a_i x_i}\bigr)\right) \\
&= (k_j+1)\,\Theta\!\left(\mathbf x^{\mathbf k}\xi\, e^{\sum_{i=1}^m a_i x_i}\right)
   + a_j\,\Theta\!\left(\mathbf x^{\mathbf k+\varepsilon_j}\xi\, e^{\sum_{i=1}^m a_i x_i}\right)\\
   &= (k_j+1)\begin{bmatrix}  0 \\ v(\mathbf k,\mathcal S) \end{bmatrix}
   + a_j \begin{bmatrix}  0 \\ v(\mathbf k+\varepsilon_j,\mathcal S) \end{bmatrix} \\
&= h_j \begin{bmatrix}  0 \\ v(\mathbf k,\mathcal S) \end{bmatrix}
 = h_j\cdot \Theta\!\left(\mathbf x^{\mathbf k}\xi\, e^{\sum_{i=1}^m a_i x_i}\right),\\
 \Theta\!\left(e_{j\bar 1}\cdot \mathbf x^{\mathbf k} e^{\sum a_i x_i}\right)
&=\Theta\!\left(x_j\partial_\xi(\mathbf x^{\mathbf k} e^{\sum a_i x_i})\right)=0
=e_{j\bar 1}\cdot \Theta\!\left(\mathbf x^{\mathbf k} e^{\sum a_i x_i}\right),\\
\Theta\!\left(e_{j\bar 1}\cdot \mathbf x^{\mathbf k}\xi\, e^{\sum a_i x_i}\right)
&=\Theta\!\left(\mathbf x^{\mathbf k+\varepsilon_j} e^{\sum a_i x_i}\right)
 =\begin{bmatrix} u(\mathbf k+\varepsilon_j,\mathcal S) \\  0 \end{bmatrix} \\
&=\begin{bmatrix} 0 & a_j^{-1}h_j \\  0 & 0 \end{bmatrix}
  \Delta_j^{-1}
  \left(\begin{bmatrix} \mathbf 0 \\ v(\mathbf k,\mathcal S) \end{bmatrix}\right)
 = e_{j\bar 1}\cdot \Theta\!\left(\mathbf x^{\mathbf k}\xi\, e^{\sum a_i x_i}\right),\\
\Theta\!\left(e_{\bar 1 j}\cdot \mathbf x^{\mathbf k} e^{\sum a_i x_i}\right)
&=\Theta\!\left(\xi\partial_{x_j}(\mathbf x^{\mathbf k} e^{\sum a_i x_i})\right) \\
&= k_j \begin{bmatrix}  0 \\ v(\mathbf k-\varepsilon_j,\mathcal S) \end{bmatrix}
  + a_j \begin{bmatrix}  0 \\ v(\mathbf k,\mathcal S) \end{bmatrix} \\
&=\begin{bmatrix}  0 &  0 \\ a_j &  0 \end{bmatrix}
  \Delta_j\left(
  \begin{bmatrix} u(\mathbf k,\mathcal S) \\  0 \end{bmatrix}\right)
 = e_{\bar 1 j}\cdot \Theta\!\left(\mathbf x^{\mathbf k} e^{\sum a_i x_i}\right),\\
\Theta\!\left(e_{\bar 1 j}\cdot \mathbf x^{\mathbf k}\xi\, e^{\sum a_i x_i}\right)
&=\Theta\!\left(\xi\partial_{x_j}(\mathbf x^{\mathbf k}\xi\, e^{\sum a_i x_i})\right)=0
= e_{\bar 1 j}\cdot \Theta\!\left(\mathbf x^{\mathbf k}\xi\, e^{\sum a_i x_i}\right).
\end{align*}

\noindent
For $j\in\mathcal S$,
\begin{align*}
\Theta\!\left(h_j\cdot \mathbf x^{\mathbf k} e^{\sum a_i x_i}\right)
&=\Theta\!\left((-x_j\partial_{x_j}-1+\xi\partial_\xi)\bigl(\mathbf x^{\mathbf k} e^{\sum a_i x_i}\bigr)\right) \\
&=(-k_j-1)\begin{bmatrix} u(\mathbf k,\mathcal S) \\  0 \end{bmatrix}
   - a_j \begin{bmatrix} u(\mathbf k+\varepsilon_j,\mathcal S) \\  0 \end{bmatrix} \\
&= h_j \begin{bmatrix} u(\mathbf k,\mathcal S) \\  0 \end{bmatrix}
 = h_j\cdot \Theta\!\left(\mathbf x^{\mathbf k} e^{\sum a_i x_i}\right),\\
\Theta\!\left(h_j\cdot \mathbf x^{\mathbf k}\xi\, e^{\sum a_i x_i}\right)
&=\Theta\!\left((-x_j\partial_{x_j}-1+\xi\partial_\xi)\bigl(\mathbf x^{\mathbf k}\xi\, e^{\sum a_i x_i}\bigr)\right) \\
&= -k_j \begin{bmatrix}  0 \\ v(\mathbf k,\mathcal S) \end{bmatrix}
   - a_j \begin{bmatrix}  0 \\ v(\mathbf k+\varepsilon_j,\mathcal S) \end{bmatrix} \\
&= h_j \begin{bmatrix}  0 \\v(\mathbf k,\mathcal S) \end{bmatrix}
 = h_j\cdot \Theta\!\left(\mathbf x^{\mathbf k}\xi\, e^{\sum a_i x_i}\right),\\
\Theta\!\left(e_{j\bar 1}\cdot \mathbf x^{\mathbf k} e^{\sum a_i x_i}\right)
&=\Theta\!\left(\partial_{x_j}\partial_\xi(\mathbf x^{\mathbf k} e^{\sum a_i x_i})\right)=0
= e_{j\bar 1}\cdot \Theta\!\left(\mathbf x^{\mathbf k} e^{\sum a_i x_i}\right),\\
\Theta\!\left(e_{j\bar 1}\cdot \mathbf x^{\mathbf k}\xi\, e^{\sum a_i x_i}\right)
&=\Theta\!\left(k_j \mathbf x^{\mathbf k-\varepsilon_j} e^{\sum a_i x_i}
                 + a_j \mathbf x^{\mathbf k} e^{\sum a_i x_i}\right) \\
&= k_j \begin{bmatrix} u(\mathbf k-\varepsilon_j,\mathcal S) \\  0 \end{bmatrix}
  + a_j \begin{bmatrix} u(\mathbf k,\mathcal S) \\  0 \end{bmatrix} \\
&=\begin{bmatrix}  0 & a_j \\  0 &  0 \end{bmatrix}
  \Delta_j^{-1}
  \left(\begin{bmatrix}  0 \\ v(\mathbf k,\mathcal S) \end{bmatrix}\right)
 = e_{j\bar 1}\cdot \Theta\!\left(\mathbf x^{\mathbf k}\xi\, e^{\sum a_i x_i}\right),\\
 \end{align*}
\begin{align*}
\Theta\!\left(e_{\bar 1 j}\cdot \mathbf x^{\mathbf k} e^{\sum a_i x_i}\right)
&=\Theta\!\left(-x_j\xi\,\mathbf x^{\mathbf k} e^{\sum a_i x_i}\right)
 = -\begin{bmatrix}  0 \\ v(\mathbf k+\varepsilon_j,\mathcal S) \end{bmatrix} \\
&=\begin{bmatrix}  0 &  0 \\ a_j^{-1}h_j &  0 \end{bmatrix}
  \Delta_j \left(\begin{bmatrix} u(\mathbf k,\mathcal S) \\  0 \end{bmatrix}\right)
 = e_{\bar 1 j}\cdot \Theta\!\left(\mathbf x^{\mathbf k} e^{\sum a_i x_i}\right),\\
\Theta\!\left(e_{\bar 1 j}\cdot \mathbf x^{\mathbf k}\xi\, e^{\sum a_i x_i}\right)
&=\Theta\!\left(-x_j\xi\,\mathbf x^{\mathbf k}\xi\, e^{\sum a_i x_i}\right)=0
= e_{\bar 1 j}\cdot \Theta\!\left(\mathbf x^{\mathbf k}\xi\, e^{\sum a_i x_i}\right).
\end{align*}

\np
Finally, $\Theta$ is an isomorphism, as it maps each basis element to its corresponding basis element. 
\end{proof}

\begin{theorem} \label{realizationofsl(m|1)}
    Let $\mathcal S\subseteq {\bf m}$ and let $a_i\in \C^\times$, $k_i\in \Z_{\ge 1}$ for all $i\in {\bf m}$. Then
$$
E\!\left(\sum_{i=1}^m a_i x_i^{k_i},\,\mathcal S\right)\;\simeq\; M(A_1,\dots,A_m),
$$
for some $A_i \in \Mat_{\prod_{j=1}^m k_j}\bigl(\C[h_i]\bigr)$.
\end{theorem}

\begin{proof} 
For brevity, set $K:=\prod_{i=1}^m k_i$. Recall from Lemma \ref{higherrank-exp} that
$$
\mathfrak{R}(k_1,\dots,k_m)\;:=\;\mathbb{Z}/k_1\mathbb{Z}\times\dots\times \mathbb{Z}/k_m\mathbb{Z}.
$$
\np
As a $\mathcal{U}(\mathfrak h)$-module,
$$E\!\left(\sum_{i=1}^m a_i x_i^{k_i},\,\mathcal S\right) = \bigoplus_{\varepsilon \in \{0,1\}}\bigoplus_{{\bf{r}} \in \mathfrak{R}(k_1,\dots,k_m)} {\bf x}^{{\bf{r}}}\xi^{\varepsilon}\C[x_1^{k_1},x_2^{k_2},\dots,x_m^{k_m}]e^{\sum_{i=1}^{m}a_ix_i^{k_i}}.$$
\np
 Since for any $\varepsilon \in \{0,1\}$ and ${\bf{b}}\in \mathfrak{R}(k_1,\dots,k_m)$,
$$\mathcal{U}(\mathfrak{h}) \cdot {\bf x}^{\bf{b}}\xi^{\varepsilon}e^{\sum_{i=1}^{m}a_ix_i^{k_i}}= {\bf x}^{\bf{b}}\xi^{\varepsilon} \C[x_1^{k_1},x_2^{k_2},\dots,x_m^{k_m}]e^{\sum_{i=1}^{m}a_ix_i^{k_i}},$$
we can rewrite the decomposition as
$$E\!\left(\sum_{i=1}^m a_i x_i^{k_i},\,\mathcal S\right) = \bigoplus_{\varepsilon \in \{0,1\}}\bigoplus_{{\bf{r}} \in \mathfrak{R}(k_1,\dots,k_m)} \mathcal{U}(\mathfrak{h}) \cdot{\bf x}^{{\bf{r}}}\xi^{\varepsilon}e^{\sum_{i=1}^{m}a_ix_i^{k_i}}.$$
\np
For $j\in{\bf m}$, ${\bf r}\in\mathfrak{R}(k_1,\dots,k_m)$, $\varepsilon\in\{0,1\}$, and $g({\bf h})\in\mathcal U(\mathfrak h)$,
\begin{eqnarray*}
    e_{j\bar{1}} \cdot \left(g({\bf{h}})\cdot {\bf x}^{{\bf{r}}}\xi^{\varepsilon}e^{\sum_{i=1}^{m}a_ix_i^{k_i}}\right) &=& \Delta_j^{-1}(g({\bf{h}}))\cdot \left(e_{j\bar{1}} \cdot {\bf x}^{{\bf{r}}}\xi^{\varepsilon}e^{\sum_{i=1}^{m}a_ix_i^{k_i}}\right).
\end{eqnarray*}
\np
It is immediate that, for every ${\bf r}\in \mathfrak{R}(k_1,\dots,k_m)$ and every $g({\bf h})\in\C[{\bf h}]$,
$$
e_{j\bar{1}}\cdot\left(g({\bf h})\,{\bf x}^{\bf r} e^{\sum_{i=1}^{m} a_i x_i^{k_i}}\right)=0.
$$
\np
We now compute the action of $e_{j\bar{1}}$ on $\left(g({\bf{h}})\cdot {\bf x}^{{\bf{r}}}\xi e^{\sum_{i=1}^{m}a_ix_i^{k_i}}\right)$; it is given by the following
case-by-case formulas (according to whether $j\in\mathcal S$ or $j\notin\mathcal S$).

\np
\textbf{Case 1:} If $j\notin \mathcal{S}$,
\begin{eqnarray*}
    e_{j\bar{1}} \cdot \left(g({\bf{h}})\cdot {\bf x}^{{\bf{r}}}\xi e^{\sum_{i=1}^{m}a_ix_i^{k_i}}\right) &=& \Delta_j^{-1}(g({\bf{h}})) \cdot \left(x_j\partial_{\xi} \left({\bf x}^{{\bf{r}}}\xi e^{\sum_{i=1}^{m}a_ix_i^{k_i}}\right)\right)\\
    &=& \Delta_j^{-1}(g({\bf{h}})) \cdot \left({\bf x}^{{\bf{r}}+\varepsilon_j}e^{\sum_{i=1}^{m}a_ix_i^{k_i}}\right).
\end{eqnarray*}
\np
Note that
$$x_j^{k_j} \left(\prod_{\ell\in {\bf m}\setminus\{j\}} x_\ell^{r_\ell}\right)e^{\sum_{i=1}^{m}a_ix_i^{k_i}} = \dfrac{1}{a_jk_j}h_j \cdot \left(\left(\prod_{\ell\in {\bf m}\setminus\{j\}} x_\ell^{r_\ell}\right)e^{\sum_{i=1}^{m}a_ix_i^{k_i}}\right). $$
\np
Therefore,
\begin{equation}\label{importanteq1}
e_{j\bar{1}}\cdot\bigl(g({\bf h})\,{\bf x}^{\bf r}\,\xi\,e^{\sum_{i=1}^{m} a_i x_i^{k_i}}\bigr)
=
\begin{cases}
\Delta_j^{-1}\!\bigl(g({\bf h})\bigr)\cdot{\bf x}^{{\bf r}+\varepsilon_j}\,e^{\sum_{i=1}^{m} a_i x_i^{k_i}}, \;\;\;\;\;\;\;\ 0\leq r_j \leq k_j-2,\\
\frac{1}{a_j k_j}\,h_j\,\Delta_j^{-1}\!\bigl(g({\bf h})\bigr)\cdot
\left(\displaystyle\prod_{\ell\in {\bf m}\setminus\{j\}} x_\ell^{r_\ell}\right)\,e^{\sum_{i=1}^{m} a_i x_i^{k_i}},\;  r_j = k_j-1.
\end{cases}
\end{equation}
\np
\textbf{Case 2:} If $j\in \mathcal{S}$, 
\begin{eqnarray*}
    e_{j\bar{1}} \cdot \left(g({\bf{h}})\cdot {\bf x}^{{\bf{r}}}\xi e^{\sum_{i=1}^{m}a_ix_i^{k_i}}\right) &=& \Delta_j^{-1}(g({\bf{h}})) \cdot \left(\partial_{x_j}\partial_{\xi} \left({\bf x}^{{\bf{r}}}\xi e^{\sum_{i=1}^{m}a_ix_i^{k_i}}\right)\right)\\
    &=& \Delta_i^{-1}(g({\bf{h}})) \cdot \left( r_j {\bf x}^{{\bf{r}}-\varepsilon_j} + a_jk_j {\bf x}^{{\bf{r}}+(k_j-1)\varepsilon_j} \right)e^{\sum_{i=1}^{m}a_ix_i^{k_i}}.
\end{eqnarray*}
\np
We observe that
\begin{align*}
h_j \cdot x_j^{r_j-1}\left(\prod_{\ell\in {\bf m}\setminus\{j\}} x_\ell^{r_\ell}\right)e^{\sum_{i=1}^{m}a_ix_i^{k_i}} &= \left(-x_j \partial_{x_j} -1 + \xi \partial_{\xi}\right)\left(  x_j^{r_j-1}\left(\prod_{\ell\in {\bf m}\setminus\{j\}} x_\ell^{r_\ell}\right)e^{\sum_{i=1}^{m}a_ix_i^{k_i}}\right)\\
&= \left(-r_j x_j^{r_j-1} -a_jk_jx_j^{r_j -1 +k_j}\right)\left(\prod_{\ell\in {\bf m}\setminus\{j\}} x_\ell^{r_\ell}\right)e^{\sum_{i=1}^{m}a_ix_i^{k_i}}.
\end{align*}
\np
Therefore,
\begin{equation}\label{importanteq2}
e_{j\bar{1}}\cdot\bigl(g({\bf h})\,{\bf x}^{\bf r}\,\xi\,e^{\sum_{i=1}^{m} a_i x_i^{k_i}}\bigr)
=
\begin{cases}
a_j k_j\,\Delta_j^{-1}\!\bigl(g({\bf h})\bigr)\cdot\,
{\bf x}^{{\bf r}+(k_j-1)\varepsilon_j}\,e^{\sum_{i=1}^{m} a_i x_i^{k_i}},\;\;  r_j=0,\\
-\,h_j\,\Delta_j^{-1}\!\bigl(g({\bf h})\bigr)\cdot
x_j^{\,r_j-1}\!\left(\displaystyle\prod_{\ell\in {\bf m}\setminus\{j\}} x_\ell^{r_\ell}\right)
e^{\sum_{i=1}^{m} a_i x_i^{k_i}},\  1\leq r_j\leq k_j-1.
\end{cases}
\end{equation}
\np
We equip $\mathfrak R(k_1,\dots,k_m)$ with the lexicographic order
$\preccurlyeq_{\mathfrak R(k_1,\dots,k_m)}$ defined as follows: for
$\mathbf r=(r_1,\dots,r_m)$ and $\mathbf r'=(r'_1,\dots,r'_m)$,
$$
\mathbf r \preccurlyeq_{\mathfrak R(k_1,\dots,k_m)} \mathbf r'
\iff
\mathbf r=\mathbf r'\;\;\; \text{or}\;\;\; \exists\;t\in \mathbf m\; \text{such that } 
r_s=r'_s\ (1\leq s<t)\ \text{and}\ r_t<r'_t.
$$
\np
With respect to the lexicographic order $\preccurlyeq_{\mathfrak R(k_1,\dots,k_m)}$ on $\mathfrak R(k_1,\dots,k_m)$, define a total order $\preccurlyeq$ on the set
$$\mathcal B \;:=\;\left\{V_{\mathbf r,\varepsilon}:= \mathbf x^{\mathbf r}\,\xi^{\varepsilon}\, e^{\sum_{i=1}^m a_i x_i^{k_i}}\bigm|  \mathbf r\in \mathfrak R(k_1,\dots,k_m),\ \varepsilon\in\{0,1\}\right\},$$
by
$$
V_{\mathbf r,\varepsilon}\ \preccurlyeq\ V_{\mathbf r',\varepsilon'}
\iff
\varepsilon<\varepsilon'\qquad \text{ or }\qquad \varepsilon=\varepsilon'\; \text{and}\; \mathbf r \preccurlyeq_{\mathfrak R(k_1,\dots,k_m)} \mathbf r'.
$$
In particular, all $V_{\mathbf r,0}$ precede all $V_{\mathbf r,1}$, and within each parity the order agrees with $\preccurlyeq_{\mathfrak R(k_1,\dots,k_m)}$.

\np
Using the the order $\preccurlyeq$ on the $\mathcal U(\mathfrak h)$-generating ordered set $\mathcal B$, the action of $e_{i\bar{1}}$ on $E\big(\sum_{i=1}^m a_i x_i^{k_i},\,\mathcal S\big)$ is given by:
$$e_{i\bar{1}} \cdot \begin{bmatrix}f_1({\bf{h}} )\\\vdots\\ f_{2K}({\bf{h}})\end{bmatrix} = \left[\begin{array}{ c | c }
    \mathbf 0 & A_{i} \\
    \hline
    \mathbf 0 & \mathbf 0
  \end{array}\right] \Delta^{-1}_i\left(\begin{bmatrix}f_1({\bf{h}})\\\vdots \\f_{2K}({\bf{h}}) \end{bmatrix} \right),\quad \text{for all}\quad i\in {\bf{m}},$$
where $A_i\in\Mat_{K}\bigl(\C[{\bf h}]\bigr)$, and the vector has been ordered so that its first $K$ entries correspond to the subspaces with $\xi$–degree $0$ and the last $K$ entries to those with $\xi$–degree $1$. Moreover, Formulas \eqref{importanteq1}–\eqref{importanteq2} imply that each column of $A_i$ has precisely one nonzero entry, and every nonzero entry is either $\alpha$ or $\alpha\,h_i$ with $\alpha\in\C^\times$.

\np
    By a similar argument, the action of $e_{\bar{1}i}$ on $E\big(\sum_{i=1}^m a_i x_i^{k_i},\,\mathcal S\big)$ can be written with respects to the order $\preccurlyeq$ on the $\mathcal U(\mathfrak h)$-generating ordered set $\mathcal B$ as follows:
  $$e_{\bar{1}i} \cdot \begin{bmatrix}f_1({\bf{h}} )\\\vdots\\ f_{2K}({\bf{h}})\end{bmatrix} = \left[\begin{array}{ c | c }
    \mathbf 0 & \mathbf 0 \\
    \hline
    B_{i} & \mathbf 0
  \end{array}\right] \Delta_i\left(\begin{bmatrix}f_1({\bf{h}})\\\vdots \\f_{2K}({\bf{h}}) \end{bmatrix} \right), \quad \text{for all}\quad i\in {\bf{m}}, $$
   where $B_{i} \in \Mat_K(\C[{\bf{h}}])$ with the property that each column of $B_{i}$ contains exactly one nonzero entry. Each nonzero entry of $B_i$ is either $\beta$ or $\beta h_i$ with $\beta \in \C^{\times}$.

   \np
   Since $e_{i\bar{1}} \cdot e_{\bar{1}i} \cdot \begin{bmatrix}f_1({\bf{h}} )\\\vdots\\ f_{2K}({\bf{h}})\end{bmatrix} + e_{\bar{1}i} \cdot e_{i\bar{1}} \cdot \begin{bmatrix}f_1({\bf{h}} )\\\vdots\\ f_{2K}({\bf{h}})\end{bmatrix} = \begin{bmatrix}h_if_1({\bf{h}} )\\\vdots\\ h_if_{2K}({\bf{h}})\end{bmatrix} $, we deduce that
   $$A_i \Delta_i^{-1} (B_i) = A_i B_i = h_i \I_K\quad \text{and}\quad B_i \Delta_i(A_i) = B_i A_i =  h_i \I_K.$$
   Hence, $A_i, B_i$ are GPMs, and $(A_i, B_i)$ is an $h_i$-companion pair for every $i\in{\bf m}$. This completes the proof.
\end{proof}

\section{The $M(A_1,\dots,A_m)$-realization of $E\!\left(\sum_{i=1}^m a_i x_i^{k_i},\,\mathcal S\right)$ and its properties}

\np 
In this section, we describe explicitly the matrices $A_i$ from Theorem \ref{realizationofsl(m|1)}. We start with some notation.

\np
For $i,j \in \Z_{\geq 1}$ and $\alpha\in\C^\times$, fix an index $\ell$ and define
$$
U_{(h_\ell,\alpha)}(i,j):=\begin{bmatrix}
\mathbf 0 & \dfrac{1}{\alpha} h_\ell \I_{j} \\
\I_{i} & \mathbf 0
\end{bmatrix},\qquad
V_{(h_\ell,\alpha)}(i,j):=\begin{bmatrix}
\mathbf 0 & -h_\ell \I_{j} \\
\alpha \I_{i} & \mathbf 0
\end{bmatrix}\in\Mat_{i+j}(\C[h_\ell]),
$$
where $\mathbf 0$ denotes zero blocks of compatible sizes. In particular, we have the following
$$U_{(h_i,\alpha)}\left(k-1,1\right):=\begin{bmatrix}
 0      &  0      & \dots &  0      & \dfrac{1}{\alpha}h_i \\
1      &  0      & \dots &  0      & 0              \\
 0      & 1      & \ddots & \vdots & \vdots         \\
\vdots & \vdots & \ddots &  0      &  0              \\
 0      & 0      & \dots & 1      & 0
\end{bmatrix} \in \Mat_{k}(\C[h_i]),$$
$$V_{(h_i,\alpha)}\left(1,k-1\right):=\begin{bmatrix}
0      & -h_i     & 0      & \dots & 0     \\
0      & 0      & -h_i     & \dots & 0     \\
\vdots & \vdots & \ddots & \ddots & \vdots \\
0      & 0      & \dots & 0      & -h_i    \\
\alpha     & 0      & \dots & 0      & 0
\end{bmatrix}\in \Mat_{k}(\C[h_i]).$$

\subsection{The $\mathcal{U}(\mathfrak{sl}(1|1))$-module $E(ax_1^k, \mathcal{S})$}

\begin{proposition}
Let $ a \in \C^{\times} $. Then:
$$
E\left(ax_1, \varnothing\right) \,\simeq \, M(h_1), \qquad E\left(ax_1, \{1\}\right)\, \simeq \, M(1).
$$
Moreover, for any $ k \in \Z_{> 1} $, we have:
\begin{enumerate}
    \item $ E(ax_1^k, \varnothing) \,\simeq\, M\left(U_{(h_1,ak)}(k-1,1)\right) \,\simeq\, M(h_1) \,\oplus \, \bigoplus_{i=1}^{k-1} M(1) $,
    \item $ E(ax_1^k, \{1\}) \,\simeq\, M\left(V_{(h_1,ak)}(1,k-1)\right) \,\simeq\, M(1)\, \oplus\, \bigoplus_{i=1}^{k-1} M(h_1) $.
\end{enumerate}
\end{proposition}

\begin{proof}
  {\bf (1)}  We prove that $E\left(ax_1, \varnothing\right) \,\simeq\, M(h_1)$; the case $E\bigl(ax_1,\{1\}\bigr)\;\simeq\; M(1)$ is analogous. For $\ell \in \Z_{\geq 0},$
    $$ h_1 \cdot x_1^\ell e^{ax_1} = (x_1\partial_{x_1} + \xi \partial_{\xi})x_1^\ell e^{ax_1} = \ell x_1^\ell e^{ax_1} + ax_{1}^{\ell+1}e^{ax_1},$$
    $$ h_1 \cdot x_1^\ell \xi e^{ax_1} = (x_1\partial_{x_1} + \xi \partial_{\xi})x_1^\ell \xi e^{ax_1} = (\ell+1)x_1^\ell \xi e^{ax_1} + ax_1^{\ell+1} \xi e^{ax_1}.$$
    \np
Hence, as a $\mathcal{U}(\mathfrak h)$-module,
$E(ax_1,\varnothing)
=\C[h_1]\cdot\,e^{a x_1} \oplus \C[h_1]\cdot\,\xi e^{a x_1}$. For $g_1(h_1) , g_2(h_1) \in \C[h_1] $,
\begin{align*}
e_{1\bar1}\cdot\bigl(g_1(h_1)e^{a x_1}+g_2(h_1)\,\xi e^{a x_1}\bigr)
&= g_2(h_1)\cdot\,\left(x_1 e^{a x_1}\right)
= \frac{1}{a}\,h_1\,g_2(h_1)\cdot\,e^{a x_1},\\
e_{\bar1 1}\cdot\bigl(g_1(h_1)e^{a x_1}+g_2(h_1)\,\xi e^{a x_1}\bigr)
&= g_1(h_1)\cdot\,\big(\xi\,\partial_{x_1}(e^{a x_1})\big)
= a\,g_1(h_1)\cdot\,\xi e^{a x_1},
\end{align*}
\np
  With respect to the ordered $\mathcal U(\mathfrak h)$-basis $\{e^{ax_1},\,\xi e^{ax_1}\}$, for all $g_1,g_2\in\C[h_1]$,
$$
e_{1\bar{1}}\cdot
\begin{bmatrix} g_1(h_1)\\ g_2(h_1)\end{bmatrix}
=
\begin{bmatrix} 0 & \frac{1}{a}\,h_1\\ 0&0\end{bmatrix}
\begin{bmatrix} g_1(h_1)\\ g_2(h_1)\end{bmatrix},
\qquad
e_{\bar{1}1}\cdot
\begin{bmatrix} g_1(h_1)\\ g_2(h_1)\end{bmatrix}
=
\begin{bmatrix} 0&0\\ a&0\end{bmatrix}
\begin{bmatrix} g_1(h_1)\\ g_2(h_1)\end{bmatrix}.
$$
\np
Hence, $E(ax_1,\varnothing)\,\simeq \, M\left(\frac{1}{a}\,h_1\right)$. By Lemma \ref{simplified-sl(1|1)} part 1,  $E(ax_1, \varnothing) \,\simeq\, M\left( h_1 \right)$. 

\np
{\bf (2)} To prove the remaining parts of the proposition, we regard
$$
M\!\left(U_{(h_1,ak)}(k-1,1)\right)
\quad\text{and}\quad
M\!\left(V_{(h_1,ak)}(1,k-1)\right)
$$
as modules with the common underlying $\Z_2$-graded vector space
$$
\C[h_1]^{\oplus 2k}:=\C[h_1]^{\oplus k}
\;\oplus\;
\C[h_1]^{\oplus k}.
$$
\np
To establish the first isomorphisms in parts~(1) and~(2), we define the following maps.
\begin{align*}
  \Phi_{\varnothing}:  E(ax_1^k, \varnothing) &\to M\left(U_{(h_1,ak)}\bigl(k-1,1\bigr)\right)\\
x_1^{\ell k+p} e^{ax_1^k} &\mapsto \left(\dfrac{1}{(ak)^\ell} \prod_{j=0}^{\ell-1}(h_1-p-kj) \right) e_{p+1},\\
x_1^{\ell k+p}\xi e^{ax_1^k} &\mapsto \left(\dfrac{1}{(ak)^\ell} \prod_{j=0}^{\ell-1}(h_1-p-1-kj) \right) e_{k+p+1};
\end{align*}
\begin{align*}
  \Phi_{\{1\}}:  E(ax_1^k, \{1\}) &\to M\left(V_{(h_1,ak)}\bigl(1,k-1\bigr)\right)\\
x_1^{\ell k+p} e^{ax_1^k} &\mapsto \left(\dfrac{(-1)^\ell}{(ak)^\ell} \prod_{j=0}^{\ell-1}(h_1+p+1+jk)\right) e_{p+1},\\
x_1^{\ell k+p}\xi e^{ax_1^k} &\mapsto \left(\dfrac{(-1)^\ell}{(ak)^\ell} \prod_{j=0}^{\ell-1}(h_1+p+jk)\right) e_{k+p+1};
\end{align*}
where $\ell\in\mathbb{Z}_{\geq 0}$, $0\leq p\leq k-1$. A direct computation shows that both $\Phi_{\varnothing}$ and $\Phi_{\{1\}}$ are isomorphisms of $\mathcal{U}(\mathfrak{sl}(1|1))$-modules.

\np
The second isomorphism in part~(1) follows from the direct-sum decomposition
$$
M\!\left(U_{(h_1,ak)}(k-1,1)\right)
\;\simeq\;
\bigl(\C[h_1]e_1 \oplus \C[h_1]e_{2k}\bigr)
\;\oplus\;
\bigoplus_{i=2}^{k}\bigl(\C[h_1]e_i \oplus \C[h_1]e_{\,k+i-1}\bigr),
$$
where each two-generator summand is stable under the $\mathcal{U}(\mathfrak{sl}(1|1))$-action. Moreover,
$$
\C[h_1]e_1 \oplus \C[h_1]e_{2k}\;\simeq\; M\!\left(\frac{1}{ak}\,h_1\right)\;\simeq\; M(h_1),
$$
and, for every $2\leq i\leq k$,
$$
\C[h_1]e_i \oplus \C[h_1]e_{\,k+i-1}\;\simeq\; M(1).
$$

\np
Similarly, the second isomorphism in part~(2) follows from the decomposition
$$
M\!\left(V_{(h_1,ak)}(1,k-1)\right)
\;\simeq\;
\bigl(\C[h_1]e_k \oplus \C[h_1]e_{k+1}\bigr)
\;\oplus\;
\bigoplus_{i=1}^{k-1}\bigl(\C[h_1]e_i \oplus \C[h_1]e_{\,k+i+1}\bigr),
$$
where each two–generator summand is stable under the $\mathcal{U}(\mathfrak{sl}(1|1))$-action. Moreover,
$$
\C[h_1]e_k \oplus \C[h_1]e_{k+1}\;\simeq\; M(ak)\;\simeq\; M(1),
$$
and, for every $1\leq i\leq k-1$,
$$
\C[h_1]e_i \oplus \C[h_1]e_{\,k+i+1}\;\simeq\; M(-h_1)\;\simeq\; M(h_1).
$$
\end{proof}

\begin{corollary}
Let $a,b\in \C^\times$ and $k\in \Z_{\geq 1}$. Then
$$
E\!\left(ax_1^{k},\,\varnothing\right)\;\simeq\;
E\!\left(bx_1^{k},\,\{1\}\right)
\quad\Longleftrightarrow\quad k=2.
$$
In particular, for $k=2$ the isomorphism holds for all $a,b\in\C^\times$.
\end{corollary}

\subsection{The $\mathcal{U}(\mathfrak{sl}(m|1))$-module $E\!\left(\sum_{i=1}^m a_i x_i^{k_i},\,\mathcal S\right)$} 
For the rest of the paper we assume $m \in \Z_{\geq 2}$. For each $i\in \mathbf m$, define \footnote{By convention, an empty product equals 1, e.g. 
in the following formulas
each of the products  equals 1 if $i+1>m$.}
$$U\left(i,a_i,\mathbf k_{[i,m]}\right):= U_{(h_i,a_ik_i)}\left((k_i-1)\prod_{j=i+1}^m k_j,\prod_{j=i+1}^m k_j\right) \in \Mat_{\prod_{j=i}^m k_j}(\C[h_i]), $$
$$V\left(i,a_i,\mathbf k_{[i,m]}\right):= V_{(h_i,a_ik_i)}\left(\prod_{j=i+1}^m k_j,(k_i-1)\prod_{j=i+1}^m k_j\right) \in \Mat_{\prod_{j=i}^m k_j}(\C[h_i]). $$

\np
Recall that, for $n\in\Z_{\geq 1}$, $S_n$ denotes the symmetric group on $\{0,1,\dots,n-1\}$. For $\mathcal S_1 \subseteq \mathcal S_2 \subseteq\mathbf m$, define the following sign vector 
$$\nu(\mathcal S_1, \mathcal S_2)\in\{\pm 1\}^{\times \card(\mathcal S_2)},\quad \text{where}\quad \nu(\mathcal S_1, \mathcal S_2)_i \;:=\;\begin{cases} 1, & i\in\mathcal{S}_2 \setminus \mathcal S_1,\\-1, & i\in\mathcal{S}_1.\end{cases}$$
\np
\begin{remark} \label{k=1-case}
    For $i \in \mathbf m$, if $k_i =1$ then
    $$U\left(i,a_i,\mathbf k_{[i,m]}\right) = \frac{h_i}{a_i}\I_{\prod_{j=i+1}^mk_j},\qquad V\left(i,a_i,\mathbf k_{[i,m]}\right) = a_i\I_{\prod_{j=i+1}^mk_j}.$$
    In particular, assuming $k_m=1$, we obtain
    $$U\left(m,a_m,\mathbf k_{[m,m]}\right)=\begin{bmatrix}\frac{h_m}{a_m}\end{bmatrix},\qquad V\left(m,a_m,\mathbf k_{[m,m]}\right)=\begin{bmatrix}a_m \end{bmatrix}.$$
\end{remark}

\begin{lemma} \label{explicitrealizationexp}
    Let $a_1, a_2,\dots, a_m \in \C^{\times}$, $k_1,k_2,\dots,k_m \in \Z_{\geq 1}$, and $\mathcal{S} \subseteq {\bf{m}}$. Then
    $$E\!\left(\sum_{i=1}^m a_i x_i^{k_i},\,\mathcal S\right) \;\simeq\; M\left(A_1,A_2,\dots,A_m\right),\qquad A_i \in \Mat_{\prod_{i=1}^m k_i}(\C[h_i])$$
   with
    $$A_i =
\begin{cases}
\diag\Big(\underbrace{U\left(i,a_i,\mathbf k_{[i,m]}\right),\dots,U\left(i,a_i,\mathbf k_{[i,m]}\right)}_{\prod_{j=1}^{i-1} k_j}\Big), & i \notin \mathcal{S},\\
\diag\Big(\underbrace{V\left(i,a_i,\mathbf k_{[i,m]}\right),\dots,V\left(i,a_i,\mathbf k_{[i,m]}\right)}_{\prod_{j=1}^{i-1} k_j}\Big), & i \in \mathcal{S}.
\end{cases}$$
\end{lemma}
\begin{proof}
   The lemma follows by combining \eqref{importanteq1} and \eqref{importanteq2} where the matrices $A_i$ 
   defined from the action of the elements $e_{i, \bar 1}$ is with respect to the $\mathcal U(\mathfrak h)$-basis
   $$\mathcal B \;=\;\left\{V_{\mathbf r,\varepsilon}:= \mathbf x^{\mathbf r}\,\xi^{\varepsilon}\, 
   e^{\sum_{i=1}^m a_i x_i^{k_i}}\bigm|  \mathbf r\in \mathfrak R(k_1,\dots,k_m),\ \varepsilon\in\{0,1\}\right\}$$
of $E\!\left(\sum_{i=1}^m a_i x_i^{k_i},\,\mathcal S\right)$ ordered by $\preccurlyeq$.
\end{proof}

\np
\begin{corollary} \label{first-explicit}
Let $a_1,\dots,a_n\in \C^\times$, $k\in\Z_{\geq 1}$, and $\mathcal S\subseteq\mathbf m$. Then 
$$
E\!\left(a_1x_1^{k}+\sum_{i=2}^m a_i x_i,\ \mathcal S\right)
\;\simeq\; M\bigl(A_1,A_2,\dots,A_m\bigr),
$$
where each $A_i\in \Mat_k\big(\C[\mathbf h]\big)$ and
$$
A_1=\begin{cases}
U_{(h_1,a_1k)}(k-1,1), & 1\notin\mathcal S,\\
V_{(h_1,a_1k)}(1,k-1), & 1\in\mathcal S,
\end{cases}
\quad\text{and}\quad A_i=\begin{cases}
\dfrac{h_i}{a_i}\,\I_k, & i\notin\mathcal S,\\
a_i\,\I_k, & i\in\mathcal S,
\end{cases}\qquad (i\in\mathbf m\setminus\{1\}).
$$
\end{corollary}

\begin{lemma} \label{orbitlemma}
Let $\mathcal{S} \subseteq \mathbf m$. For each $i\in\mathbf m$, define
$$
  \pi_{k_i} :=
  \begin{cases}
    (0\;\;1\;\;\dots\;\;k_i\!-\!1) \in S_{k_i}, & \text{if } i \notin \mathcal S, \\
    (0\;\;k_i\!-\!1\;\;k_i\!-\!2\;\;\dots\;\;1) \in S_{k_i}, & \text{if } i \in \mathcal S,
  \end{cases}
$$    
viewed as a permutation of $\mathbb{Z}/k_{i}\mathbb{Z}$. Extend $\pi_{k_{i}}$ to a permutation of $\mathfrak{R}(k_{1},\dots,k_{m})$ by letting it act on the $i$-th component and fix every other component. For ${\bf{a}} \in \mathfrak{R}(k_1,\dots,k_m)$, let
$$\|{\bf{a}}\|_{\mathcal{S}} := \sum_{i\notin S} a_i - \sum_{j\in S} a_j.$$
Then $H:= \left\langle \pi_{k_i} \pi_{k_{i+1}}^{-1}\bigm|\ i\in \mathbf m \setminus \{m\} \right\rangle$ acts transitively on
$$
O_{p,\mathcal S}\;:=\;\bigl\{{\bf a}\in \mathfrak R(k_1,\dots,k_m)\bigm|\ \|{\bf a}\|_{\mathcal S}\equiv p \pmod s\bigr\},\qquad p\in\{0,1,\dots, s-1\}, 
$$
where $s= \gcd(k_1,\dots,k_m)$.
\end{lemma}

\begin{proof}
We recall that
$$\nu(\mathcal S, \mathbf m)\in\{\pm 1\}^{\times m},\quad \text{where}\quad \nu(\mathcal S, \mathbf m)_i \;:=\;\begin{cases} 1, & i\notin \mathcal{S} ,\\-1, & i\in\mathcal{S}.\end{cases}$$
\np
Then, for ${\bf{a}} \in \mathfrak{R}(k_1,\dots,k_m)$,
$$\|{\bf{a}}\|_{\mathcal{S}} = \sum_{i\notin S} a_i - \sum_{j\in S} a_j = \sum_{i=1}^m \nu(\mathcal S, \mathbf m)_i\, a_i.$$
\np
For each $i\in \mathbf m $, $\pi_{k_i}$ acts on $a_i$ by $a_i\mapsto a_i+\nu(\mathcal S, \mathbf m)_i $ with the addition taken in $\Z/k_i\Z$. Therefore, for $i\in \mathbf m \setminus \{m\} $,  $\pi_{k_i} \pi_{k_{i+1}}^{-1}$ acts via
$${\bf{a}} \mapsto {\bf{a}}+ v_i,\qquad  v_i :=  \nu(\mathcal S, \mathbf m)_i \,e_i - \nu(\mathcal S, \mathbf m)_{i+1}\, e_{i+1}.$$
\np
Since
$$\|{{\bf{a}}+ v_i}\|_{\mathcal{S}} = \|{\bf{a}}\|_{\mathcal{S}}  +\bigl(\nu(\mathcal S, \mathbf m)_i\bigr)^2 - \bigl(\nu(\mathcal S, \mathbf m)_{i+1}\bigr)^2 = \|{\bf{a}}\|_{\mathcal{S}},$$

$\pi_{k_i} \pi_{k_{i+1}}^{-1}$ preserves $\|{\bf{a}}\|_{\mathcal{S}}$  (hence its residue class mod $s$). Consequently, $H$ stabilizes $O_{p,\mathcal S}$.
Let $\mathbf a, \mathbf b \in O_{p,\mathcal S}$. Since $\|{\bf{a}}\|_{\mathcal{S}} \equiv \|{\bf{b}}\|_{\mathcal{S}} \equiv p \pmod{s},$ it follows that
$$\sum_{i=1}^m \nu(\mathcal S, \mathbf m)_i\, (b_i-a_i) \equiv 0 \pmod{s}.$$
Hence there exists $t\in\Z$ such that $\sum_{i=1}^m \nu(\mathcal S, \mathbf m)_i\, (b_i-a_i) = ts .$
Since $s = \gcd(k_1, \dots, k_m)$, there exist $u_i \in \Z$ such that $\sum_{i=1}^mu_ik_i =s$. Set $z_i := -t\nu(\mathcal S, \mathbf m)_i\,u_i$, for all $i\in \mathbf m$. Then
$$
  \sum_{i=1}^m \nu(\mathcal S, \mathbf m)_i k_i z_i \;=\;-t\sum_{i=1}^m \left(\nu(\mathcal S, \mathbf m)_i\right)^2 k_i u_i \;=\;  -ts.
$$
Therefore, $\sum_{i=1}^m \nu(\mathcal S, \mathbf m)_i\, (b_i-a_i+z_ik_i) =0$. Set 
$$c_j:= \sum_{\ell=1}^j \nu(\mathcal S, \mathbf m)_\ell\, (b_\ell-a_\ell+z_\ell k_\ell),\qquad 1\leq j \leq m-1,$$ 
$$h:= \left(\pi_{k_{m-1}} \pi_{k_{m}}^{-1}\right)^{c_{m-1}}\left(\pi_{k_{m-2}} \pi_{k_{m-1}}^{-1}\right)^{c_{m-2}}\dots\, \left(\pi_{k_1} \pi_{k_{2}}^{-1}\right)^{c_{1}}\in H.$$ 

\np
Since $\sum_{i=1}^m \nu(\mathcal S, \mathbf m)_i\, (b_i-a_i+z_ik_i) =0$, then $h$ acts on $\mathbf a$ as
$$a_i \mapsto a_i+ \nu(\mathcal S, \mathbf m)_i^2\, (b_i-a_i+z_ik_i) = b_i + z_ik_i,\qquad 1\leq i \leq m. $$

\np
This shows that $h\cdot\mathbf a=\mathbf b$ in $\mathfrak R(k_1,\dots,k_m)$; hence $H$ acts transitively on $O_{p,\mathcal S}$.
\end{proof}

\begin{corollary} \label{generalizedH-S}
    Let $H$ be as in Lemma~\ref{orbitlemma}, $H$ acts transitively on $\mathfrak{R}(k_1,\dots,k_m)$ if and only if $\gcd(k_1,\dots,k_m) =1$.
\end{corollary}

\begin{remark} \label{importantremark}
    Since
    $$
\mathfrak{R}(k_1,\dots,k_m)\;=\;\mathbb{Z}/k_1\mathbb{Z}\times \dots \times \mathbb{Z}/k_m\mathbb{Z},
$$
the symmetric group on this set, $S_{\mathfrak{R}(k_{1},\dots,k_{m})}$, is isomorphic to $S_{\prod_{i=1}^m k_i}$.
By Lemma~\ref{explicitrealizationexp}, the permutation induced by $A_i$ is the $k_i$-cycle
    $$
\pi(A_i)=
\begin{cases}
(0\;\;1\;\;\dots\;\;k_i\!-\!1), & \text{if } i\notin \mathcal S,\\
(0\;\;k_i\!-\!1\;\;k_i\!-\!2\;\dots\;1), & \text{if } i\in \mathcal S.
\end{cases}
$$
For $\mathbf r \in \mathfrak{R}(k_1,\dots,k_m)$, $\pi(A_i)$ acts natually on the $i$-th component of $\mathbf r$ and fixes all other components.
\end{remark}

\begin{lemma} \label{decompositionlemma}
Let $O_{p,\mathcal S}$ be as in Lemma~\ref{orbitlemma}. For $i\in\mathbf m$, define
$$\overline{O}_{p,\mathcal{S}}^{(i)}:=\left\{\pi\left(A_i\right)^{-1}\cdot {\bf{a}}\bigm|\ {\bf{a}} \in {O}_{p,\mathcal{S}}  \right\}. $$
Then 
\begin{enumerate}
    \item The set $\overline{O}_{p,\mathcal S}^{(i)}$ is independent of $i$. In particular, we may write
$$
\overline{O}_{p,\mathcal S}
:=\pi(A_i)^{-1}\!\cdot O_{p,\mathcal S}\quad\text{for any}\quad i\in\mathbf m.
$$
\item For every $i\in\mathbf m$ and every $\mathbf r\in\overline{O}_{p,\mathcal S}$, $\pi(A_i)\cdot \mathbf r \in O_{p,\mathcal S}$.
\end{enumerate}
\end{lemma}

\begin{proof}
    Let $i,j \in \mathbf m$ with $i\neq j$, and let $\mathbf b\in\overline{O}_{p,\mathcal S}^{(i)}$. Then there exists $\mathbf{a} \in {O}_{p,\mathcal{S}} $ such that $\mathbf{b} = \pi\left(A_i\right)^{-1}\cdot\mathbf{a}$.
    Since 
    $$\big\|\pi(A_j)\cdot\mathbf b\big\|_{\mathcal{S}} =\big\|\pi(A_j)\pi\left(A_i\right)^{-1}\cdot\mathbf{a}\big\|_{\mathcal{S}}= \|{\bf{a}}\|_{\mathcal{S}} - \left(\nu(\mathcal S, \mathbf m)_i\right)^2 + \left(\nu(\mathcal S, \mathbf m)_j\right)^2 = \|{\bf{a}}\|_{\mathcal{S}},$$

it follows that $\pi(A_j)\cdot \mathbf b \in {O}_{p,\mathcal{S}}$, which implies $\mathbf b \in \overline{O}_{p,\mathcal{S}}^{(j)}$. This proves part (1). Part (2) follows directly from (1).
\end{proof}

\begin{corollary} \label{onlyif-main}
Let $a_1,\dots,a_m\in\C^{\times}$, $\mathcal S\subseteq\mathbf m$, and $k_1,\dots,k_m\in\Z_{\geq1}$ with
$s:=\gcd(k_1,\dots,k_m)$. Then, the module $E\big(\sum_{i=1}^m a_i x_i^{k_i},\,\mathcal S\big)$ admits the decomposition:
$$E\!\left(\sum_{i=1}^m a_i x_i^{k_i},\,\mathcal S\right) \;=\; \bigoplus_{j\in \Z/s\Z} M^{j}, \qquad
M^{j}=M^{j}_{\bar 0}\oplus M^{j}_{\bar 1}$$
where,
$$
M^{j}_{\bar 0}
=\bigoplus_{\mathbf r\in O_{j,\mathcal S}}
\mathbf x^{\mathbf r}\,\C[x_1^{k_1},\dots,x_m^{k_m}]\,
e^{\sum_{i=1}^{m} a_i x_i^{k_i}},
\quad
M^{j}_{\bar 1}
=\bigoplus_{\mathbf r'\in \overline{O}_{j,\mathcal S}}
\mathbf x^{\mathbf r'}\,\xi\,\C[x_1^{k_1},\dots,x_m^{k_m}]\,
e^{\sum_{i=1}^{m} a_i x_i^{k_i}}.
$$
\end{corollary}

\begin{proof}
For each $j\in\Z/s\Z$, a direct computation shows that, for all $i\in\mathbf m$, the actions of $e_{i\bar{1}}$ and $e_{\bar{1}i}$ are stable on $M^{j}$. Consequently, $M^{j}$ is a submodule of $E\!\left(\sum_{i=1}^m a_i x_i^{k_i},\,\mathcal S\right)$. The sum is direct since
$$\mathfrak{R}(k_1,\dots,k_m) \;=\; \bigsqcup_{p\in \Z/s\Z}\;O_{p,\mathcal{S}}\;=\; \bigsqcup_{p\in \Z/s\Z}\;\overline{O}_{p,\mathcal{S}}. $$ 
\end{proof}

\begin{theorem} \label{maintheorem_indec}
    Let $a_1,\dots, a_m \in \C^{\times}$, $k_1,\dots,k_m \in \Z_{\geq 1}$, and $\mathcal{S} \subseteq {\bf{m}}$. Then,
    $$E\!\left(\sum_{i=1}^m a_i x_i^{k_i},\,\mathcal S\right)$$
    is indecomposable if and only if $\gcd(k_1,\dots,k_m)=1$.
\end{theorem}    

\begin{proof}
The ``only if" direction follows directly from Corollary \ref{onlyif-main}. For the “if’’ direction, suppose that $\gcd(k_1,\dots,k_m)=1$. By Proposition~\ref{explicitrealizationexp}, it suffices to prove the indecomposability of $ M(A_1,\dots,A_m)$, where
$$ A_i =
\begin{cases}
\diag\Big(\underbrace{U\left(i,a_i,\mathbf k_{[i,m]}\right),\dots,U\left(i,a_i,\mathbf k_{[i,m]}\right)}_{\prod_{j=1}^{i-1} k_j}\Big), & i \notin \mathcal{S},\\
\diag\Big(\underbrace{V\left(i,a_i,\mathbf k_{[i,m]}\right),\dots,V\left(i,a_i,\mathbf k_{[i,m]}\right)}_{\prod_{j=1}^{i-1} k_j}\Big), & i \in \mathcal{S}.
\end{cases}$$
\np
 Let $\Phi \in \End_{\,\mathcal{U}(\mathfrak{sl}(m|1))} \left(M\bigl(A_1,\dots,A_m \bigr) \right)$. By Proposition \ref{simplified-W(h)}, there exist matrices  $W_1({\bf{h}}), W_4({\bf{h}}) \in \Mat_{\prod_{j=1}^m k_j}(\C[{\bf{h}}])$ such that
$$\Phi( \mathbf f({\bf h})) = \left[\begin{array}{ c | c }
    W_1({\bf{h}}) & \mathbf 0 \\
    \hline
    \mathbf 0 & W_4({\bf{h}})
  \end{array}\right]\,\mathbf f({\bf h}),\quad \text{for all}\quad \mathbf f({\bf h}) \in \C[{\bf{h}}]^{\oplus 2\prod_{j=1}^m k_j}, $$
  and, for every $i\in \mathbf m$,
  $$W_1({\bf{h}})\, A_i \,=\,  A_i\, \Delta^{-1}_i\left(W_4({\bf{h}})\right). $$
\np
In particular,$\, W_1(\mathbf h)\,A_i \;=\; A_i\,\Delta_1^{-1}\big(W_4(\mathbf h)\big)$ for $i \in \{1,2\}$. Eliminating $W_4$ gives
$$
W_1(\mathbf h)
= \Delta_1^{-1}\!\big(A_1A_2^{-1}\big)\;\Delta_1^{-1}\!\big(\Delta_2(W_1(\mathbf h))\big)\;
\Delta_1^{-1}\!\big(A_2A_1^{-1}\big).
$$
\np
Equivalently,
$$ W_1({\bf{h}}) =A_1\sigma_2^{-1}(A_2^{-1}) ~\sigma_1\sigma_2^{-1}(W_1({\bf{h}}))~\sigma_2^{-1}(A_2)A_1^{-1},$$
which yields
\begin{multline*}
 W_1({\bf{h}}) =U\left(1,a_1,\mathbf k\right)\sigma_2^{-1}\left(\diag\Big(\underbrace{\mathfrak{U}(2)^{-1},\dots, \mathfrak{U}(2)^{-1}}_{k_1\, \text{copies}} \Big)\right) ~\sigma_1\sigma_2^{-1}(W_1({\bf{h}}))\\
 ~\sigma_2^{-1}\left(\diag\Big(\underbrace{\mathfrak{U}(2),\dots, \mathfrak{U}(2)}_{k_1\, \text{copies}} \Big)\right)U(1,a_1,\mathbf k)^{-1},\quad\text{if}\quad 1\notin \mathcal{S};
\end{multline*}
\begin{multline*}
 W_1({\bf{h}}) =V(1,a_1,\mathbf k)\sigma_2^{-1}\left(\diag\Big(\underbrace{\mathfrak{U}(2)^{-1},\dots, \mathfrak{U}(2)^{-1}}_{k_1\, \text{copies}} \Big)\right) ~\sigma_1\sigma_2^{-1}(W_1({\bf{h}}))\\
 ~\sigma_2^{-1}\left(\diag\Big(\underbrace{\mathfrak{U}(2),\dots, \mathfrak{U}(2)}_{k_1\, \text{copies}} \Big)\right)V(1,a_1,\mathbf k)^{-1},\quad\text{if}\quad 1\in \mathcal{S};
\end{multline*}
where 
$$
\mathfrak{U}(2) =
\begin{cases}
U\left(2,a_2,\mathbf k_{[2,m]}\right), & \text{if } 2 \notin \mathcal{S}, \\
V\left(2,a_2,\mathbf k_{[2,m]}\right), & \text{if } 2 \in \mathcal{S}.
\end{cases}
$$
\np
\begin{remark}
    If $k_1=1$, we have $W_1(\mathbf h)  \in \Mat_{\prod_{i=2}^m k_i}(\C[\mathbf h])$; 
otherwise, assume $k_1\geq 2$ and proceed as follows.
\end{remark}
\np
We now compute the products $$U(1,a_1,\mathbf k)\sigma_2^{-1}\left(\diag\Big(\underbrace{\mathfrak{U}(2)^{-1},\dots, \mathfrak{U}(2)^{-1}}_{k_1\, \text{copies}} \Big)\right),\; V(1,a_1,\mathbf k)\sigma_2^{-1}\left(\diag\Big(\underbrace{\mathfrak{U}(2)^{-1},\dots, \mathfrak{U}(2)^{-1}}_{k_1\, \text{copies}} \Big)\right)$$ 
explicitly. Let ${\bf 0}$ denote the $\left(\prod_{i=2}^m k_i\right) \times \left(\prod_{i=2}^m k_i\right)$ zero matrix. Then

\begin{align*}
&U(1,a_1,\mathbf k) \sigma_2^{-1}\left(\diag\Big(\underbrace{\mathfrak{U}(2)^{-1},\dots, \mathfrak{U}(2)^{-1}}_{k_1\, \text{copies}} \Big)\right) \\
    &= \left[
\begin{array}{ccccc}
{\bf 0} & {\bf 0} & \cdots & {\bf 0} &
      \dfrac{h_{1}}{a_{1}k_{1}}\,\sigma_2^{-1}\left(\mathfrak{U}(2)^{-1}\right) \\
\sigma_2^{-1}\left(\mathfrak{U}(2)^{-1}\right) & {\bf 0} & \cdots & {\bf 0} & {\bf 0} \\
{\bf 0} & \sigma_2^{-1}\left(\mathfrak{U}(2)^{-1}\right) & \vdots & \vdots & \vdots \\
\vdots & \ddots & \vdots & {\bf 0} & {\bf 0}\\
{\bf 0} & \cdots & {\bf 0} & \sigma_2^{-1}\left(\mathfrak{U}(2)^{-1}\right) & {\bf 0}
\end{array}
\right],
\end{align*}

\begin{align*}
    & V(1,a_1,\mathbf k)\sigma_2^{-1}\left(\diag\Big(\underbrace{\mathfrak{U}(2)^{-1},\dots, \mathfrak{U}(2)^{-1}}_{k_1\, \text{copies}} \Big)\right) \\
    &=-h_1\begin{bmatrix}
{\bf 0}      & \sigma_2^{-1}\left(\mathfrak{U}(2)^{-1}\right)     & {\bf 0}      & \dots & {\bf 0}     \\
{\bf 0}     & {\bf 0}      &  \sigma_2^{-1}\left(\mathfrak{U}(2)^{-1}\right)    & \dots & {\bf 0}    \\
\vdots & \vdots & \ddots & \ddots & \vdots \\
{\bf 0}      & {\bf 0}     & \dots & {\bf 0}     & \sigma_2^{-1}\left(\mathfrak{U}(2)^{-1}\right)    \\
-\frac{a_1k_1}{h_1}\sigma_2^{-1}\left(\mathfrak{U}(2)^{-1}\right)     & {\bf 0}      & \dots & {\bf 0}     & {\bf 0}
\end{bmatrix}.
\end{align*}
\np
Set
$$W_1({\bf{h}}):= \bigl(\mathcal{W}_{i,j}({\bf{h}})\bigr)_{i,j=1}^{k_{1}},\quad \text{where}\quad \mathcal{W}_{i,j}({\bf{h}}) \in \Mat_{\prod_{i=2}^m k_i}(\C[{\bf{h}}]).$$
$$\widetilde{\mathcal{W}}_{i,j}({\bf{h}}):= \sigma_2^{-1}(\mathfrak{U}(2)^{-1}) \sigma_1 \sigma_2^{-1}(\mathcal{W}_{i,j}({\bf{h}})) \sigma_2^{-1}(\mathfrak{U}(2)).$$
\np
Since
$$\bigl(\mathcal{W}_{i,j}({\bf{h}})\bigr)_{i,j=1}^{k_1}  = A_1\sigma_2^{-1}(A_2^{-1})\, \bigl(\sigma_1\sigma_2^{-1}\left(\mathcal{W}_{i,j}({\bf{h}})\right)\bigr)_{i,j=1}^{k_1} \, \sigma_2^{-1}(A_2)A_1^{-1}, $$
then, if $ 1 \notin \mathcal{S} $, 
$$
\bigl(\mathcal{W}_{i,j}({\bf h})\bigr)_{i,j=1}^{k_1}
=
\begin{bmatrix}
\widetilde{\mathcal{W}}_{k_1,k_1}({\bf h}) &
\displaystyle \frac{h_1}{a_1 k_1}\,\widetilde{\mathcal{W}}_{k_1,1}({\bf h}) &
\cdots &
\displaystyle \frac{h_1}{a_1 k_1}\,\widetilde{\mathcal{W}}_{k_1,k_1-1}({\bf h}) \\
\displaystyle \frac{a_1 k_1}{h_1}\,\widetilde{\mathcal{W}}_{1,k_1}({\bf h}) &
\widetilde{\mathcal{W}}_{1,1}({\bf h}) &
\cdots &
\widetilde{\mathcal{W}}_{1,k_1-1}({\bf h}) \\
\vdots & \vdots & \ddots & \vdots \\
\displaystyle \frac{a_1 k_1}{h_1}\,\widetilde{\mathcal{W}}_{k_1-1,k_1}({\bf h}) &
\widetilde{\mathcal{W}}_{k_1-1,1}({\bf h}) &
\cdots &
\widetilde{\mathcal{W}}_{k_1-1,k_1-1}({\bf h})
\end{bmatrix}.
$$
\np
If $ 1 \in \mathcal{S} $, 
$$
\bigl(\mathcal{W}_{i,j}({\bf h})\bigr)_{i,j=1}^{k_1}
=
\begin{bmatrix}
\widetilde{\mathcal{W}}_{2,2}({\bf h}) & \cdots & \widetilde{\mathcal{W}}_{2,k_1}({\bf h}) &
\displaystyle -\frac{h_1}{a_1 k_1}\,\widetilde{\mathcal{W}}_{2,1}({\bf h}) \\
\vdots & \ddots & \vdots & \vdots \\
\widetilde{\mathcal{W}}_{k_1,2}({\bf h}) & \cdots & \widetilde{\mathcal{W}}_{k_1,k_1}({\bf h}) &
\displaystyle -\frac{h_1}{a_1 k_1}\,\widetilde{\mathcal{W}}_{k_1,1}({\bf h}) \\
\displaystyle -\frac{a_1 k_1}{h_1}\,\widetilde{\mathcal{W}}_{1,2}({\bf h}) &
\cdots &
\displaystyle -\frac{a_1 k_1}{h_1}\,\widetilde{\mathcal{W}}_{1,k_1}({\bf h}) &
\widetilde{\mathcal{W}}_{1,1}({\bf h})
\end{bmatrix}.
$$
\np
For $i \neq j$,
\begin{multline*}
    \mathcal{W}_{i,j}({\bf{h}}) =\frac{h_1-i+1}{h_1-j+1} \,\left(\sigma_2^{-1}\left(\mathfrak{U}(2)^{-1}\right)\dots\sigma_2^{-k_1}\left(\mathfrak{U}(2)^{-1}\right)\right) \, \left(\sigma_1 \sigma_2^{-1}\right)^{k_1} \left(\mathcal{W}_{i,j}({\bf{h}})\right) \\\left(\sigma_2^{-1}\left(\mathfrak{U}(2)^{-1}\right)\dots\sigma_2^{-k_1}\left(\mathfrak{U}(2)^{-1}\right)\right)^{-1}, \quad \text{if }\quad 1 \notin \mathcal{S},
\end{multline*}
\begin{multline*}
    \mathcal{W}_{i,j}({\bf{h}}) =\frac{h_1-k_1+j}{h_1-k_1+i} \,\left(\sigma_2^{-1}\left(\mathfrak{U}(2)^{-1}\right)\dots\sigma_2^{-k_1}\left(\mathfrak{U}(2)^{-1}\right)\right) \, \left(\sigma_1 \sigma_2^{-1}\right)^{k_1} \left(\mathcal{W}_{i,j}({\bf{h}})\right) \\\left(\sigma_2^{-1}\left(\mathfrak{U}(2)^{-1}\right)\dots\sigma_2^{-k_1}\left(\mathfrak{U}(2)^{-1}\right)\right)^{-1}, \quad \text{if }\quad 1 \in \mathcal{S}.
\end{multline*}
\np
Since  
$$
\left(\sigma_2^{-1}\bigl(\mathfrak{U}(2)^{-1}\bigr) \dots \sigma_2^{-k_1}\bigl(\mathfrak{U}(2)^{-1}\bigr)\right) 
\in \Mat_{\prod_{j=2}^m k_j}\!\bigl(\C[h_2]\bigr)
$$
is a GPM, it follows from Lemma~\ref{simplifymatrixentries2} that $$\mathcal{W}_{i,j}({\bf{h}}) = {\bf 0},\quad\text{if}\quad i\neq j.$$
\np
Consequently,
$$
W_1(\mathbf h)=\diag\bigl(\mathcal T_1(\mathbf h),\dots,\mathcal T_{k_1}(\mathbf h)\bigr),
\qquad
\mathcal T_i(\mathbf h):=\mathcal W_{i,i}(\mathbf h)\in
\Mat_{\prod_{r=2}^m k_r}\big(\C[\mathbf h]\big)\ \ (1\le i\le k_1).
$$
\np
Similarly, $W_1(\mathbf h)=A_2\,\sigma_3^{-1}(A_3^{-1})\,
\sigma_2\sigma_3^{-1}(W_1(\mathbf h))\,
\sigma_3^{-1}(A_3)\,A_2^{-1}$. Hence, for each $i\in\{1,\dots,k_1\}$, the diagonal block $\mathcal T_i(\mathbf h)$ satisfies
\begin{multline*}
\mathcal{T}_i({\bf{h}}) = 
U\left(2,a_2,\mathbf k_{[2,m]}\right)\,\sigma_3^{-1}\left(\diag\!\Big(\underbrace{\mathfrak{U}(3)^{-1},\dots, \mathfrak{U}(3)^{-1}}_{k_2\, \text{copies}} \Big)\right)\sigma_2\sigma_3^{-1}\left(\mathcal{T}_i({\bf{h}})\right) \\
\sigma_3^{-1}\left(\diag\Big(\underbrace{\mathfrak{U}(3),\dots, \mathfrak{U}(3)}_{k_2\, \text{copies}} \Big)\right)U\left(2,a_2,\mathbf k_{[2,m]}\right)^{-1},\quad\text{if}\quad 2\notin \mathcal{S};
\end{multline*}
\begin{multline*}
 \mathcal{T}_i({\bf{h}}) =V\left(2,a_2,\mathbf k_{[2,m]}\right)\sigma_3^{-1}\left(\diag\Big(\underbrace{\mathfrak{U}(3)^{-1},\dots, \mathfrak{U}(3)^{-1}}_{k_2\, \text{copies}} \Big)\right) ~\sigma_2\sigma_3^{-1}(\mathcal{T}_i({\bf{h}}))\\
 ~\sigma_3^{-1}\left(\diag\Big(\underbrace{\mathfrak{U}(3),\dots, \mathfrak{U}(3)}_{k_2\, \text{copies}} \Big)\right)V\left(2,a_2,\mathbf k_{[2,m]}\right)^{-1},\quad\text{if}\quad 2\in \mathcal{S};
\end{multline*}
where
$$
\mathfrak{U}(3) =
\begin{cases}
U\left(3,a_3,\mathbf k_{[3,m]}\right), & \text{if } 3 \notin \mathcal{S}, \\
V\left(3,a_3,\mathbf k_{[3,m]}\right), & \text{if } 3 \in \mathcal{S}.
\end{cases}
$$
\np
By the same argument, for each $i$, we further deduce
$$
\mathcal T_i(\mathbf h)
= \diag\bigl(\mathcal T_{(i,1)}(\mathbf h),\ldots,\mathcal T_{(i,k_2)}(\mathbf h)\bigr),\qquad \mathcal T_{(i,j)}(\mathbf h)\in \Mat_{\prod_{r=3}^m k_r}\big(\C[\mathbf h]\big),\;\forall j\in\{1,\ldots,k_2\}.
$$

\np
Proceeding inductively in $i$ from $1$ to $m$ and omitting the steps with $k_i=1$, we get
$$
W_1({\bf h})
= \diag\Bigl(\mathcal{P}_1({\bf h}),\dots,\mathcal{P}_{\prod_{r=1}^{m-1} k_r}({\bf h})\Bigr),\qquad\mathcal{P}_{p}({\bf h}) \in \Mat_{k_m}\bigl(\C[{\bf h}]\bigr),\;\forall p\in \bigl\{1,\dots,\prod_{r=1}^{m-1} k_r\bigr\}.
$$
\np
With respect to the lexicographic order $\preccurlyeq_{\mathfrak{R}(k_1,\dots,k_m)}$ on $\mathfrak{R}(k_1,\dots,k_m)$, we index the entries of $W_1({\bf{h}})$ as follows:
$$W_1({\bf{h}}):= \Bigl( w_{{\bf{r}},{\bf{r'}}} ({\bf{h}}) \Bigr)_{{\bf{r}},{\bf{r'}} = (0,\dots,0)}^{(k_1-1,\dots,k_m-1)},\;\;\text{where}\quad w_{{\bf{r}},{\bf{r'}}} ({\bf{h}}) \in \C[{\bf{h}}].$$
 In particular, each $k_m \times k_m$ diagonal block of $W_1({\bf{h}})$ is indexed by a prefix $\mathbf r^{<m}:=(r_1,\dots,r_{m-1})$ using the lexicographic order $\preccurlyeq_{\mathfrak{R}(k_1,\dots,k_{m-1})}$ on $\mathfrak{R}(k_1,\dots,k_{m-1})$. In explicit terms,
$$W_1({\bf{h}}):=\diag\Bigl(\mathcal{P}_{\mathbf r^{<m}}({\bf{h}})\Bigr)_{\mathbf r^{<m} \in \mathfrak{R}(k_1,\dots,k_{m-1})} =  \diag\Bigl(\mathcal{P}_{(0,\dots,0)}({\bf{h}}),\dots, \mathcal{P}_{(k_1-1,\dots,k_{m-1}-1)}({\bf{h}}) \Bigr),$$
where
$$\mathcal{P}_{\mathbf r^{<m}}({\bf{h}}):=\Bigl( w_{{\bf{r}},{\bf{r'}}}({\bf{h}}) \Bigr)_{{\bf{r}},{\bf{r'}} = (\mathbf r^{<m},0)}^{(\mathbf r^{<m},k_{m}-1)}. $$
\np
Note that, for all $i,j \in \{0,\ldots,k_m-1\}$ and all $\mathbf r^{<m},\mathbf r'^{<m}\in \mathfrak R(k_1,\dots,k_{m-1})$ with $\mathbf r^{<m}\neq \mathbf r'^{<m}$, we have
$$
w_{(\mathbf r^{<m},\,i),\,(\mathbf r'^{<m},\,j)}(\mathbf h)=0.
$$

\begin{remark}
Our goal is to show that $W_1(\mathbf h)$ is diagonal. The case $k_m=1$ is immediate, so we henceforth assume $k_m\geq 2$ and proceed with the proof.
\end{remark}
\np
\textbf{Claim.}
Let $\mathbf r^{<m}\in \mathfrak R(k_1,\dots,k_{m-1})$. For any distinct $i,j\in\{0,\ldots,k_m-1\}$,
$$ w_{(\mathbf r^{<m},\,i),\,(\mathbf r^{<m},\,j)}(\mathbf h)=0.
$$

\noindent\textbf{Proof of Claim.}  For $\mathbf r^{<m} :=(r_1,\dots,r_{m-1}) \in \mathfrak R(k_1,\dots,k_{m-1})$, we first compute the products
$$
\sigma_m^{-1}\big(U(m,a_m,k_m)^{-1}\big)\sigma_1\sigma_m^{-1}\big(\mathcal P_{\mathbf r^{<m}}(\mathbf h)\big)\sigma_m^{-1}\big(U(m,a_m,k_m)\big),$$
and
$$\sigma_m^{-1}\big(V(m,a_m,k_m)^{-1}\big)\sigma_1\sigma_m^{-1}\big(\mathcal P_{\mathbf r^{<m}}(\mathbf h)\big)\sigma_m^{-1}\big(V(m,a_m,k_m)\big).
$$
\np
For brevity, we set
$$
\gamma:=\frac{h_m+1}{a_m k_m},
\qquad
\widehat{w}_{p,q}(\mathbf h):=(\sigma_1\sigma_m^{-1})\bigl(w_{(\mathbf r^{<m},p),(\mathbf r^{<m},q)}(\mathbf h)\bigr)\quad \text{for}\;\; p,q \in \{0,\dots, k_m-1\}.
$$
\np
Then
\begin{align*}
    &\sigma_m^{-1}\!\bigl(U(m,a_m,k_m)^{-1}\bigr)\;
\sigma_1\sigma_m^{-1}\!\bigl(\mathcal P_{\mathbf r^{<m}}(\mathbf h)\bigr)\;
\sigma_m^{-1}\!\bigl(U(m,a_m,k_m)\bigr)\\
=&\begin{bmatrix}
\widehat{w}_{1,1}(\mathbf h) & \cdots & \widehat{w}_{1,k_m-1}(\mathbf h) & \gamma\,\widehat{w}_{1,0}(\mathbf h)\\
\vdots            & \ddots & \vdots                 & \vdots\\
\widehat{w}_{k_m-1,1}(\mathbf h) & \cdots & \widehat{w}_{k_m-1,k_m-1}(\mathbf h) & \gamma\,\widehat{w}_{k_m-1,0}(\mathbf h)\\
\gamma^{-1}\widehat{w}_{0,1}(\mathbf h) & \cdots & \gamma^{-1}\widehat{w}_{0,k_m-1}(\mathbf h) & \widehat{w}_{0,0}(\mathbf h)
\end{bmatrix},
\end{align*}
and
\begin{align*}
    &\sigma_m^{-1}\!\bigl(V(m,a_m,k_m)^{-1}\bigr)\;
\sigma_1\sigma_m^{-1}\!\bigl(\mathcal P_{\mathbf r^{<m}}(\mathbf h)\bigr)\;
\sigma_m^{-1}\!\bigl(V(m,a_m,k_m)\bigr)\\
=&\begin{bmatrix}
\widehat{w}_{k_m-1,k_m-1}(\mathbf h) & -\gamma\,\widehat{w}_{k_m-1,0}(\mathbf h) & \cdots & -\gamma\,\widehat{w}_{k_m-1,k_m-2}(\mathbf h)\\
-\gamma^{-1}\widehat{w}_{0,k_m-1}(\mathbf h) & \widehat{w}_{0,0}(\mathbf h) & \cdots & \widehat{w}_{0,k_m-2}(\mathbf h)\\
\vdots & \vdots & \ddots & \vdots\\
-\gamma^{-1}\widehat{w}_{k_m-2,k_m-1}(\mathbf h) & \widehat{w}_{k_m-2,0}(\mathbf h) & \cdots & \widehat{w}_{k_m-2,k_m-2}(\mathbf h)
\end{bmatrix}.
\end{align*}
\np
From Remark \ref{importantremark}, we have
$$
\pi(A_i)=
\begin{cases}
(0\;\;1\;\;\dots\;\;k_i\!-\!1), & \text{if } i\notin \mathcal S,\\
(0\;\;k_i\!-\!1\;\;k_i\!-\!2\;\dots\;1), & \text{if } i\in \mathcal S.
\end{cases}
$$
\np
From 
$$
W_1(\mathbf h)=A_1\,\sigma_m^{-1}(A_m^{-1})\,\sigma_1\sigma_m^{-1}(W_1(\mathbf h))\,\sigma_m^{-1}(A_m)\,A_1^{-1},
$$

we obtain the following relations among the entries. Fix
$\mathbf r^{<m}\in\mathfrak R(k_1,\dots,k_{m-1})$ and distinct $i, j\in\{0,\dots,k_m-1\}$.

\np
\textbf{Case 1 ($m\notin\mathcal S$).}Then
$$w_{(\mathbf r^{<m},\,i),\,(\mathbf r^{<m},\,j)}(\mathbf h) = f_{(i,j)}(h_m)(\sigma_1\sigma_m^{-1})\left(w_{\pi\left(A_1^{-1}\right)\cdot(\mathbf r^{<m},\,i+1),\,\pi\left(A_1^{-1}\right)\cdot(\mathbf r^{<m},\,j+1)}(\mathbf h) \right),$$
where  
$$f_{(i,j)}(h_m) = \begin{cases}\dfrac{a_mk_m}{h_m+1}, & \text{if } i=k_m-1,\\\dfrac{h_m+1}{a_m k_m} , & \text{if } j=k_m-1, \\
 1 , & \text{otherwise. }\end{cases}$$
\np
\textbf{Case 2 ($m\in\mathcal S$).} Then
$$w_{(\mathbf r^{<m},\,i),\,(\mathbf r^{<m},\,j)}(\mathbf h) = g_{(i,j)}(h_m)(\sigma_1\sigma_m^{-1})\left(w_{\pi\left(A_1^{-1}\right)\cdot(\mathbf r^{<m},\,i-1),\,\pi\left(A_1^{-1}\right)\cdot(\mathbf r^{<m},\,j-1)}(\mathbf h) \right),$$
where $$g_{(i,j)}(h_m) = \begin{cases}-\dfrac{h_m+1}{a_m k_m} , & \text{if } i=0, \\
-\dfrac{a_mk_m}{h_m+1}, & \text{if } j=0,\\ 1 , & \text{otherwise. }\end{cases}$$
\np
Set $\ell:=\lcm(k_1,k_m)/k_m$. In either case, there exist positive integers (depending on $i,j$) 
$$
1\leq i_1 < j_1 < i_2 < j_2 < \dots < i_\ell < j_\ell \leq \lcm(k_1,k_m) ,$$
or
$$1\leq j_1 < i_1 < j_2 < i_2 < \dots < j_\ell < i_\ell \leq \lcm(k_1,k_m),
$$
such that
$$
w_{(\mathbf r^{<m},\,i),\,(\mathbf r^{<m},\,j)}({\bf{h}}) 
= \prod_{r=1}^{\ell}\frac{h_m+i_r}{\,h_m+j_r\,}\,\left(\sigma_1\sigma_m^{-1}\right)^{\lcm(k_1,k_m)}\left(w_{(\mathbf r^{<m},\,i),\,(\mathbf r^{<m},\,j)}({\bf{h}})\right).
$$
\np
By Lemma~\ref{simplifymatrixentries}, it follows that $w_{(\mathbf r^{<m},\,i),\,(\mathbf r^{<m},\,j)}({\bf{h}})  =0$. This proves the claim.

\np
Consequently, $W_1(\mathbf h)$ is diagonal, i.e., $W_1(\mathbf h)
=\diag\bigl( w_{\mathbf r,\mathbf r}(\mathbf h)\bigr)_{\mathbf r\in \mathfrak R(k_1,\dots,k_m)}$. Fix ${\bf r}\in \mathfrak R(k_1,\dots,k_m)$. From
$$
W_1({\bf{h}}) = A_i\,\sigma_j^{-1}(A_j^{-1})\,\sigma_i\,\sigma_j^{-1}(W_1({\bf{h}}))\,\sigma_j^{-1}(A_j)\,A_i^{-1},\quad \text{for all} \quad i, j \in \mathbf m,
$$
we obtain, for every $i,j\in \mathbf m$, $w_{{\bf r},{\bf r}}(\mathbf h)
=\bigl(\sigma_i\sigma_j^{-1}\bigr)^{\lcm(k_i,k_j)}
\!\bigl(w_{{\bf r},{\bf r}}(\mathbf h)\bigr)
$. Hence,
$$w_{{\bf{r}},{\bf{r}}}({\bf{h}}) = F\left( \sum_{i=1}^m h_i\right), \quad \text{for some}\quad F(X) \in \C[X].$$
\np
Since $\gcd(k_1, k_2, \dots, k_m) = 1$, then by Corollary \ref{generalizedH-S}, the group $H$ generated by
$$
\left\{ \pi(A_i)\pi(A_m^{-1}) : i \in {\bf m} \setminus \{m\} \right\}
$$
acts transitively on the set $\mathfrak{R}(k_1, \dots, k_m)$. Hence,
$$w_{{\bf{r}},{\bf{r}}}({\bf{h}}) = F\left( \sum_{i=1}^m h_i\right),\quad \text{for all}\quad {\bf{r}} \in \mathfrak{R}(k_1,\dots,k_m).$$
\np
Since $W_1({\bf{h}})\, A_i =  A_i\, \Delta^{-1}_i\left(W_4({\bf{h}})\right)$, for all $i\in {\bf{m}}$, then
$$
W_4(\mathbf h)\;=\;\diag\bigl(v_{\mathbf r,\mathbf r}(\mathbf h)\bigr)_{\mathbf r\in \mathfrak R(k_1,\dots,k_m)},
\quad
v_{\mathbf r,\mathbf r}(\mathbf h)\;=\;F\left(\sum_{j=1}^m h_j - m+1\right).
$$
\np
Therefore, if $\Phi$ is an idempotent, then $\Phi = 0$ or $\Phi = \I_{2\prod_{i=1}^mk_i}$. Henceforth, for any subset $\mathcal S \subseteq \mathbf m$ and any choice of $a_1,\dots,a_m \in \C^{\times}$, the module $E\left(\sum_{i=1}^m a_i x_i^{k_i},\,\mathcal S\right)$
is indecomposable.
\end{proof}

\begin{corollary} \label{indecfirst-explicit}
    Let $a_1,a_2,\dots,a_m \in \C^{\times}$, $k \in \Z_{\geq 1}$, and $\mathcal{S} \subseteq {\bf{m}}$. Then, the module
     $$E\!\left(a_1 x_1^{k}+\sum_{i=2}^m a_i x_i,\ \mathcal S\right)$$
    is indecomposable.
\end{corollary}

\begin{corollary} \label{End-descrip}
    Let $a_1,\dots,a_m\in\C^{\times}$, $\mathcal S\subseteq\mathbf m$, and $k_1,\dots,k_m\in\Z_{\geq1}$ with
$s:=\gcd(k_1,\dots,k_m)$.
Then $$\End_{\,\mathcal{U}(\mathfrak{sl}(m|1))}\left(E\!\left(\sum_{i=1}^m a_i x_i^{k_i},\,\mathcal S\right) \right) \;\simeq\; \C[X]^{\oplus s} .$$
\end{corollary}

\begin{proof}
With $A_i$ as in Theorem~\ref{maintheorem_indec}, and $\Phi \in \End_{\,\mathcal{U}(\mathfrak{sl}(m|1))} \left(M\bigl(A_1,\dots,A_m \bigr) \right)$, we have
$$\Phi = \left[\begin{array}{ c | c }
    W_1({\bf{h}}) & \mathbf 0 \\
    \hline
    \mathbf 0 & W_4({\bf{h}})
  \end{array}\right],\qquad
W_1(\mathbf h),\,W_4(\mathbf h)\in
\Mat_{\prod_{j=1}^m k_j}\bigl(\C[\mathbf h]\bigr),$$
where
$$
W_1({\bf h})\;=\;\diag\bigl(w_{\mathbf r}\bigr)_{\mathbf r \in \mathfrak{R}(k_1,\dots,k_{m})}
\;=\;\diag\!\left(
w_{(0,\dots,0)}({\bf h}),\,\ldots,\,
w_{(k_1-1,\dots,k_m-1)}({\bf h})
\right),
$$
$$
W_4({\bf h})\;=\;\diag\bigl(v_{\mathbf r}\bigr)_{\mathbf r \in \mathfrak{R}(k_1,\dots,k_{m})}
\;=\;\diag\!\left(
v_{(0,\dots,0)}({\bf h}),\,\ldots,\,
v_{(k_1-1,\dots,k_m-1)}({\bf h})
\right),
$$
with the diagonal entries ordered lexicographically by $\preccurlyeq_{\mathfrak{R}(k_1,\dots,k_{m})}$. Using the argument from the “if” direction of Theorem~\ref{maintheorem_indec}, the relations
$$W_1({\bf{h}})\, A_i \,=\,  A_i\, \Delta^{-1}_i\left(W_4({\bf{h}})\right),\qquad \text{for all}\;\; i\in \mathbf m.$$
imply that the diagonal entries of $W_1$ (resp. $W_4$) are the same on each $O_{p,\mathcal S}$ (resp. $\overline{O}_{p,\mathcal S}$). More precisely, for each $p\in \Z/s\Z$, $\mathbf{r}',\mathbf{r} \in O_{p,\mathcal{S}} $, and $\mathbf{r}'',\mathbf{r}''' \in \overline{O}_{p,\mathcal S}, $ there exists $F_p(X)\in\C[X]$ such that
$$w_{{\bf{r}}}({\bf{h}}) = w_{{\bf{r}'}}({\bf{h}})=F_p\left( \sum_{i=1}^m h_i\right),\qquad w_{\mathbf{r}''}({\bf{h}}) = w_{\mathbf{r}'''}({\bf{h}})=F_p\left( \sum_{i=1}^m h_i +m-1\right).$$
This proves the corollary.
\end{proof}

\begin{proposition}
Let $a_1,a_2,\dots,a_m \in \C^{\times}$, $\mathcal{S} \subseteq \mathbf{m}$, and
$k_1,k_2,\dots,k_m \in \Z_{\geq 1}$ with
$$
\gcd(k_1,k_2,\dots,k_m) \;=\; s \geq 2 .
$$ 
Then
\begin{equation} \label{eq5.16}
E\!\left(\sum_{i=1}^m a_i x_i^{k_i},\,\mathcal S\right) \;\simeq\; \bigoplus_{p\in \Z/s\Z}\, M\left(A_1,A_2,\dots,A^p_m\right),
\end{equation}
    where 
    \[A_i,\, A_m^{p} \in \Mat_{\frac{1}{s}\prod_{j=1}^m k_j}\bigl(\C[h_i]\bigr)\] 
    are the following matrices:

    \begin{enumerate}
\item[$\bullet$] For $1 \leq i \leq m-1$, $A_i =  \diag\!\Big(\underbrace{X_{i,a_i},\dots,X_{i,a_i}}_{\prod_{j=1}^{i-1} k_j\ \text{copies}}\Big)$ with 
   $$X_{i,a_i}\;=\;
\begin{cases}
U_{(h_i,a_i k_i)}\!\left(\dfrac{1}{s}(k_i-1)\!\prod_{j=i+1}^m k_j,\ \dfrac{1}{s}\!\prod_{j=i+1}^m k_j\right), & i \notin \mathcal S,\\
V_{(h_i,a_i k_i)}\!\left(\dfrac{1}{s}\!\prod_{j=i+1}^m k_j,\ \dfrac{1}{s}(k_i-1)\!\prod_{j=i+1}^m k_j\right), & i \in \mathcal S.
\end{cases} $$

\item[$\bullet$] For $0 \leq p \leq s-1$, 
the matrix $A_m^p$ is the block-diagonal matrix with blocks $\mathcal P_{\mathbf r}^p(\mathbf h)$ of size $k_m/s\times k_m/s$
indexed by the lexicographically ordered set 
$\mathfrak R(k_1,\dots,k_{m-1})$, where
$$
\mathcal P^p_{\mathbf r}(\mathbf h)
=
\begin{cases}
U_{(h_m,a_m k_m)}\!\left(\frac{k_m}{s}-1,\,1\right), 
& \text{if } m\notin\mathcal S \text{ and }\ \|\mathbf r\|_{\mathcal S}\equiv p \pmod s,\\
\I_{k_m/s},& \text{if } m\notin\mathcal S \text{ and }\ \|\mathbf r\|_{\mathcal S}\not\equiv p \pmod s,\\
V_{(h_m,a_m k_m)}\!\left(1,\,\frac{k_m}{s}-1\right),
& \text{if } m\in\mathcal S \text{ and }\ \|\mathbf r\|_{\mathcal S}\equiv p-1 \pmod s,\\
-h_m\,\I_{k_m/s}, & \text{if } m\in\mathcal S \text{ and }\ \|\mathbf r\|_{\mathcal S}\not\equiv p-1 \pmod s.
\end{cases}
$$
Here, $\|\mathbf r\|_{\mathcal S} = \sum_{i=1}^{m-1} \nu(\mathcal S, \mathbf m)_i\, r_i$. 
\end{enumerate}
\end{proposition}

\begin{proof}
From Corollary \ref{onlyif-main},
$$E\!\left(\sum_{i=1}^m a_i x_i^{k_i},\,\mathcal S\right) \;=\; \bigoplus_{p\in \Z/s\Z} M^{p}, \qquad
M^{p}=M^{p}_{\bar 0}\oplus M^{p}_{\bar 1}$$
where,
$$
M^{p}_{\bar 0}
=\bigoplus_{\mathbf r\in O_{p,\mathcal S}}
\mathbf x^{\mathbf r}\,\C[x_1^{k_1},\dots,x_m^{k_m}]\,
e^{\sum_{i=1}^{m} a_i x_i^{k_i}},
\quad
M^{p}_{\bar 1}
=\bigoplus_{\mathbf r'\in \overline{O}_{p,\mathcal S}}
\mathbf x^{\mathbf r'}\,\xi\,\C[x_1^{k_1},\dots,x_m^{k_m}]\,
e^{\sum_{i=1}^{m} a_i x_i^{k_i}}.
$$
\np
We will show that
$$M^{p} \;\simeq\; M\left(A_1,A_2,\dots,A^p_m\right),\qquad
A_i,\, A_m^{p} \in \Mat_{\frac{1}{s}\prod_{j=1}^m k_j}\bigl(\C[h_i]\bigr).$$
To determine $A_i$ where $i\in\mathbf m\setminus\{m\}$, we use Lemma~\ref{decompositionlemma}, which yields $\overline{O}_{p,\mathcal S}=\pi(A_i)^{-1}\!\cdot O_{p,\mathcal S}$. Using the total order $\preccurlyeq$ restricted to the $\mathcal U(\mathfrak{h})$-basis of the module $M^{p}$
$$\mathcal B_{M^{p}}:=\left\{ \mathbf x^{\mathbf r'}\, e^{\sum_{i=1}^m a_i x_i^{k_i}}\bigm|  \mathbf r'\in O_{p,\mathcal S}\right\}\, \sqcup \,\left\{ \mathbf x^{\mathbf r''}\,\xi\, e^{\sum_{i=1}^m a_i x_i^{k_i}}\bigm|  \mathbf r''\in \pi(A_i)^{-1}\!\cdot O_{p,\mathcal S} \right\},$$
together with the formulas \eqref{importanteq1} and \eqref{importanteq2}, the $A_i$ (with respect to $\mathcal B_{M^{p}}$) is block-diagonal with $\prod_{j=1}^{i-1} k_j$ identical square blocks, each of size $\frac{1}{s}\prod_{j=i}^{m} k_j \times \frac{1}{s}\prod_{j=i}^{m} k_j$ (each $(r_1,\dots,r_{i-1})$ indexes a block). Moreover, every block equals
$$\begin{cases}
U_{(h_i,a_i k_i)}\!\left(\dfrac{1}{s}(k_i-1)\!\prod_{j=i+1}^m k_j,\ \dfrac{1}{s}\!\prod_{j=i+1}^m k_j\right), & i \notin \mathcal S,\\
V_{(h_i,a_i k_i)}\!\left(\dfrac{1}{s}\!\prod_{j=i+1}^m k_j,\ \dfrac{1}{s}(k_i-1)\!\prod_{j=i+1}^m k_j\right), & i \in \mathcal S.
\end{cases} $$
\np
To determine the matrix $A_m^{p}$, we first use the total order $\preccurlyeq$ restricted to the $\mathcal U(\mathfrak{h})$-basis 
$$\mathcal B_{M^{p}}:=\left\{ \mathbf x^{\mathbf r'}\, e^{\sum_{i=1}^m a_i x_i^{k_i}}\bigm|  \mathbf r'\in O_{p,\mathcal S}\right\}\, \sqcup \,\left\{ \mathbf x^{\mathbf r''}\,\xi\, e^{\sum_{i=1}^m a_i x_i^{k_i}}\bigm|  \mathbf r''\in \pi(A_m)^{-1}\!\cdot O_{p,\mathcal S} \right\}.$$
With respect to this ordered basis,
$$
A_m^{p}=\diag\Bigl(\,\mathcal P^{p}_{\mathbf r}(\mathbf h)\Bigr)_{\;\mathbf r\in \mathfrak R(k_1,\dots,k_{m-1})}, \qquad \mathcal P^p_{\mathbf r}(\mathbf h) \in \Mat_{k_m/s} (\C[\mathbf h]).
$$

\np
We consider the case when $m \notin \mathcal S$, the case $m \in \mathcal S$ is analogous. 

\np Fix $\mathbf r:=(r_1,\dots,r_{m-1})$. It determines the block $\mathcal{P}^{p}_{\mathbf r}(\mathbf h)$ which is indexed by the set
\[ \left\{(r_1, \dots, r_{m-1}, i_{\mathbf r} + t s )\, | \, 0 \leq t \leq \frac{k_m}{s}-1 \,\right\} \]
where $(r_1, \dots, r_{m-1}, i_{\mathbf r})$ is the index corresponding to the leftmost entry in the top row of $\mathcal{P}^{p}_{\mathbf r}(\mathbf h)$. Using the formula \eqref{importanteq1}, we deduce 
$$
\mathcal P^{p}_{\mathbf r}(\mathbf h)
=
\begin{cases}
U_{(h_m,a_m k_m)}\!\left(\frac{k_m}{s}-1,\,1\right), 
& \text{if }\;\; i_{\mathbf r} =0,\\
\I_{k_m/s}, & \text{otherwise}.
\end{cases}
$$
\np 
The condition $i_{\mathbf r} =0$ is equivalent to $\|\mathbf r\|_{\mathcal S}\equiv p \pmod s$. Therefore, the proposition follows.
\end{proof}

\begin{remark}
By Proposition~\ref{End-descrip}, the direct summands in \eqref{eq5.16} are indecomposable and pairwise nonisomorphic.
\end{remark}

\begin{proposition} \label{isomthm1}
Let $\mathcal S_1,\mathcal S_2\subseteq{\bf m}$, let
$\mathbf a=(a_1,\dots,a_m)\in(\C^\times)^m,\,\mathbf b=(b_1,\dots,b_m)\in(\C^\times)^m$, and 
let $\mathbf k=(k_1,\dots,k_m)\in\Z_{\geq 1}^m,\,\boldsymbol{\ell}=(\ell_1,\dots,\ell_m)\in\Z_{\geq 1}^m$.
If
$$
E\Bigl(\textstyle\sum_{i=1}^m a_i x_i^{k_i},\,\mathcal S_1\Bigr) \;\simeq\; E\Bigl(\textstyle\sum_{i=1}^m b_i x_i^{\ell_i},\,\mathcal S_2\Bigr),
$$
then $\mathbf k=\boldsymbol{\ell}$.  
Moreover, for any $i \in S_1 \symdiffsmall S_2$, we have $k_i=\ell_i=2$.
\end{proposition}

\begin{proof}
 Since
$$
E\!\left(\sum_{i=1}^m a_i x_i^{k_i},\,\mathcal S_1\right)
\in
\mathcal{M}_{\mathfrak{sl}(m|1)}\!\left(\prod_{i=1}^{m} k_i \,\Big|\, \prod_{i=1}^{m} k_i\right),
$$
then
$$
E\!\left(\sum_{i=1}^m a_i x_i^{k_i},\,\mathcal S_1\right)
\simeq
E\!\left(\sum_{i=1}^m b_i x_i^{\ell_i},\,\mathcal S_2\right)
\quad \text{implies} \quad
\prod_{i=1}^{m} k_i \;=\; \prod_{i=1}^{m} \ell_i.
$$
By Proposition~\ref{explicitrealizationexp}, we have
$$E\!\left(\sum_{i=1}^m a_i x_i^{k_i},\,\mathcal S_1\right)\;\simeq\; M\left(A_1,A_2,\dots,A_m\right),$$
where for each $i\in\mathbf{m}$,
$$
A_i=
\diag\bigl(\underbrace{T_i,\dots,T_i}_{\prod_{j=1}^{i-1} k_j\ \text{copies}}\bigr),
\qquad
T_i=
\begin{cases}
U\left(i,a_i,\mathbf k_{[i,m]}\right), & i\notin\mathcal{S}_1,\\
V\left(i,a_i,\mathbf k_{[i,m]}\right), & i\in\mathcal{S}_1.
\end{cases}
$$

$$E\!\left(\sum_{i=1}^m b_i x_i^{\ell_i},\,\mathcal S_2\right) \;\simeq\; M\left(B_1,B_2,\dots,B_m\right),$$
where for each $i\in\mathbf{m}$,
$$
B_i=
\diag\!\Big(\underbrace{T'_i,\dots,T'_i}_{\prod_{j=1}^{i-1} \ell_j\ \text{copies}}\Big),
\qquad
T'_i=
\begin{cases}
U\left(i,b_i,\boldsymbol{\ell}_{[i,m]}\right), & i\notin\mathcal{S}_2,\\
V\left(i,b_i,\boldsymbol{\ell}_{[i,m]}\right), & i\in\mathcal{S}_2.
\end{cases}
$$
\np
By Proposition ~\ref{simplified-W(h)}, we have $M(A_1,\dots, A_m) \simeq M(B_1,\dots, B_m)$ if and only if there exist matrices $ W_1({\bf{h}}), W_4({\bf{h}}) \in \GL_{\prod_{i=1}^{m}k_i}(\C[{\bf{h}}]),$ such that
\begin{equation} \label{maineq}
    W_1({\bf{h}})A_i =B_i \Delta^{-1}_i(W_4({\bf{h}})),\quad \text{for all}\quad i\in {\bf{m}}.
\end{equation}
\np
Assume $k_1>\ell_1$. Then $\prod_{i=2}^m k_i < \prod_{i=2}^m \ell_i$. We consider four cases.

\noindent
 \textbf{Case 1:} Assume $1\in\mathcal S_1\cap\mathcal S_2$. Then $A_1=U(1,a_1,\mathbf k)$ and $B_1=U(1,b_1,\boldsymbol{\ell})$. By \eqref{maineq},
$$(-1)^{k_1+1} \left(\frac{h_1}{a_1k_1}\right)^{\prod_{i=2}^m k_i}\, \det \bigl(W_1(\mathbf h)\bigr)= (-1)^{\ell_1+1} \left(\frac{h_1}{b_1\ell_1}\right)^{\prod_{i=2}^m \ell_i}\det \bigl(W_4(\mathbf h)\bigr).$$
\np
Equating $h_1$-degrees forces $\prod_{i=2}^m k_i=\prod_{i=2}^m \ell_i$, contradicting
$\prod_{i=2}^m k_i<\prod_{i=2}^m \ell_i$.

\np
  \textbf{Case 2:} Assume $1\notin\mathcal S_1\cup\mathcal S_2$. Then $A_1=V(1,a_1,\mathbf k)$ and $B_1=V(1,b_1,\boldsymbol{\ell})$. By  \eqref{maineq},
    \begin{multline*}
        (-1)^{k_1+1} \left({a_1k_1}\right)^{\prod_{i=2}^m k_i} (-h_1)^{(k_1-1)\prod_{j=2}^mk_j}\, \det \bigl(W_1(\mathbf h)\bigr)\\
        =(-1)^{\ell_1+1} \left({b_1\ell_1}\right)^{\prod_{i=2}^m \ell_i} (-h_1)^{(\ell_1-1)\prod_{j=2}^m\ell_j}\,\det \bigl(W_4(\mathbf h)\bigr).
    \end{multline*}
\np
If $\ell_1=1$, this yields a contradiction. Hence, $\ell_1>1$. 
Comparing the $h_1$-degrees gives
$$
(k_1-1)\prod_{j=2}^m k_j \;=\; (\ell_1-1)\prod_{j=2}^m \ell_j,
$$
which, together with $\prod_{i=1}^m k_i=\prod_{i=1}^m \ell_i$, forces $k_1=\ell_1$, a contradiction.

\np
  \textbf{Case 3:} Assume $1\in\mathcal S_1$ and $1\notin\mathcal S_2$. Then $A_1=V(1,a_1,\mathbf k)$ and $B_1=U(1,b_1,\boldsymbol{\ell})$. By ~\eqref{maineq},
\begin{multline*}
        (-1)^{k_1+1} \left({a_1k_1}\right)^{\prod_{i=2}^m k_i} (-h_1)^{(k_1-1)\prod_{j=2}^mk_j}\, \det \bigl(W_1(\mathbf h)\bigr)=(-1)^{\ell_1+1} \left(\frac{h_1}{b_1\ell_1}\right)^{\prod_{i=2}^m \ell_i}\det \bigl(W_4(\mathbf h)\bigr).
\end{multline*}
\np    
Comparing the $h_1$-degrees gives $(k_1-1)\prod_{j=2}^m k_j=\prod_{j=2}^m \ell_j.
$
Since $\prod_{i=1}^m k_i=\prod_{i=1}^m \ell_i$, we get $k_1-1=\dfrac{k_1}{\ell_1}$, i.e. $(\ell_1-1)k_1=\ell_1$, forcing $\ell_1=2$ and $k_1=2$, a contradiction.

\np
 \textbf{Case 4:} Assume $1\notin\mathcal S_1$ and $1\in\mathcal S_2$. Then $A_1=U(1,a_1,\mathbf k)$ and $B_1=V(1,b_1,\boldsymbol{\ell})$. By ~\eqref{maineq},
$$(-1)^{k_1+1} \left(\frac{h_1}{a_1k_1}\right)^{\prod_{i=2}^m k_i}\, \det \bigl(W_1(\mathbf h)\bigr)= (-1)^{\ell_1+1} \left({b_1\ell_1}\right)^{\prod_{i=2}^m \ell_i} (-h_1)^{(\ell_1-1)\prod_{j=2}^m\ell_j}\,\det \bigl(W_4(\mathbf h)\bigr).$$
\np
Comparing the $h_1$-degrees gives 
$$\prod_{j=2}^m k_j \;=\; (\ell_1-1)\prod_{j=2}^m \ell_j.$$
\np
If $\ell_1=1$, this yields a contradiction. Hence $\ell_1>1$. From $\prod_{i=1}^m k_i=\prod_{i=1}^m \ell_i$ we get
$$
\ell_1-1=\frac{\prod_{j=2}^m k_j}{\prod_{j=2}^m \ell_j}
=\frac{\ell_1}{k_1},
$$
i.e. $(k_1-1)\ell_1=k_1$. The only integer solution with $\ell_1>1$ is $(k_1,\ell_1)=(2,2)$, contradicting $k_1>\ell_1$.

\np
Therefore, $k_1=\ell_1$. Repeating the same argument for each index $i=2,\dots,m$ yields
$$
k_i=\ell_i \qquad \text{for all}\quad i\in \mathbf m.
$$
\np
Assume $\mathcal S_1\neq \mathcal S_2$. Without loss of generality, pick 
 $i\in \mathcal S_1\setminus \mathcal S_2$. Then
$$A_i=\diag\!\Big(\underbrace{V\left(i,a_i,\mathbf k_{[i,m]}\right),\dots,V\left(i,a_i,\mathbf k_{[i,m]}\right)}_{\prod_{j=1}^{i-1} k_j\ \text{copies}}\Big),\; B_i =\diag\!\Big(\underbrace{U\left(i,b_i,\mathbf k_{[i,m]}\right),\dots,U\left(i,b_i,\mathbf k_{[i,m]}\right)}_{\prod_{j=1}^{i-1} k_j\ \text{copies}}\Big).$$
\np
Equation ~\eqref{maineq} yields
$$\det\Big(V\left(i,a_i,\mathbf k_{[i,m]}\right)\Big)^{\prod_{j=1}^{i-1} k_j}\, \det \bigl(W_1(\mathbf h)\bigr)= \det\Big(U\left(i,b_i,\mathbf k_{[i,m]}\right)\Big)^{\prod_{j=1}^{i-1} k_j}\,\det \bigl(W_4(\mathbf h)\bigr).$$
\np
Comparing the $h_i$-degrees gives $(k_i-1)\prod_{j=i+1}^m k_j \;=\;  \prod_{j=i+1}^m k_j,
$ which implies that $k_i =2$.    
\end{proof}
\np
Throughout the rest of the section, we denote
$$ \mathcal S_1 \,{\symdiffsmall^c}\, \mathcal S_2\;:=\;\mathbf m \setminus (\mathcal{S}_1 \,{\symdiffsmall}\, \mathcal{S}_2) \;=\; \left(\mathcal{S}_1\cap \mathcal{S}_2 \right) \;\cup\;(\mathbf{m}\setminus (\mathcal{S}_1\cup \mathcal{S}_2)).$$

\np
\begin{definition}
    Let $\mathbf k:=(k_1,\dots,k_m) \in \Z_{\geq 1}^m$. 
    Set 
$$
N(\mathbf k):=\{\,i\in\mathbf m\bigm| k_i=2\,\},\qquad G_{\mathbf k}:=\big(\Z/2\Z\big)^{\times \card(N(\mathbf k))}.
$$
\np
On the weighted projective space
$$
\mathbb P(\mathbf k):=\mathbb P(k_1,\dots,k_m)=\operatorname{Proj}\,\mathbb C[x_1,\dots,x_m],\quad \deg x_i=k_i,
$$
consider the subset
$$
U_m:=\{\,(a_1,\dots,a_m)\in\mathbb P(\mathbf k)\ :\ a_i\neq 0 \;\;\text{for all }\;\; i\in \mathbf m\,\}.
$$
\np
Define a group action of $G_{\mathbf k}$ on $U_m$, which acts nontrivially only in the coordinates $j$ with $k_j=2$. Explicitly, for $\varepsilon=(\varepsilon_i)_{i\in N(\mathbf k)}\in G_{\mathbf k}$, 
$
\varepsilon\cdot_{G_{\mathbf k}}(a_1,\dots,a_m)
 = (a'_1,\dots,a'_m)$ where 
$$a'_i = \begin{cases}
    a_i,&\text{if } k_i \neq 2;\\
    \left(-\frac{1}{4}\right)^{\varepsilon_i}a_i^{\,1-2\varepsilon_i},
    & \text{if } k_i =2.
\end{cases}$$
\end{definition}
\np
 For $\varepsilon \in G_{\mathbf k}$, define
$$\Supp(\varepsilon) := \left\{i\in N(\mathbf k) \bigm| \varepsilon_i =1 \right\}.$$

\np 
\begin{theorem} \label{main-isom}
    Let $\mathcal S_1,\mathcal S_2\subseteq{\bf m}$, $\mathbf a=(a_1,\dots,a_m)\in(\C^\times)^m,\,\mathbf b=(b_1,\dots,b_m)\in(\C^\times)^m$ 
and $\mathbf k=(k_1,\dots,k_m) \in\Z_{\geq 1}^m$. Then
$$
E\Bigl(\textstyle\sum_{i=1}^m a_i x_i^{k_i},\,\mathcal S_1\Bigr) \;\simeq\; E\Bigl(\textstyle\sum_{i=1}^m b_i x_i^{k_i},\,\mathcal S_2\Bigr),
$$
if and only if there exists $g \in G_{\mathbf k}$ such that 
$$\Supp(g) =\mathcal S_1 \symdiffsmall \mathcal S_2  \qquad\text{and}\qquad \left(g\cdot_{G_{\mathbf k}}\mathbf a\right)_{\mathcal{S}_2} =\mathbf b_{\mathcal{S}_2}\quad\text{in }\quad \mathbb P(\mathbf k). $$
\np
Here $(g\cdot_{G_{\mathbf k}}\mathbf a)_{\mathcal S_2}$ and $\mathbf b_{\mathcal S_2}$ are as defined in Theorem~\ref{classi-sl(m|1)asD(m|1)-mod}.
\end{theorem}

\begin{proof}
By Proposition~\ref{explicitrealizationexp},
 $$ E\Bigl(\textstyle\sum_{i=1}^m a_i x_i^{k_i},\,\mathcal S_1\Bigr)
    \;\simeq\; M\left(A_1,A_2,\dots,A_m\right),$$
 where 
 $$
A_i=
\diag\!\Big(\underbrace{T_{i,a_i},\dots,T_{i,a_i}}_{\prod_{j=1}^{i-1} k_j\ \text{copies}}\Big),\qquad 
T_{i,a_i}=
\begin{cases}
U\left(i,a_i,\mathbf k_{[i,m]}\right), & i\notin\mathcal{S}_1,\\
V\left(i,a_i,\mathbf k_{[i,m]}\right), & i\in\mathcal{S}_1.
\end{cases}
$$
\np
Similarly, 
$$E\Bigl(\textstyle\sum_{i=1}^m b_i x_i^{k_i},\,\mathcal S_2\Bigr)\;\simeq\; M\left(B_1,B_2,\dots,B_m\right),$$
 where
  $$
B_i=
\diag\!\Big(\underbrace{T_{i,b_i},\dots,T_{i,b_i}}_{\prod_{j=1}^{i-1} k_j\ \text{copies}}\Big),\qquad 
T_{i,b_i}=
\begin{cases}
U\left(i,b_i,\mathbf k_{[i,m]}\right), & i\notin\mathcal{S}_2,\\
V\left(i,b_i,\mathbf k_{[i,m]}\right), & i\in\mathcal{S}_2.
\end{cases}
$$
\np
From Proposition \ref{isomthm1}, it suffices to assume that $k_i=2$, for all $ i\in \mathcal{S}_1 \,{\symdiffsmall}\, \mathcal{S}_2$.
 Let 
$$\Phi:M(A_1,\dots, A_m) \;\to \; M(B_1,\dots, B_m) $$ 
be a $\mathcal{U}(\mathfrak{sl}(m|1))$-isomorphism. Then, by Proposition \ref{simplified-W(h)}, 
$$
\Phi\big( \mathbf f({\bf h})\big) \;=\;
\left[\begin{array}{ c | c }
    W_1({\bf{h}}) & \mathbf 0 \\
    \hline
    \mathbf 0 & W_4({\bf{h}})
  \end{array}\right]\,
 \mathbf f({\bf h}),
\qquad
\text{for all }\quad  \mathbf f({\bf h})\in\C[\mathbf{h}]^{\oplus 2\prod_{i=1}^m k_i},
$$
where $W_1(\mathbf{h}),\,W_4(\mathbf{h})\in \Mat_{\prod_{i=1}^m k_i}(\C[\mathbf{h}])$ satisfy
$$
W_1(\mathbf{h})\,A_i \;=\;
B_i\,\Delta_i^{-1}\!\bigl(W_4(\mathbf{h})\bigr),
\qquad
\text{for all }\quad i\in\mathbf{m}.
$$
\np
For every $i\in\mathbf m$ (including $i\in \mathcal{S}_1 \,{\symdiffsmall}\, \mathcal{S}_2$),  the permutations 
$\pi(A_i)$ and $\pi(B_i)$ are equal. Namely,
$$
\pi(A_i)=\pi(B_i)=\begin{cases}(0\;\;\;1\;\;\;\dots\;\;\;k_i-1), & \text{if } i\notin\mathcal S,\\
(0\;\;\;k_i-1\;\;\;k_i-2\;\;\;\dots\;\;\;1), & \text{if } i\in\mathcal S.
\end{cases}
$$
Applying the argument from the “if” direction of Theorem~\ref{maintheorem_indec}, we conclude that
$$
W_1({\bf h})\;=\;\diag\bigl(w_{\mathbf r}\bigr)_{\mathbf r \in \mathfrak{R}(k_1,\dots,k_{m})}
\;=\;\diag\!\left(
w_{(0,\dots,0)}({\bf h}),\,\ldots,\,
w_{(k_1-1,\dots,k_m-1)}({\bf h})
\right),
$$
where the diagonal entries are indexed using the lexicographic order $\preccurlyeq_{\mathfrak{R}(k_1,\dots,k_{m})}$. Since $\Phi$ is a $\mathcal{U}(\mathfrak{sl}(m|1))$-isomorphism, $ W_1({\bf{h}}), W_4({\bf{h}}) \in \GL_{\prod_{i=1}^{m}k_i}(\C[{\bf{h}}])$.  Consequently, 
$$w_{{\bf r}}({\bf h}) \in \C^\times,\quad \text{for all}\quad {\bf r} \in \mathfrak{R}(k_1,\dots, k_{m}).$$

\np
Let $W\;:=\;\diag\bigl(v_{\mathbf r}\bigr)_{\mathbf r \in \mathfrak{R}(k_1,\dots,k_{m})}$, where $v_{\mathbf r} \in \C[\mathbf h]$. For each $i\in\mathbf m$, the map $W \mapsto B_i^{-1}\,W \,A_i$ permutes the diagonal entries of $W$ 
according to $\pi(B_i)$ and rescales the entry $v_{\mathbf r}$ by $\beta^{(i)}_{\mathbf r}\in\C$, i.e., the ${\mathbf r}$-th diagonal
entry of $B_i^{-1}\,W \,A_i$ is $\beta^{(i)}_{\mathbf r}\,v_{\pi(B_i)\,\cdot\,\mathbf r}$,
where the scalar $\beta^{(i)}_{\mathbf r}$ is given as follows:

\begin{enumerate}
\item[$\bullet$] For $i \in \mathcal S_1{\symdiffsmall^c} \mathcal S_2$,
$$\beta^{(i)}_{\mathbf r} = \begin{cases}\dfrac{b_i}{a_i}, & \text{if } i\in \mathbf{m}\setminus (\mathcal{S}_1\cup \mathcal{S}_2),\;\;\text{and}\;\; r_i = k_i-1,  \\ \dfrac{a_i}{b_i}, & \text{if } i \in  \mathcal{S}_1 \cap \mathcal{S}_2,\;\;\text{and}\;\;  r_i = 0,\\1 , & \text{otherwise}.\end{cases}$$

\np
\item[$\bullet$]
For $i\in \mathcal{S}_1 \,{\symdiffsmall}\, \mathcal{S}_2$ ($k_i =2$ for all $i \in \mathcal{S}_1 \,{\scriptstyle\triangle}\, \mathcal{S}_2$),
$$\beta^{(i)}_{\mathbf r} = \begin{cases}\dfrac{1}{2b_i}, & \text{if } i\in \mathcal S_2 \setminus \mathcal S_1,\;\;\text{and}\;\;  r_i = 0,  \\-\dfrac{1}{2a_i}, & \text{if } i\in \mathcal S_2 \setminus \mathcal S_1,\;\;\text{and}\;\;  r_i = 1,  \\
2a_i, & \text{if } i\in \mathcal S_1 \setminus \mathcal S_2,\;\;\text{and}\;\; r_i = 0,  \\-2b_i, & \text{if } i\in \mathcal S_1 \setminus \mathcal S_2,\;\;\text{and}\;\;  r_i = 1.  \end{cases}$$
\end{enumerate}

\np
Similarly, for each $j\in\mathbf m$, the map $W \mapsto B_j\,W \,A_j^{-1}$ permutes the diagonal entries of $W$ according to $\pi(B_j)^{-1}$ and rescales the entries $v_{\mathbf r}$ by $\alpha^{(j)}_{\mathbf r} \in \C$, i.e., the ${\mathbf r}$-th diagonal entry of $B_j\,W \,A_j^{-1}$ is $\alpha^{(j)}_{\mathbf r}\,v_{\pi(B_j)^{-1}\,\cdot\,\mathbf r}$,
where the scalar $\alpha^{(j)}_{\mathbf r}$ is given as follows:

\np
\begin{enumerate}
\item[$\bullet$]For $j \in \mathcal S_1{\symdiffsmall^c} \mathcal S_2$,
$$\alpha^{(j)}_{\mathbf r} = \begin{cases}\dfrac{a_j}{b_j}, & \text{if } j\in \mathbf{m}\setminus (\mathcal{S}_1\cup \mathcal{S}_2),\;\;\text{and}\;\;  r_j = 0,  \\\dfrac{b_j}{a_j}, & \text{if } j \in  \mathcal{S}_1 \cap \mathcal{S}_2,\;\;\text{and}\;\; r_j = k_j-1,\\1 , & \text{otherwise}.\end{cases}$$

\np
\item[$\bullet$]For $j\in \mathcal{S}_1 \,{\symdiffsmall}\, \mathcal{S}_2$,
$$\alpha^{(j)}_{\mathbf r} = \begin{cases}-2a_j, & \text{if } j\in \mathcal S_2 \setminus \mathcal S_1,\;\;\text{and}\;\; r_j = 0,  \\2b_j, & \text{if } j\in \mathcal S_2 \setminus \mathcal S_1,\;\;\text{and}\;\;  r_j = 1,  \\
-\dfrac{1}{2b_j}, & \text{if } j\in \mathcal S_1 \setminus \mathcal S_2,\;\;\text{and}\;\;  r_j = 0,  \\\dfrac{1}{2a_j}, & \text{if } j\in \mathcal S_1 \setminus \mathcal S_2,\;\;\text{and}\;\;  r_j = 1.  \end{cases}$$
\end{enumerate}
\np
Then
$$
W_1(\mathbf h)=B_j B_i^{-1}\,W_1(\mathbf h)\,A_i A_j^{-1}
\qquad\text{for all}\quad i,j\in\mathbf m
$$
if and only if the following constraints hold:

\np 
\begin{enumerate}
\item[] For $i,j\in \mathcal S_1{\symdiffsmall^c}\mathcal S_2$

\begin{equation} \label{eq720}
\left(\frac{a_i}{b_i}\right)^{\frac{\nu\left(\mathcal S_1\cap\mathcal S_2,\, \mathcal S_1{\symdiffsmall^c}\mathcal S_2\right)_i\,\lcm(k_i,k_j)}{k_i}}
=
\left(\frac{a_j}{b_j}\right)^{\frac{\nu\left(\mathcal S_1\cap\mathcal S_2,\, \mathcal S_1{\symdiffsmall^c}\mathcal S_2\right)_j\,\lcm(k_i,k_j)}{k_j}} \ ;
\end{equation}

\item[] 

\item[] for  $i\in \mathcal S_1{\symdiffsmall^c}\mathcal S_2$ and  $j\in \mathcal S_1{\symdiffsmall}\,\mathcal S_2$
\begin{equation} \label{eq730}
(-4a_j b_j)^{\frac{\nu\left(\mathcal S_1\setminus\mathcal S_2,\; \mathcal S_1{\symdiffsmall}\mathcal S_2\right)_j\,\lcm(k_i,2)}{2}}
=
\left(\frac{a_i}{b_i}\right)^{\frac{\nu\left(\mathcal S_1\cap\mathcal S_2,\; \mathcal S_1{\symdiffsmall^c}\mathcal S_2\right)_i\,\lcm(k_i,2)}{k_i}} \ ;
\end{equation}

\item[] 

\item[] for $i,j\in \mathcal S_1{\symdiffsmall}\,\mathcal S_2$
\begin{equation} \label{eq740}
(4a_i b_i)^{\nu\left(\mathcal S_1\setminus\mathcal S_2,\; \mathcal S_1{\symdiffsmall}\mathcal S_2\right)_i}
=
(4a_j b_j)^{\nu\left(\mathcal S_1\setminus\mathcal S_2,\; \mathcal S_1{\symdiffsmall}\mathcal S_2\right)_j} \ .
\end{equation}
\end{enumerate}

\np
Since $k_i=2$ for all $i\in \mathcal S_1{\symdiffsmall}\,\mathcal S_2$, there exists $g\in G_{\mathbf k}$ with
$\Supp(g)=\mathcal S_1{\symdiffsmall}\,\mathcal S_2$. For $\mathbf a=(a_1,\dots,a_m)\in(\C^\times)^m$,
$$(g\cdot_{G_{\mathbf k}}\mathbf a)_i = \begin{cases}
    a_i &\text{if}\;\;\; i\notin \Supp(g),\\ -\dfrac{1}{4a_i} &\text{if}\;\;\; i\in \Supp(g).
\end{cases}$$

\np
A direct computation shows that for each $i\in\mathbf m$,
$$
\frac{\bigl((g\cdot_{G_{\mathbf k}}\mathbf a)_{\mathcal S_2}\bigr)_i}{(\mathbf b_{\mathcal S_2})_i}
=
\begin{cases}
-4\,a_i b_i, & i\in \Supp(g)\cap \mathcal S_1\;\; (i \notin \mathcal S_2),\\
\dfrac{1}{-4\,a_i b_i}, & i\in \Supp(g)\setminus \mathcal S_1\;\; (i \in \mathcal S_2),\\
\dfrac{a_i}{b_i}, & i\notin \Supp(g),\ i\in \mathcal S_1\;\; (i \in \mathcal S_2),\\
\dfrac{b_i}{a_i}, & i\notin \Supp(g),\ i\notin \mathcal S_1\;\; (i \notin \mathcal S_2).
\end{cases}
$$

\np
Consequently, the constraints \eqref{eq720},\eqref{eq730}, and \eqref{eq740} together are equivalent to the existence of $\lambda\in\C^\times$ such that
$$
\frac{\bigl((g\cdot_{G_{\mathbf k}}\mathbf a)_{\mathcal S_2}\bigr)_i}{(\mathbf b_{\mathcal S_2})_i}
=\lambda^{k_i}
\qquad \text{for all}\quad i\in \mathbf m.
$$

\np
Moreover, $W_1(\mathbf{h}),\;W_4(\mathbf{h}) \;\in\;
\GL_{\prod_{i=1}^{m} k_i}\bigl(\mathbb{C}[\mathbf{h}]\bigr)$ satisfying
$$ W_1({\bf{h}})A_i =B_i \Delta^{-1}_i(W_4({\bf{h}})),\quad \text{for all}\quad i\in {\bf{m}},$$
exist if and only if 
$$
W_1(\mathbf h)=B_j B_i^{-1}\,W_1(\mathbf h)\,A_i A_j^{-1}\qquad\text{for all } i,j\in\mathbf m.
$$
Then, by Proposition \ref{simplified-W(h)}, the theorem follows.
\end{proof}

\np 
\begin{corollary} \label{isomclass}
Fix ${\mathbf k} = (k_1, k_2, \ldots, k_m)\in\Z_{\geq 1}^m$. Then isomorphism classes of exponential modules of the form $E\left(\sum_{i=1}^m a_i x_i^{k_i},\,\mathcal S\right)$ are parametrized by $\bigcup_{S \subseteq \left(\mathbf m\setminus N(\mathbf k)\right)}\mathbb{P}({\mathbf k}).$
\end{corollary}

\begin{proof}
    Given $a_1,\dots,a_m\in\C^\times$ and $\mathcal S\subseteq\mathbf m$, there exists $g\in G_{\mathbf k}$ such that $\Supp(g)=\mathcal S \,\cap\, N(\mathbf k)$. By Theorem \ref{main-isom},
     $$
    E\!\left(\sum_{i=1}^m a_i x_i^{k_i},\,\mathcal{S}\right)\;\simeq\; E\!\left(\sum_{i=1}^m b_i x_i^{k_i},\,\mathcal{S}\,\setminus N(\mathbf k) \right),
  $$
  where $\left(g\cdot_{G_{\mathbf k}}\mathbf a\right)_{\mathcal{S}\,\setminus N(\mathbf k)} =\mathbf b_{\mathcal{S}\,\setminus N(\mathbf k)}$ in  $\mathbb P(\mathbf k)$. Hence, it suffices to consider the case $\mathcal S \subseteq \left(\mathbf m\setminus N(\mathbf k)\right)$. For $\mathcal{S}_1, \mathcal{S}_2\subseteq \left(\mathbf m\setminus N(\mathbf k)\right) $, by Theorem \ref{main-isom}, 
   $$
    E\!\left(\sum_{i=1}^m a_i x_i^{k_i},\,\mathcal{S}_1\right)\;\simeq\; E\!\left(\sum_{i=1}^m c_i x_i^{k_i},\,\mathcal{S}_2\right),
  $$
  if and only if $\mathcal{S}_1=\mathcal{S}_2=: \overline{\mathcal{S}}$ and $\mathbf a_{\overline{\mathcal S}} = \mathbf b_{\overline{\mathcal S}}$ in $\mathbb P(\mathbf k)$. This completes the proof.
\end{proof}
 
\np 
 \textbf{Acknowledgments.}
All authors were partially supported by the Natural Sciences and Engineering Research Council of Canada. In addition, C.P. and D.W.  were partially supported by the Canadian Defence Academy Research Programme.

\end{document}